\documentclass[10pt,english]{amsart}
\usepackage[T1]{fontenc}
\usepackage[latin9]{inputenc}
\usepackage{geometry}
\usepackage{color}
\usepackage{babel}
\usepackage{float}
\usepackage{amsbsy}
\usepackage{amstext}
\usepackage{amsthm}
\usepackage{amssymb}
\usepackage{cancel}
\makeindex
\usepackage{graphicx}
\usepackage[unicode=true,pdfusetitle,
 bookmarks=true,bookmarksnumbered=false,bookmarksopen=false,
 breaklinks=false,pdfborder={0 0 1},backref=false,colorlinks=false,pdfpagemode=FullScreen]
 {hyperref}

\makeatletter
\theoremstyle{plain}
\newtheorem{thm}{\protect\theoremname}
\theoremstyle{plain}
\newtheorem{cor}[thm]{\protect\corollaryname}
\theoremstyle{remark}
\newtheorem{rem}[thm]{\protect\remarkname}
\theoremstyle{plain}
\newtheorem{assumption}[thm]{\protect\assumptionname}
\theoremstyle{definition}
\newtheorem{defn}[thm]{\protect\definitionname}
\theoremstyle{plain}
\newtheorem{lem}[thm]{\protect\lemmaname}
\theoremstyle{plain}
\newtheorem{prop}[thm]{\protect\propositionname}

\@ifundefined{date}{}{\date{}}
\usepackage{tikz}
\usepackage{pgfplots}
\usetikzlibrary{matrix,arrows,decorations.pathmorphing}
\usepackage{verbatim}
\usepackage{pgf}
\usepackage{color}
\definecolor{BLACK}{RGB}{0, 0, 0}
\definecolor{green}{RGB}{0, 0, 0}
\definecolor{cyan}{RGB}{0,100, 0}
\definecolor{yellow}{RGB}{0,100,0}
\definecolor{red}{RGB}{0,100,0}
\definecolor{BLUE}{RGB}{0,100,0}

\usetikzlibrary[patterns]
\definecolor{blue}{RGB}{0,100,0}
\colorlet{linkequation}{BLACK}

\usepackage{scalerel,stackengine}
\stackMath
\newcommand\widecheck[1]{%
\savestack{\tmpbox}{\stretchto{%
  \scaleto{%
    \scalerel*[\widthof{\ensuremath{#1}}]{\kern-.6pt\bigwedge\kern-.6pt}%
    {\rule[-\textheight/2]{1ex}{\textheight}}
  }{\textheight}%
}{0.5ex}}%
\stackon[1pt]{#1}{\scalebox{-1}{\tmpbox}}%
}
\DeclareRobustCommand{\oddevenhead} {%
\scshape
\ifodd\value{page}
Naturality of the Contact Invariant %
\else
Mariano Echeverria%
\fi
}

\makeatother

\providecommand{\assumptionname}{Assumption}
\providecommand{\corollaryname}{Corollary}
\providecommand{\definitionname}{Definition}
\providecommand{\lemmaname}{Lemma}
\providecommand{\propositionname}{Proposition}
\providecommand{\remarkname}{Remark}
\providecommand{\theoremname}{Theorem}

\begin{document}
\title[\oddevenhead ]{Naturality of the Contact Invariant in Monopole Floer Homology Under
Strong Symplectic Cobordisms}
\author{\textcolor{black}{Mariano Echeverria}}
\begin{abstract}
The contact invariant is an element in the monopole Floer homology
groups of an oriented closed three manifold canonically associated
to a given contact structure. A non-vanishing contact invariant implies
that the original contact structure is tight, so understanding its
behavior under symplectic cobordisms is of interest if one wants to
further exploit this property. 

By extending the gluing argument of Mrowka and Rollin to the case
of a manifold with a cylindrical end, we will show that the contact
invariant behaves naturally under a strong symplectic cobordism. 

As quick applications of the naturality property, we give alternative
proofs for the vanishing of the contact invariant in the case of an
overtwisted contact structure, its non-vanishing in the case of strongly
fillable contact structures and its vanishing in the reduced part
of the monopole Floer homology group in the case of a planar contact
structure. We also prove that a strong filling of a contact manifold
which is an $L$-space must be negative definite.
\end{abstract}

\maketitle

\section{1. Stating the Result and Some Applications}

Monopole Floer Homology associates to a closed, oriented, connected
3-manifold $Y$ three abelian groups $\widecheck{HM}_{\bullet}(Y)$
, $\widehat{HM}_{\bullet}(Y)$, $\overline{HM}_{\bullet}(Y)$, pronounced
$HM$-to, $HM$-from and $HM$-bar respectively. They admit a direct
sum decomposition over spin-c structures of $Y$, in the sense that
\begin{align*}
\widecheck{HM}_{\bullet}(Y)=\bigoplus_{\mathfrak{s}}\widecheck{HM}_{\bullet}(Y,\mathfrak{s})\\
\widehat{HM}_{\bullet}(Y)=\bigoplus_{\mathfrak{s}}\widehat{HM}_{\bullet}(Y,\mathfrak{s})\\
\overline{HM}_{\bullet}(Y)=\bigoplus_{\mathfrak{s}}\overline{HM}_{\bullet}(Y,\mathfrak{s})
\end{align*}

In fact, the previous decomposition is finite \cite[Proposition 3.1.1]{MR2388043}.
The chain complexes whose homology are the previous groups are built
using solutions of a perturbed version of the three dimensional Seiberg-Witten
equations, which are at the same time critical points of a perturbed
Chern-Simons-Dirac functional \cite[Section 4]{MR2388043}. There
are three different types of solutions (the boundary stable, boundary
unstable and irreducible solutions) and each group uses two of the
three types in their corresponding construction. 

Now suppose that $Y$ is equipped with a co-orientable contact structure
$\xi$ compatible with the orientation of the manifold. In practice
this means that there exists a globally defined one form $\theta$
on $Y$ for which $\xi=\ker\theta$ and $\theta\wedge d\theta$ is
positive everywhere \cite[Lemma 1.1.1]{MR2397738}. As we will review
momentarily, $\xi$ determines a spin-c structure $\mathfrak{s}_{\xi}$
and one can exploit the additional structure provided by $\xi$ in
order to construct an element $\mathbf{c}(\xi)\in\ensuremath{\widecheck{HM}_{\bullet}(-Y,\mathfrak{s}_{\xi})}$
known as the contact invariant of $(Y,\xi)$.

It is important to observe that $\mathbf{c}(\xi)$ belongs to the
monopole Floer homology groups of the manifold $-Y$, that is, $Y$
with the opposite orientation. This is because the contact invariant
$\mathbf{c}(\xi)$ should actually be regarded as a cohomology element
$\mathbf{c}(\xi)\in\widehat{HM}^{\bullet}(Y,\mathfrak{s}_{\xi})$,
and there is a natural isomorphism between $\widehat{HM}^{\bullet}(Y,\mathfrak{s}_{\xi})$
and $\widecheck{HM}_{\bullet}(-Y,\mathfrak{s}_{\xi})$ \cite[Section 22.5]{MR2388043}.
However, we will work with the homology version of the contact invariant
since most of the formulas in \cite{MR2388043} are given explicitly
for the homology groups.

Monopole Floer homology also has TFQT-like features, which concretely
means that given a cobordism $W:Y\rightarrow Y'$ between two three
manifolds, there are group homomorphisms between the corresponding
homology groups 
\begin{align*}
\widecheck{HM}_{\bullet}(W,\mathfrak{s}_{W}):\widecheck{HM}_{\bullet}(Y,\mathfrak{s}_{Y})\rightarrow\widecheck{HM}_{\bullet}(Y',\mathfrak{s}_{Y'})\\
\widehat{HM}_{\bullet}(W,\mathfrak{s}_{W}):\widehat{HM}_{\bullet}(Y,\mathfrak{s}_{Y})\rightarrow\widehat{HM}_{\bullet}(Y',\mathfrak{s}_{Y'})\\
\overline{HM}_{\bullet}(W,\mathfrak{s}_{W}):\overline{HM}_{\bullet}(Y,\mathfrak{s}_{Y})\rightarrow\overline{HM}_{\bullet}(Y',\mathfrak{s}_{Y'})
\end{align*}
Here $\mathfrak{s}_{W}$ denotes a spin-c structure which restricts
in an appropriate sense to the given spin-c structures on $Y$ and
$Y'$. Just as in the contact case, if $(W,\omega):(Y,\xi)\rightarrow(Y',\xi')$
is equipped with a symplectic form $\omega$, it determines a spin-c
structure $\mathfrak{s}_{\omega}$ , and so it makes sense to ask
the \textit{naturality} \textit{question}, that is, whether or not
\begin{equation}
\widecheck{HM}_{\bullet}(W^{\dagger},\mathfrak{s}_{\omega})\mathbf{c}(\xi')\overset{?}{=}\mathbf{c}(\xi)\label{naturality question}
\end{equation}
where $W^{\dagger}:-Y'\rightarrow-Y$ denotes the cobordism turned
``upside-down''. The main result of this work is that the answer
to the previous question is positive in the case of a strong symplectic
cobordism: 
\begin{thm}
\label{thm: Naturality}Let $(W,\omega):(Y,\xi)\rightarrow(Y',\xi')$
be a strong symplectic cobordism between two contact manifolds $(Y,\xi)$
and $(Y',\xi')$. Then 
\[
\widecheck{HM}_{\bullet}(W^{\dagger},\mathfrak{s}_{\omega})\mathbf{c}(\xi';\mathbb{F}_{2})=\mathbf{c}(\xi;\mathbb{F}_{2})
\]
\end{thm}

Here we have added an $\mathbb{F}_{2}$ to our notation of the contact
invariant to emphasize that we are using the coefficient field $\mathbb{F}=\mathbb{Z}/2$
so that we can ignore orientation issues for the moduli spaces. However,
throughout the paper we will drop the $\mathbb{F}_{2}$ from our notation
for simplicity. Clearly one can also ask whether or not one there
is an analogous statement in the case of integer coefficients. Unfortunately,
Theorem $H$ in \cite{Lin-Ruberman-Saveliev[2018]} shows that there
is no canonical choice of sign in the definition of the contact invariant,
so the best naturality statement one could hope for in this case is
one given up to a sign.

At this point it is important to specify that our notion of a \textbf{strong
symplectic cobordism} is that of a symplectic cobordism for which
the symplectic form is given in collar neighborhoods of the concave
\textit{and} convex boundaries by symplectizations of the corresponding
contact structures.

To give some context we should point out that this theorem appears
stated as Theorem 2.4 in \cite{MR2928982}, though the reference given
is a paper by Mrowka and Rollin in preparation that was never published.
Also, as will be discussed later in this paper the ``special'' condition
imposed on the cobordism in \cite{MR2928982} and \cite{MR2199446}
can be removed.

One can also ask what is known in the twin versions of monopole Floer
homology, namely, embedded contact homology and Heegaard Floer homology.
It is not by any means obvious that the corresponding homology groups
from Heegaard Floer and ECH are isomorphic to the ones coming from
monopole Floer homology and the proof can be found in \cite{Kutluhan-Lee-Taubes[2010b],Kutluhan-Lee-Taubes[2010c],Kutluhan-Lee-Taubes[2011a],Kutluhan-Lee-Taubes[2011d],Kutluhan-Lee-Taubes[2012e],Colin-Ghiggini-Honda[2012a],Colin-Ghiggini-Honda[2012b],Colin-Ghiggini-Honda[2012c],MR2746723,MR2746724,MR2746725,MR2746726,MR2746727}.
Also, the corresponding contact invariants in each version are isomorphic
to each other. 

In Heegaard Floer homology naturality holds (for example) if $(Y',\xi')$
is obtained from $(Y,\xi)$ by Legendrian surgery along a Legendrian
knot $L$ \cite[Theorem 2.3]{MR2059189}. This is an interesting case
because a 1-handle surgery, or a 2-handle surgery along a Legendrian
knot $K$ with framing $-1$ relative to the canonical framing gives
rise to a strong symplectic cobordism. On the ECH side the contact
invariant is known to be well behaved with respect to weakly exact
symplectic cobordisms \cite[Remark 1.11]{MR3190296}. Moreover, Michael
Hutchings has communicated to the author that he can improve this
result to the case of a strong symplectic cobordism, with the additional
advantage that the contact manifolds can be disconnected \cite{Hutchings[FieldTheory]}.

As we will see, the contact invariant with mod-2 coefficients is a
useful tool for understanding contact structures and our naturality
result is good enough to find properties of this invariant, though
the properties we discuss in this work were previously proven by other
means. Before we mention these applications, however, we will give
some brief history that puts into perspective the construction of
the contact invariant and why the following results were natural things
to look for.

In \cite{MR1474156} Kronheimer and Mrowka used the contact structure
of $Y$ to extend the definition of the Seiberg-Witten invariants
to the case of a compact oriented four manifold $X$ bounding it.

More precisely, one considers the non-compact four manifold $X^{+}=X\cup_{Y}([1,\infty)\times Y)$,
where $[1,\infty)\times Y$ is given the structure of an almost Kähler
cone using a symplectization $\omega$ of a contact form $\theta$
defining $\xi$. In particular, the symplectic form induced by $\theta$
determines a canonical spin-c structure $\mathfrak{s}_{\omega}$ on
$[1,\infty)\times Y$ , which we can think of as a complex vector
bundle $S=S^{+}\oplus S^{-}$ together with a Clifford multiplication
$\rho:T^{*}\left([1,\infty)\times Y\right)\rightarrow\hom_{\mathbb{C}}(S,S)$
satisfying certain conditions. 

The canonical spin-c structure $\mathfrak{s}_{\omega}$ carries a
canonical section $\varPhi_{0}$ of $S^{+}$ together with a canonical
spin-c connection $A_{0}$ on the spinor bundle. Kronheimer and Mrowka
then study solutions of the Seiberg Witten equations on $X^{+}$ which
are asymptotic to $(A_{0},\varPhi_{0})$ on the conical end. These
solutions end up having uniform exponential decay with respect to
the canonical configuration $(A_{0},\varPhi_{0})$ (Proposition 3.15
in \cite{MR1474156} or Propositions 5.7 and 5.10 in \cite{Zhang[2016]}
for a similar situation), which means that the Seiberg Witten equations
on $X^{+}$ behave very similar to how they would if the manifold
were compact, more specifically, the moduli spaces of gauge equivalence
classes of such solutions are compact. This allows as in the closed
manifold case to define a map
\[
SW_{(X,\xi)}:\text{Spin}^{c}(X,\xi)\rightarrow\mathbb{Z}
\]
 where $\text{Spin}^{c}(X,\xi)$ denotes the set of isomorphism classes
of relative spin-c structures on $X$ that restrict to the spin-c
structure $\mathfrak{s}_{\xi}$ on $Y$ determined by the contact
structure $\xi$. This map can be used to detect properties of contact
structures on three manifolds. For example, Theorem 1.3 in \cite{MR1474156}
shows that for any closed three manifold $Y$ there are only finitely
many homotopy classes of 2-plane fields which are realized as semi-fillable
contact structures. In section 1.3 of the same paper Kronheimer and
Mrowka mention as well that if $(X,\xi)$ is a 4-manifold with an
overtwisted contact structure on its boundary, then $SW_{(X,\xi)}$
vanishes identically. 

The latter result is Corollary $B$ in a different paper \cite{MR2199446}
by Mrowka and Rollin, where they analyzed how the map $SW_{(X,\xi)}$
behaves under a symplectic cobordism $(W,\omega):(Y,\xi)\rightarrow(Y',\xi')$
which they called a \textbf{special symplectic cobordism} \cite[Page 4]{MR2199446}.
Theorem D in \cite{MR2199446} shows that
\begin{equation}
SW_{(X,\xi)}=\pm SW_{(X\cup W,\xi')}\circ\jmath\label{Mrowka Rollin nat}
\end{equation}
where $\jmath:\text{Spin}^{c}(X,\xi)\rightarrow\text{Spin}^{c}(X\cup W,\xi')$
is a canonical map that extends the spin-c structure of $X$ across
the cobordism $W$. With respect to $\mathbb{Z}/2\mathbb{Z}$ coefficients,
the previous theorem can be interpreted as saying that the mod 2 Seiberg-Witten
invariants are the same. 

In order to detect more properties of the contact structure, we need
to use the machinery of Monopole Floer Homology, whose canonical reference
is \cite{MR2388043}. 

As first defined in section 6.3 of \cite{MR2299739} , one constructs
the contact invariant $\mathbf{c}(\xi)\in\ensuremath{\widecheck{HM}_{\bullet}(-Y,\mathfrak{s}_{\xi})}$
by studying the Seiberg Witten equations on $(\mathbb{R}^{+}\times-Y)\cup([1,\infty)\times Y)$
which are asymptotic on the symplectic cone to the canonical configuration
$(A_{0},\varPhi_{0})$ mentioned before and asymptotic on the half-cylinder
to a solution of the (perturbed) three dimensional Seiberg-Witten
equations. We will give more details about this construction in the
next section. However, it should be clear that based on the analogy
with the numerical Seiberg-Witten invariants $SW_{(X,\xi)}$, one
would expect the naturality property (our main theorem \ref{thm: Naturality})
as well as the vanishing of the contact invariant for an overtwisted
structure. It is the latter which we now indicate how to prove.
\begin{cor}
Let $(Y,\xi)$ be an overtwisted contact 3 manifold. Then the contact
invariant of $\xi$ vanishes, that is, $\mathbf{c}(\xi)=0$. 
\end{cor}

\begin{proof}
First we show that the 3-sphere $S^{3}$ admits an overtwisted structure
$\xi_{ot}$ for which $\mathbf{c}(\xi_{ot})=0$. For this we will
use Eliashberg's theorem \cite[Theorem 1.6.1]{MR1022310} on the existence
of an overtwisted contact structure in every homotopy class of oriented
plane field, and the fact that the Floer groups of any three manifold
$Y$ are graded by the set of homotopy classes of oriented plane fields
\cite[Section 3.1]{MR2388043}. 

Thanks to Proposition 3.3.1 in \cite{MR2388043}, which describes
the Floer homology groups of $S^{3}$, we can find a homotopy class
of plane field $[\xi]$ for which $\widecheck{HM}_{[\xi]}(S^{3})=0$.
Notice that in this case we are not specifying the spin-c structure
because $S^{3}$ has only one up to isomorphism. By Eliashberg's theorem
we can choose an overtwisted structure $\xi_{ot}$ in the homotopy
class $[\xi]$. Now, $\mathbf{c}(\xi_{ot})$ is supported in $\widecheck{HM}_{[\xi]}(-S^{3})\simeq\widecheck{HM}_{[\xi]}(S^{3})=0$
and so it will automatically vanish, i.e, $\mathbf{c}(\xi_{ot})=0$. 

Now, if $(Y,\xi)$ is an arbitrary overtwisted contact 3 manifold,
using Theorem 1.2 in \cite{MR1906906}, we can find a Stein cobordism
$(W,\omega):(Y,\xi)\rightarrow(S^{3},\xi_{ot})$. Such cobordisms
are in fact strong cobordisms so we can conclude that 
\[
\mathbf{c}(\xi)=\widecheck{HM}_{\bullet}(W^{\dagger},\mathfrak{s}_{\omega})\mathbf{c}(\xi_{ot})=\widecheck{HM}_{\bullet}(W^{\dagger},\mathfrak{s}_{\omega})(0)=0
\]
and therefore $\mathbf{c}(\xi)$ vanishes.
\end{proof}
\begin{rem}
For a proof that does not use the naturality property see Theorem
4.2 in \cite{MR2496048}. The vanishing of the contact invariant for
overtwisted contact structures is also known on the Heegaard Floer
side \cite[Theorem 1.4]{MR2153455}. For a proof on the ECH side see
Michael Hutchings' blog \cite{HutchingsBlog}. In fact, in the case
of ECH one can show that the contact invariant vanishes in the case
of planar torsion \cite{MR3102479}. The same is also true in the
Monopole Floer Homology side thanks to our naturality result and Theorem
1 in \cite{MR3128981}.
\end{rem}

\begin{cor}
Let $(X,\omega)$ be a strong filling of $(Y,\xi)$. Then the contact
invariant of $\xi$ is non-vanishing, that is, $\mathbf{c}(\xi)\neq0$.
\end{cor}

\begin{proof}
By Darboux's theorem we can remove a standard small ball $B$ of $X$
to obtain a strong cobordism $(W,\omega):(S^{3},\xi_{tight})\rightarrow(Y,\xi)$.
Naturality says that $\mathbf{c}(\xi_{tight})=\widecheck{HM}_{\bullet}(W^{\dagger},\mathfrak{s}_{\omega})\mathbf{c}(\xi)$
but the left hand side is non-vanishing and so we conclude that $\mathbf{c}(\xi)$
is non-vanishing as well.
\end{proof}
\begin{rem}
The Heegaard Floer version of this fact appears as Theorem 2.13 in
\cite{MR2206641}. That same paper contains an example of a weak filling
where the contact invariant vanishes.
\end{rem}

To explain the next corollary we do a quick review of some of the
properties of the monopole Floer homology groups. Formally they behave
like the ordinary homology groups $H_{*}(Z)$, $H_{*}(Z,A)$ and $H_{*}(A)$
for a pair of spaces in that they are related by a long exact sequence
\cite[Section 3.1]{MR2388043}
\begin{equation}
\cdots\overset{i_{*}}{\longrightarrow}\widecheck{HM}_{\bullet}(Y,\mathfrak{s})\overset{j_{*}}{\longrightarrow}\widehat{HM}_{\bullet}(Y,\mathfrak{s})\overset{p_{*}}{\longrightarrow}\overline{HM}_{\bullet}(Y,\mathfrak{s})\overset{i_{*}}{\longrightarrow}\widecheck{HM}_{\bullet}(Y,\mathfrak{s})\overset{j_{*}}{\longrightarrow}\cdots\label{eq:LONG EXACT}
\end{equation}
An important subgroup of $\widehat{HM}_{\bullet}(Y,\mathfrak{s})$
is the image of $j_{*}:\widecheck{HM}_{\bullet}(Y,\mathfrak{s})\rightarrow\widehat{HM}_{\bullet}(Y,\mathfrak{s})$
which is known as the \textbf{reduced Floer homology group }$HM_{\bullet}(Y,\mathfrak{s})$,
and in general it is of great interest to determine whether or not
a particular element belongs to it. For example, if $j_{*}=0$ we
say that $Y$ is an \textbf{$L$-space} in analogy with the terminology
from Heegaard Floer \cite[Section 42.6]{MR2388043}. To relate this
question to the naturality of the contact invariant, we need to use
the fact that for a cobordism $(W^{\dagger},\mathfrak{s}_{W}):(-Y',\mathfrak{s}_{Y})\rightarrow(-Y,\mathfrak{s}_{Y'})$
there is a commutative diagram 
\begin{equation}
\begin{array}{ccccccccc}
\cdots & \widecheck{HM}_{\bullet}(-Y',\mathfrak{s}_{Y'}) & \overset{j_{*}}{\longrightarrow} & \widehat{HM}_{\bullet}(-Y',\mathfrak{s}_{Y'}) & \overset{p_{*}}{\longrightarrow} & \overline{HM}_{\bullet}(-Y',\mathfrak{s}_{Y'}) & \overset{i_{*}}{\longrightarrow} & \widecheck{HM}_{\bullet}(-Y',\mathfrak{s}_{Y'}) & \cdots\\
\\
 & \Bigl\downarrow_{\check{HM}_{\bullet}(W^{\dagger},\mathfrak{s}_{W})} &  & \Bigl\downarrow_{\widehat{HM}_{\bullet}(W^{\dagger},\mathfrak{s}_{W})} &  & \Bigl\downarrow_{\overline{HM}_{\bullet}(W^{\dagger},\mathfrak{s}_{W})} &  & \Bigl\downarrow_{\check{HM}_{\bullet}(W^{\dagger},\mathfrak{s}_{W})}\\
\\
\cdots & \widecheck{HM}_{\bullet}(-Y,\mathfrak{s}_{Y}) & \overset{j_{*}}{\longrightarrow} & \widehat{HM}_{\bullet}(-Y,\mathfrak{s}_{Y}) & \overset{p_{*}}{\longrightarrow} & \overline{HM}_{\bullet}(-Y,\mathfrak{s}_{Y}) & \overset{i_{*}}{\longrightarrow} & \widecheck{HM}_{\bullet}(-Y,\mathfrak{s}_{Y}) & \cdots
\end{array}\label{commutative}
\end{equation}

\begin{cor}
\label{cor:Gompf} Let $(X,\omega)$ be a strong filling of $(Y',\xi')$.
Assume in addition that $Y'$ is an $L$-space. Then $X$ must be
negative definite. 
\end{cor}

\begin{proof}
Suppose by contradiction that $b^{+}(X)\geq1$. Remove a Darboux ball
as before to obtain a cobordism $(W,\omega):(S^{3},\xi_{tight})\rightarrow(Y',\xi')$.
By proposition 3.5.2 in \cite{MR2388043} we have that $\overline{HM}_{\bullet}(W^{\dagger},\mathfrak{s}_{\omega})=0$.
By the commutative diagram and the fact that $j_{*}$ vanishes for
$Y'$ we have that $\mathbf{c}(\xi')\in\ker j_{*}=\text{im}i_{*}$.
Hence $\mathbf{c}(\xi')=i_{*}\left(\left[\varPsi'\right]\right)$
for some $[\varPsi']\in\overline{HM}_{\bullet}(-Y',\mathfrak{s}_{\xi'})$
and the commutativity together with the naturality says that
\[
0=i_{*}\overline{HM}(W^{\dagger},\mathfrak{s}_{\omega})\left([\varPsi']\right)=\widecheck{HM}_{\bullet}(W^{\dagger},\mathfrak{s}_{\omega})\mathbf{c}(\xi')=\mathbf{c}(\xi_{tight})
\]
which is a contradiction.
\end{proof}
\begin{rem}
Corollary \ref{cor:Gompf} appears as Theorem 1.4 in \cite{MR2023281}.
\end{rem}

\begin{cor}
Suppose that $(Y,\xi)$ is a planar contact manifold. Then $j_{*}\mathbf{c}(\xi)=0$
and in particular any strong filling of a planar contact manifold
must be negative definite.
\end{cor}

\begin{proof}
Observe that the last statement is exactly the proof of the previous
corollary, which only used the fact that $\mathbf{c}(\xi)\in\ker j_{*}$.
If $(Y,\xi)$ is a planar contact manifold Theorem 4 in \cite{MR3128981}
(and the remarks after it) shows that there is a strong symplectic
cobordism $(W,\omega):(Y,\xi)\rightarrow(S^{3},\xi_{tight})$. The
result follows using the commutative diagram \ref{commutative} and
the fact that $j_{*}$ vanishes on $S^{3}$ because it admits a metric
of positive scalar curvature \cite[Proposition 36.1.3]{MR2388043}.
\end{proof}
\begin{rem}
Theorem 1.2 in \cite{MR2200085} shows that if the contact structure
$\xi$ on $Y$ is compatible with a planar open book decomposition
then its contact invariant vanishes when regarded as an element of
the quotient group $HF_{red}(-Y,\mathfrak{s}_{\xi})$. The second
part of our corollary should be compared with Theorem 1.2 in \cite{MR2126827},
where it is shown (among other things) that any symplectic filling
of a planar contact manifold is negative definite.
\end{rem}

The proof of the previous corollary can be extended to the case when
$Y'$ admits a metric with positive scalar curvature. First of all,
it should be pointed out that this class of manifolds is not very
large. Thanks to results of Schoen and Yau an orientable 3-manifold
with positive scalar curvature can always be obtained from a manifold
with positive scalar curvature with $b_{1}=0$ by making a connected
sum of a number of copies of $S^{1}\times S^{2}$.
\begin{cor}
Suppose that $(W,\omega):(Y,\xi)\rightarrow(Y',\xi')$ is a strong
symplectic cobordism with $Y'$ (hence $-Y'$) admitting a metric
with positive scalar curvature. Then 

a) If $c_{1}(\mathfrak{s}_{\xi'})$ is not torsion, then the contact
invariant $\mathbf{c}(\xi')$ vanishes automatically and by naturality
so will the contact invariant $\mathbf{c}(\xi)$. 

b) If $c_{1}(\mathfrak{s}_{\xi'})$ is torsion, then $j_{*}\mathbf{c}(\xi')=0$
and so by naturality $j_{*}\mathbf{c}(\xi)=0$. In particular, if
there exists a strong cobordism $(W,\omega):(Y,\xi)\rightarrow(S^{3},\xi_{tight})$
we must have that $j_{*}\mathbf{c}(\xi)=0$. 
\end{cor}

\begin{proof}
Proposition 36.1.3 in \cite{MR2388043} shows that $j_{*}$ vanishes
when $c_{1}(\mathfrak{s})$ is torsion and that the Floer groups are
zero when $c_{1}(\mathfrak{s})$ is not torsion, from which the corollary
follows immediately. 
\end{proof}
In the next section we will sketch the main argument in the proof
of Theorem (\ref{thm: Naturality}). It is our hope that this summary
captures the essential ideas of the proof of our main theorem, since
the remaining (and more technical) part of the paper will follow in
large part the paper \cite{MR2199446}, which is ``required reading''
for someone interested in understanding why the naturality theorem
will be true.

\subsubsection*{\textbf{Acknowledgments}}

I would like to thank my advisor Thomas Mark for suggesting this problem
to me and for all of his indispensable help and support. Also, I would
like to thank Tomasz Mrowka for advising me with the gluing argument
and other technical issues, Jianfeng Lin for many useful conversations
and Boyu Zhang for explaining to me several aspects of his paper \cite{Zhang[2016]}
and other discussions key for this paper.

Finally, I would like to thank the referee for many helpful comments
and corrections.

\section{2. Summary of the Proof}

As stated before, we now give a brief summary of the main ideas involved
in the proof of Theorem \ref{thm: Naturality}. In a nutshell, to
show that $\mathbf{c}(\xi)$ equals $\widecheck{HM}_{\bullet}(W^{\dagger},\mathfrak{s}_{\omega})\mathbf{c}(\xi')$,
we will define an intermediate ``hybrid'' invariant $\mathbf{c}(\xi',Y)\in\widecheck{HM}_{\bullet}(-Y,\mathfrak{s}_{\xi})$
which will work as bridge between $\mathbf{c}(\xi)$ and $\widecheck{HM}_{\bullet}(W^{\dagger},\mathfrak{s}_{\omega})\mathbf{c}(\xi')$.
Namely, using a ``stretching the neck'' argument we will show that
\[
\widecheck{HM}_{\bullet}(W^{\dagger},\mathfrak{s}_{\omega})\mathbf{c}(\xi')=\mathbf{c}(\xi',Y)
\]
while adapting the strategy of \cite{MR2199446} (which as we will
explain momentarily involves a ``dilating the cone'' argument) we
will show that 
\[
\mathbf{c}(\xi',Y)=\mathbf{c}(\xi)
\]
giving us the desired naturality result. 

First we review the definition of the contact invariant, following
section 6.2 in \cite{MR2299739} (in their paper the contact invariant
was denoted $[\check{\psi}_{Y,\xi}]$ but we have decided to switch
to the more standard notation used in Heegaard Floer homology). As
mentioned in the introduction, given a contact manifold $(Y,\xi)$
we construct the manifold
\[
Z_{Y,\xi}^{+}=\left(\mathbb{R}^{+}\times(-Y)\right)\cup\left([1,\infty)\times Y\right)
\]
and study the Seiberg-Witten equations which are asymptotic to the
canonical solution $(A_{0},\varPhi_{0})$ on the conical end $[1,\infty)\times Y$
and to a critical point $\mathfrak{c}$ of the three dimensional Seiberg
Witten equations on the cylindrical end $\mathbb{R}^{+}\times(-Y)$.
To write the Seiberg Witten equations a choice of spin-c structure
needs to be made, and in this case the contact structure $\xi$ determines
a canonical spin-c structure $\mathfrak{s}$ on $Z_{Y,\xi}^{+}$ which
we will describe later.

There is a gauge group action on such solutions and we define the
moduli space $\mathcal{M}(Z_{Y,\xi}^{+},\mathfrak{s},[\mathfrak{c}])$
as the gauge equivalence classes of the solutions to the Seiberg-Witten
equations on $Z_{Y,\xi}^{+}$. As a matter of notation, $[\cdot]$
will represent the gauge-equivalence class of a configuration so $[\mathfrak{c}]$
in this case denotes the gauge equivalence class of the critical point
$\mathfrak{c}$. The moduli space $\mathcal{M}(Z_{Y,\xi}^{+},\mathfrak{s},[\mathfrak{c}])$
is not equidimensional, in fact, it admits a partition into components
of different topological type
\[
\mathcal{M}(Z_{Y,\xi}^{+},\mathfrak{s},[\mathfrak{c}])=\bigcup_{z}\mathcal{M}_{z}(Z_{Y,\xi}^{+},\mathfrak{s},[\mathfrak{c}])
\]
where $z$ indexes the different connected components of $\mathcal{M}(Z_{Y,\xi}^{+},\mathfrak{s},[\mathfrak{c}])$.
We count points in the zero dimensional moduli spaces (which will
be compact, hence finite) and define 
\[
m_{z}(Z_{Y,\xi}^{+},\mathfrak{s},[\mathfrak{c}])=\begin{cases}
|\mathcal{M}_{z}(Z_{Y,\xi}^{+},\mathfrak{s},[\mathfrak{c}])|\;\;\;\mod2 & \text{if }\dim\mathcal{M}_{z}(Z_{Y,\xi}^{+},\mathfrak{s},[\mathfrak{c}])=0\\
0 & \text{otherwise}
\end{cases}
\]
The contact invariant is then defined at the chain level as 
\begin{equation}
c(\xi)=(c^{o}(\xi),c^{s}(\xi))\in\check{C}_{*}(-Y,\mathfrak{s}_{\xi})=\mathfrak{C}^{o}(-Y,\mathfrak{s}_{\xi})\oplus\mathfrak{C}^{s}(-Y,\mathfrak{s}_{\xi})\label{chain level def}
\end{equation}
by 
\begin{align*}
c^{o}(\xi)=\sum_{[\mathfrak{a}]\in\mathfrak{C}^{o}(-Y,\mathfrak{s}_{\xi})}\sum_{z}m_{z}(Z_{Y,\xi}^{+},\mathfrak{s},[\mathfrak{a}])e_{[\mathfrak{a}]}\\
c^{s}(\xi)=\sum_{[\mathfrak{a}]\in\mathfrak{C}^{s}(-Y,\mathfrak{s}_{\xi})}\sum_{z}m_{z}(Z_{Y,\xi}^{+},\mathfrak{s},[\mathfrak{a}])e_{[\mathfrak{a}]}
\end{align*}
In the above notation $\check{C}_{*}(-Y,\mathfrak{s}_{\xi})$ is the
free abelian group generated by the irreducible critical points $[\mathfrak{a}]\in\mathfrak{C}^{o}(-Y,\mathfrak{s}_{\xi})$
and the boundary stable critical points $[\mathfrak{a}]\in\mathfrak{C}^{s}(-Y,\mathfrak{s}_{\xi})$.
Also, $e_{[\mathfrak{a}]}$ is a bookkeeping device for each critical
point considered as a generator in the group. Lemma 6.6 in \cite{MR2299739}
then shows that $c(\xi)$ is a cycle, that is, it defines an element
$\mathbf{c}(\xi)$ of the Monopole Floer Homology group $\widecheck{HM}_{\bullet}(-Y,\mathfrak{s}_{\xi})$. 

Returning to the naturality question, suppose we have a symplectic
cobordism $(W,\omega):(Y,\xi)\rightarrow(Y',\xi')$ and we want to
decide whether or not $\widecheck{HM}_{\bullet}(W^{\dagger},\mathfrak{s}_{\omega})\mathbf{c}(\xi')=\mathbf{c}(\xi)$.
Clearly this is equivalent to showing that at the \textit{chain} level
\[
\check{m}c(\xi')-c(\xi)\in\text{im}\check{\partial}_{-Y}
\]
where $\check{\partial}_{-Y}:\check{C}_{*}(-Y,\mathfrak{s}_{\xi})\rightarrow\check{C}_{*}(-Y,\mathfrak{s}_{\xi})$
is the differential that generates $\widecheck{HM}_{\bullet}(-Y,\mathfrak{s}_{\xi})$.
Here $\check{m}$ is the chain map \cite[Definition 25.3.3]{MR2388043}
\[
\check{m}=\left(\begin{array}{cc}
m_{o}^{o} & -m_{o}^{u}\bar{\partial}_{u}^{s}-\partial_{o}^{u}\bar{m}_{u}^{s}\\
m_{s}^{o} & \bar{m}_{s}^{s}-m_{s}^{u}\bar{\partial}_{u}^{s}-\partial_{s}^{u}\bar{m}_{u}^{s}
\end{array}\right):\check{C}_{\bullet}(-Y',\mathfrak{s}_{\xi'})\rightarrow\check{C}_{\bullet}(-Y,\mathfrak{s}_{\xi})
\]
To see what $\check{m}$ does, we will explain the meaning of $m_{s}^{o}$
and $\bar{\partial}_{u}^{s}$, since the action of the remaining terms
can be inferred easily from these two examples. The map $m_{s}^{o}$
counts solutions on $W^{\dagger}:-Y'\rightarrow-Y$ with a half-cylinder
attached on each end:
\[
W_{*}^{\dagger}=\left(\mathbb{R}^{-}\times-Y'\right)\cup W^{\dagger}\cup(\mathbb{R}^{+}\times-Y)
\]
which are asymptotic on $\mathbb{R}^{-}\times-Y'$ to an irreducible
critical point $[\mathfrak{a}]\in\mathfrak{C}^{o}(-Y',\mathfrak{s}_{\xi'})$
and asymptotic on $\mathbb{R}^{+}\times-Y$ to a boundary stable critical
point $[\mathfrak{b}]\in\mathfrak{C}^{s}(-Y,\mathfrak{s}_{\xi})$
. 

\begin{figure}[H]
\begin{centering}
\includegraphics[scale=0.5]{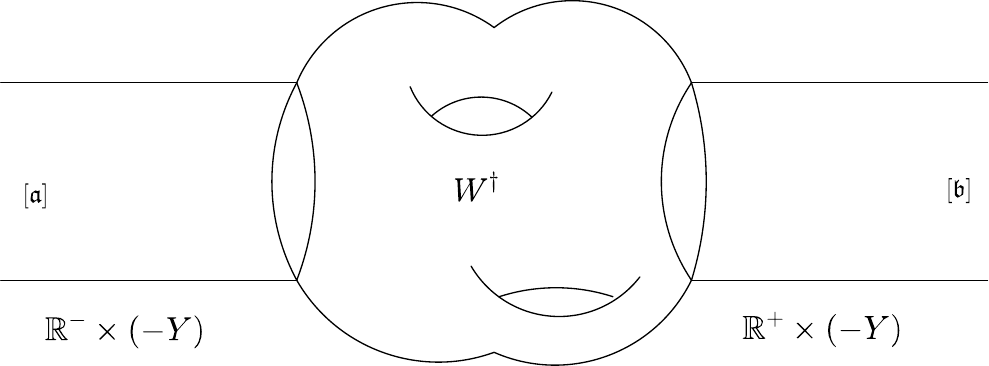}
\par\end{centering}
\caption{Manifold $W_{*}^{\dagger}$ with two cylindrical ends used to define
the cobordism maps.}

\end{figure}

Again, we obtain a moduli space $\mathcal{M}([\mathfrak{a}],W_{*}^{\dagger},\mathfrak{s}_{\omega},[\mathfrak{b}])$
and as before we can define 
\[
n_{z}([\mathfrak{a}],W_{*}^{\dagger},\mathfrak{s}_{\omega},[\mathfrak{b}])=\begin{cases}
|\mathcal{M}_{z}([\mathfrak{a}],W_{*}^{\dagger},\mathfrak{s}_{\omega},[\mathfrak{b}])|\;\;\;\mod2 & \text{if }\dim\mathcal{M}_{z}([\mathfrak{a}],W_{*}^{\dagger},\mathfrak{s}_{\omega},[\mathfrak{b}])=0\\
0 & \text{otherwise }
\end{cases}
\]
 On the other hand, the map $\bar{\partial}_{u}^{s}$ counts solutions
on $\mathbb{R}\times(-Y')$ which are asymptotic to a boundary stable
critical point $[\mathfrak{a}]\in\mathfrak{C}^{s}(-Y',\mathfrak{s}_{\xi'})$
as $t\rightarrow-\infty$ and to a boundary unstable critical point
$[\mathfrak{b}]\in\mathfrak{C}^{u}(-Y',\mathfrak{s}_{\xi'})$ as $t\rightarrow\infty$
(in our context a map like $\partial_{o}^{u}$ would count solutions
on $\mathbb{R}\times-Y$ instead). The bar indicates that we are only
considering reducible solutions, i.e, solutions where the spinor vanishes
identically. In the case of a cylinder there is a natural $\mathbb{R}$
action and the corresponding moduli space after we quotient out by
this action is denoted $\check{\mathcal{M}}([\mathfrak{a}],\mathfrak{s}_{\xi'},[\mathfrak{b}])$
(the notation in \cite{MR2388043} for this moduli space is $\check{M}_{z}([\mathfrak{a}],\mathfrak{s}_{\xi'},[\mathfrak{b}])$).
In this case we define 
\[
n_{z}([\mathfrak{a}],\mathfrak{s}_{\xi'},[\mathfrak{b}])=\begin{cases}
|\check{\mathcal{M}}_{z}([\mathfrak{a}],\mathfrak{s}_{\xi'},[\mathfrak{b}])|\;\;\;\mod2 & \text{if }\dim\check{\mathcal{M}}_{z}([\mathfrak{a}],\mathfrak{s}_{\xi'},[\mathfrak{b}])=0\\
0 & \text{otherwise}
\end{cases}
\]
From the formula one can see that $\check{m}c(\xi')$ has two terms,
and since are working $\mod2$ we will write them without the signs
to simplify the expression. The term corresponding to 
\[
m_{o}^{o}c^{o}(\xi')+m_{o}^{u}\bar{\partial}_{u}^{s}c^{s}(\xi')+\partial_{o}^{u}\bar{m}_{u}^{s}c^{s}(\xi')
\]
 is equivalent to 
\[
\begin{aligned} & \sum_{[\mathfrak{a}]\in\mathfrak{C}^{o}(-Y'),[\mathfrak{c}]\in\mathfrak{C}^{o}(-Y)}\;\;\sum_{z_{1},z_{2}}m_{z_{1}}(Z_{Y',\xi'}^{+},\mathfrak{s}',[\mathfrak{a}])n_{z_{2}}([\mathfrak{a}],W_{*}^{\dagger},\mathfrak{s}_{\omega},[\mathfrak{c}])e_{[\mathfrak{c}]}\\
+ & \sum_{[\mathfrak{a}]\in\mathfrak{C}^{s}(-Y'),[\mathfrak{b}]\in\mathfrak{C}^{u}(-Y'),[\mathfrak{c}]\in\mathfrak{C}^{o}(-Y)}\;\;\sum_{z_{1},z_{2},z_{3}}m_{z_{1}}(Z_{Y',\xi'}^{+},\mathfrak{s}',[\mathfrak{a}])\bar{n}_{z_{2}}([\mathfrak{a}],\mathfrak{s}_{\xi'},[\mathfrak{b}])n_{z_{3}}([\mathfrak{b}],W_{*}^{\dagger},\mathfrak{s}_{\omega},[\mathfrak{c}])e_{[\mathfrak{c}]}\\
+ & \sum_{[\mathfrak{a}]\in\mathfrak{C}^{s}(-Y'),[\mathfrak{b}]\in\mathfrak{C}^{u}(-Y),[\mathfrak{c}]\in\mathfrak{C}^{o}(-Y)}\;\;\sum_{z_{1},z_{2},z_{3}}m_{z_{1}}(Z_{Y',\xi'}^{+},\mathfrak{s}',[\mathfrak{a}])\bar{n}_{z_{2}}([\mathfrak{a}],W_{*}^{\dagger},\mathfrak{s}_{\omega},[\mathfrak{b}])n_{z_{3}}([\mathfrak{b}],\mathfrak{s}_{\xi},[\mathfrak{c}])e_{[\mathfrak{c}]}
\end{aligned}
\]

Notice that if we fix a critical point $[\mathfrak{c}]\in\mathfrak{C}^{o}(-Y,\mathfrak{s}_{\xi})$
we can consider the coefficient
\begin{align}
 & \sum_{[\mathfrak{a}]\in\mathfrak{C}^{o}(-Y')}\;\;\sum_{z_{1},z_{2}}m_{z_{1}}(Z_{Y',\xi'}^{+},\mathfrak{s}',[\mathfrak{a}])n_{z_{2}}([\mathfrak{a}],W_{*}^{\dagger},\mathfrak{s}_{\omega},[\mathfrak{c}])\label{sum naturality 1}\\
+ & \sum_{[\mathfrak{a}]\in\mathfrak{C}^{s}(-Y'),[\mathfrak{b}]\in\mathfrak{C}^{u}(-Y')}\;\;\sum_{z_{1},z_{2},z_{3}}m_{z_{1}}(Z_{Y',\xi'}^{+},\mathfrak{s}',[\mathfrak{a}])\bar{n}_{z_{2}}([\mathfrak{a}],\mathfrak{s}_{\xi'},[\mathfrak{b}])n_{z_{3}}([\mathfrak{b}],W_{*}^{\dagger},\mathfrak{s}_{\omega},[\mathfrak{c}])\nonumber \\
+ & \sum_{[\mathfrak{a}]\in\mathfrak{C}^{s}(-Y'),[\mathfrak{b}]\in\mathfrak{C}^{u}(-Y)}\;\;\sum_{z_{1},z_{2},z_{3}}m_{z_{1}}(Z_{Y',\xi'}^{+},\mathfrak{s}',[\mathfrak{a}])\bar{n}_{z_{2}}([\mathfrak{a}],W_{*}^{\dagger},\mathfrak{s}_{\omega},[\mathfrak{b}])n_{z_{3}}([\mathfrak{b}],\mathfrak{s}_{\xi},[\mathfrak{c}])\nonumber 
\end{align}
Similarly, for each critical point $[\mathfrak{c}]\in\mathfrak{C}^{s}(-Y,\mathfrak{s}_{\xi})$
, the coefficient of $e_{[\mathfrak{c}]}$ in 
\[
m_{s}^{o}c^{o}(\xi')+\bar{m}_{s}^{s}c^{s}(\xi')+m_{s}^{u}\bar{\partial}_{u}^{s}c^{s}(\xi')+\partial_{s}^{u}\bar{m}_{u}^{s}c^{s}(\xi')
\]
 is given by 
\begin{align}
 & \sum_{[\mathfrak{a}]\in\mathfrak{C}^{o}(-Y')}\;\;\sum_{z_{1},z_{2}}m_{z_{1}}(Z_{Y',\xi'}^{+},\mathfrak{s}',[\mathfrak{a}])n_{z_{2}}([\mathfrak{a}],W_{*}^{\dagger},\mathfrak{s}_{\omega},[\mathfrak{c}])\label{sum naturality 2}\\
+ & \sum_{[\mathfrak{a}]\in\mathfrak{C}^{s}(-Y')}\;\;\sum_{z_{1},z_{2}}m_{z}(Z_{Y',\xi'}^{+},\mathfrak{s}',[\mathfrak{a}])\bar{n}_{z_{2}}([\mathfrak{a}],W_{*}^{\dagger},\mathfrak{s}_{\omega},[\mathfrak{c}])\nonumber \\
+ & \sum_{[\mathfrak{a}]\in\mathfrak{C}^{s}(-Y'),[\mathfrak{b}]\in\mathfrak{C}^{u}(-Y')}\;\;\sum_{z_{1},z_{2},z_{3}}m_{z_{1}}(Z_{Y',\xi'}^{+},\mathfrak{s}',[\mathfrak{a}])\bar{n}_{z_{2}}([\mathfrak{a}],\mathfrak{s}_{\xi'},[\mathfrak{b}])n_{z_{3}}([\mathfrak{b}],W_{*}^{\dagger},\mathfrak{s}_{\omega},[\mathfrak{c}])\nonumber \\
+ & \sum_{[\mathfrak{a}]\in\mathfrak{C}^{s}(-Y'),[\mathfrak{b}]\in\mathfrak{C}^{u}(-Y)}\;\;\sum_{z_{1},z_{2},z_{3}}m_{z_{1}}(Z_{Y',\xi'}^{+},\mathfrak{s}',[\mathfrak{a}])\bar{n}_{z_{2}}([\mathfrak{a}],W_{*}^{\dagger},\mathfrak{s}_{\omega},[\mathfrak{b}])n_{z_{3}}([\mathfrak{b}],\mathfrak{s}_{\xi},[\mathfrak{c}])\nonumber 
\end{align}
Therefore, we want to show that up to a boundary term, $\sum_{z}m_{z}(Z_{Y,\xi}^{+},\mathfrak{s},[\mathfrak{c}])$
is equal to (\ref{sum naturality 1}) (if $[\mathfrak{c}]$ is irreducible)
or (\ref{sum naturality 2}) (if $[\mathfrak{c}]$ is boundary stable).

If there is any hope of showing the equality between these two quantities
we need to find a geometric interpretation to the sums (\ref{sum naturality 1}),
(\ref{sum naturality 2}). In order to do this we will consider the
Seiberg-Witten equations on a slightly more general scenario, one
that combines the construction of the contact invariant with the cobordism.
More precisely, we will study the Seiberg Witten equations on 
\[
W_{\xi',Y}^{+}=\left([1,\infty)\times Y'\right)\cup W^{\dagger}\cup\left(\mathbb{R}^{+}\times-Y\right)
\]
which are asymptotic on $[1,\infty)\times Y'$ to the canonical solution
coming from the contact structure $\xi'$ and asymptotic on $\mathbb{R}^{+}\times-Y$
to a critical point $[\mathfrak{c}]\in\check{C}_{*}(-Y,\mathfrak{s}_{\xi})$.
The moduli space of such solutions will naturally be denoted $\mathcal{M}(W_{\xi',Y}^{+},\mathfrak{s}_{\omega},[\mathfrak{c}])$
.

\begin{figure}[H]
\begin{centering}
\includegraphics[scale=0.5]{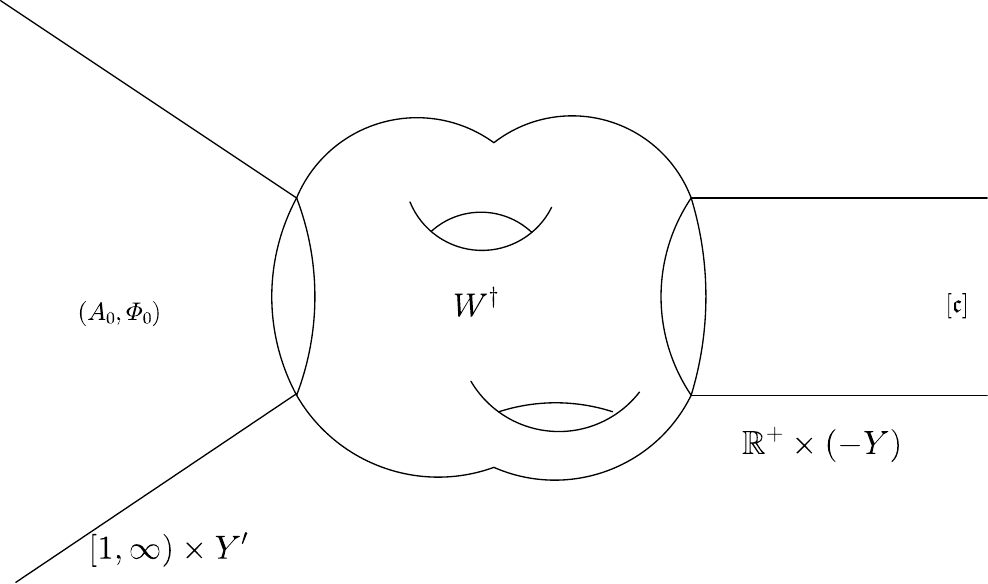}
\par\end{centering}
\caption{\label{fig:Hybrid invariant}Manifold $W_{\xi',Y}^{+}$ used to define
the ``hybrid'' invariant $\mathbf{c}(\xi',Y)$. }
\end{figure}

Thanks to the compactness arguments in \cite{MR1474156,MR2388043}
and \cite{Zhang[2016]} (which guarantee uniform exponential decay
along the conical end) we can proceed as before and define
\[
m_{z}(W_{\xi',Y}^{+},\mathfrak{s}_{\omega},[\mathfrak{c}])=\begin{cases}
|\mathcal{M}_{z}(W_{\xi',Y}^{+},\mathfrak{s}_{\omega},[\mathfrak{c}])|\;\;\;\mod2 & \text{if }\dim\mathcal{M}_{z}(W_{\xi',Y}^{+},\mathfrak{s}_{\omega},[\mathfrak{c}])=0\\
0 & \text{otherwise}
\end{cases}
\]
These numbers give rise to the hybrid invariant $\mathbf{c}(\xi',Y)$
mentioned at the beginning of this section. In order to show the equality
$\widecheck{HM}_{\bullet}(W^{\dagger},\mathfrak{s}_{\omega})\mathbf{c}(\xi')=\mathbf{c}(\xi',Y)$
we must consider the parametrized moduli space
\begin{equation}
\bigcup_{L\in[0,\infty)}\{L\}\times\mathcal{M}(W_{\xi',Y}^{+}(L),\mathfrak{s}_{\omega},[\mathfrak{c}])\label{parametrized}
\end{equation}
where $\mathcal{M}(W_{\xi',Y}^{+}(L),\mathfrak{s}_{\omega},[\mathfrak{c}])$
denotes the moduli space of solutions to the Seiberg-Witten equations
on the manifold 
\[
W_{\xi',Y}^{+}(L)=\left([1,\infty)\times Y'\right)\cup\left([0,L]\times-Y'\right)\cup W^{\dagger}\cup\left(\mathbb{R}^{+}\times-Y\right)
\]

\begin{figure}[H]
\begin{centering}
\includegraphics[scale=0.5]{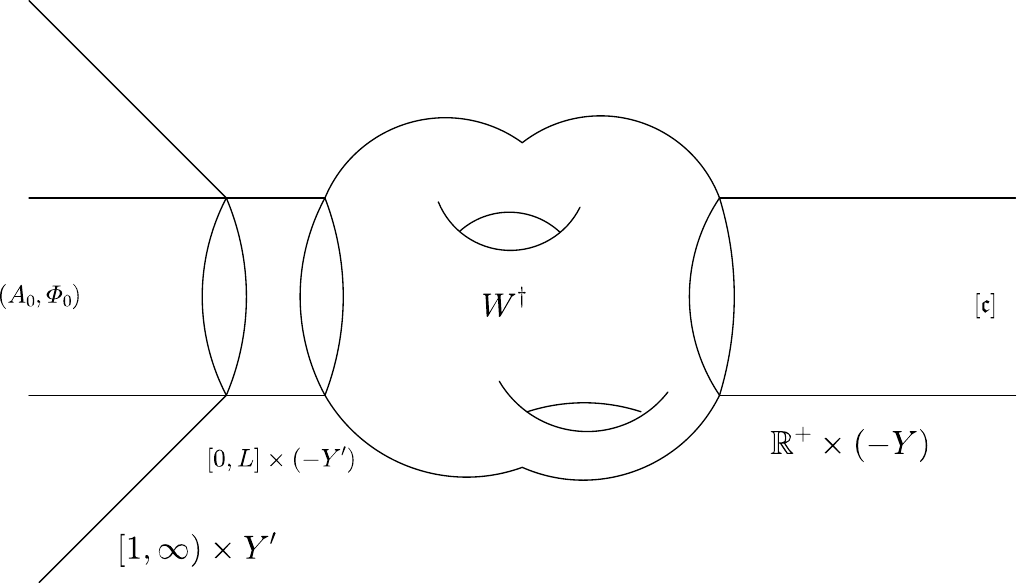} 
\par\end{centering}
\caption{$\check{HM}(W^{\dagger},\mathfrak{s}_{\omega})\mathbf{c}(\xi')=\mathbf{c}(\xi',Y)$
via a ``stretching the neck'' argument.}
\end{figure}

The parametrized moduli space (\ref{parametrized}) is not compact;
its compactification will be denoted 
\begin{equation}
\bigcup_{L\in[0,\infty]}\{L\}\times\mathcal{M}^{+}(W_{\xi',Y}^{+}(L),\mathfrak{s}_{\omega},[\mathfrak{c}])\label{parametrized compact}
\end{equation}
where the definition of $\mathcal{M}^{+}(W_{\xi',Y}^{+}(\infty),\mathfrak{s}_{\omega},[\mathfrak{c}])$
is given in (\ref{infinite stretch}).\textbf{ }For now, it suffices
to say that when we count the endpoints of all one dimensional moduli
spaces inside (\ref{parametrized compact}) we will get $0$. 

The count coming from the fiber over $L=0$ will give the term $\sum_{z}m_{z}(W_{\xi',Y}^{+},\mathfrak{s}_{\omega},[\mathfrak{c}])$
corresponding to the hybrid invariant $c(\xi',Y)$ while the count
coming from the fiber over $L=\infty$ will give the coefficients
(\ref{sum naturality 1}) and (\ref{sum naturality 2}) of the image
of $\check{m}c(\xi')$. Finally, the count coming from the other fibers
will contribute a boundary term (see Theorem (\ref{boundary parametrized})
for the precise statement). At the level of homology, this means that
$\widecheck{HM}_{\bullet}(W^{\dagger},\mathfrak{s}_{\omega})\mathbf{c}(\xi')=\mathbf{c}(\xi',Y)$
so at this point the naturality proof has been reduced to showing
that $\mathbf{c}(\xi',Y)=\mathbf{c}(\xi)$. Again, from the chain
level perspective this means that up to boundary terms, for each critical
point $[\mathfrak{c}]$ the numbers $\sum_{z}m_{z}(W_{\xi',Y}^{+},\mathfrak{s}_{\omega},[\mathfrak{c}])$
must equal $\sum_{z'}m_{z'}(Z_{Y,\xi}^{+},\mathfrak{s},[\mathfrak{c}])$. 

If one were to replace the half-cylindrical end $\mathbb{R}^{+}\times(-Y)$
with a compact piece $X$ so that we could work with numbers instead
of homology classes, the previous quantities would be the same due
to Theorem $D$ in \cite{MR2199446} (i.e, equation (\ref{Mrowka Rollin nat})
in our paper). Therefore, it becomes clear at this point that what
we need to do is adapt the Mrowka-Rollin theorem to the case in which
we have a half-infinite cylinder. 

Two things that change in this new setup are that certain inclusions
of Sobolev spaces are no longer compact, and in order to achieve transversality
(i.e, obtain unobstructed moduli spaces in the terminology of \cite{MR2199446})
one must use the ``abstract perturbations'' defined by Kronheimer
and Mrowka in \cite{MR2388043}. In particular, these perturbations
introduce new terms that do not appear in the usual linearizations
of the Seiberg-Witten equations, so for the gluing argument we will
employ one needs to check that the new contributions do not mess up
the desired behavior of the linearized Seiberg Witten equations. Namely,
we will see that the contributions have leading terms which are quadratic
in a appropriate sense. Had the leading term been linear, the gluing
argument would not have worked.

Our gluing argument and the proof of Theorem $D$ \cite{MR2199446}
morally follows the same basic ideas as the other gluing arguments
in gauge theory but as expected differs in the specific details (a
few references include \cite{MR2388043,MR1787219,MR2191904,MR1883043,MR1287851,MR1081321,MR2465077}).
Perhaps the most common gluing argument in gauge theory is the one
involving the ``\textit{stretching the neck}'' operation on a closed
oriented Riemannian 4 manifold $X$ which has a separating hypersurface
$Y$ inside it. 
\begin{center}
\begin{figure}[H]
\begin{centering}
\includegraphics[scale=0.3]{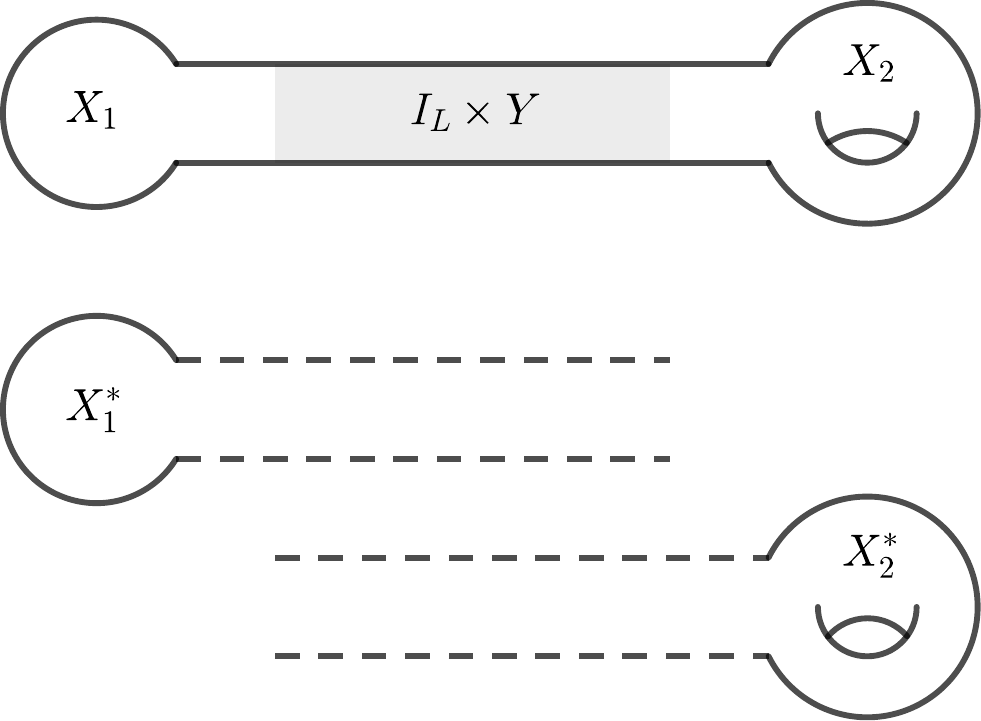}
\par\end{centering}
\caption{Gluing technique for the ``stretching the neck'' argument. }
\end{figure}
\par\end{center}

Namely, one writes $X$ as $X=X_{1}\cup X_{2}$ and after choosing
a metric which is cylindrical near $Y$ one can stretch the metric
along $Y$ in order to have a cylinder $I_{L}\times Y$ of length
$L$ inserted between $X_{1}$ and $X_{2}$ as shown in the picture.
The point is that as $L$ increases, the Seiberg Witten equations
on $X_{L}=X_{1}\cup(I_{L}\times Y)\cup X_{2}$ start behaving more
like the solutions on the manifolds with cylindrical ends $X_{1}^{*}$
and $X_{2}^{*}$. More precisely, one can start from solutions on
$X_{1}^{*}$ and $X_{2}^{*}$ which agree on their respective ends
in order to construct a pre-solution on $X_{L}$ , that is, a configuration
on $X_{L}$ which is a solution to the Seiberg Witten equations on
$X_{L}$, except perhaps for a region supported on $I_{L}\times Y$.
The main point of the gluing argument is that one can find an $L_{0}$
sufficiently large, so that for all $L$ bigger than $L_{0}$ we can
obtain an actual solution to the Seiberg Witten equations on $X_{L}$
thanks to an application of the implicit function theorem for Banach
spaces. In order for this to work it is imperative to have estimates
that become independent of $L$. 

Likewise, in our situation we want to take advantage of the fact that
for a strong symplectic cobordism the symplectic structure is given
near the boundary by the symplectization of the contact structure,
so that in analogy with the cylindrical case we can perform a ``\textit{dilating
the cone}'' operation, where now the key parameter is a dilation
parameter $\tau$ , which determines the size of the cone $C_{\tau}$
determined by the symplectization of the contact structure near the
boundary. As in the cylindrical case, the main idea is that once $\tau$
is sufficiently large, the moduli space of solutions to the Seiberg
Witten equations on the manifold shown below can be described in terms
of the moduli space used to define the contact invariant of $(Y,\xi)$.
Again, this will rely on an application of the implicit function theorem,
which requires guaranteeing that certain estimates become independent
of $\tau$ (once it becomes sufficiently large). As we will explain
near the end of the paper this gluing theorem will establish that
$\mathbf{c}(\xi',Y)=\mathbf{c}(\xi)$. 

\begin{figure}[H]
\begin{centering}
\includegraphics[scale=0.35]{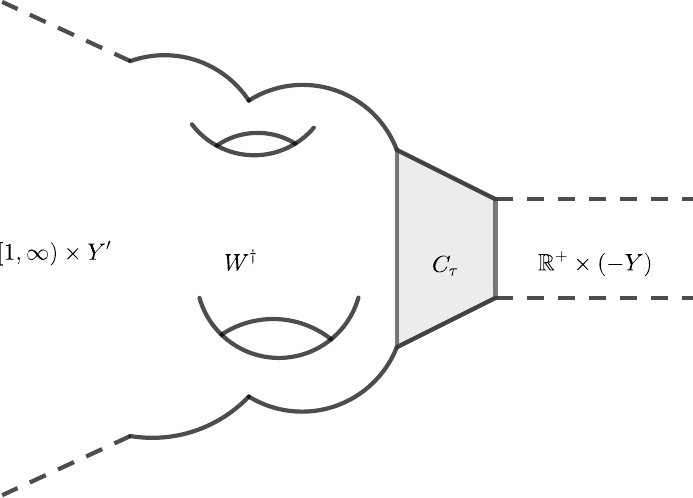}
\par\end{centering}
\caption{``Dilating the cone'' argument used to show that $\mathbf{c}(\xi',Y)=\mathbf{c}(\xi)$.}
\end{figure}

\section{3. Setting Up The Equations }

As explained in the previous section, we will analyze first the equations
on $W_{\xi',Y}^{+}$. In particular, we begin by stating some basic
geometric properties of the manifolds we are going to be working with.
\begin{assumption}
Suppose we have a closed oriented three manifold $Y$ with contact
structure $\xi$. We write $\xi=\ker\theta$ and choose the unique
Riemannian metric $g_{\theta}$ such that \cite[Section 2.3]{MR1474156}:

$\bullet$ The contact form $\theta$ has unit length.

$\bullet$ $d\theta=2*_{Y}\theta$, where $*_{Y}$ is the Hodge star
on $(Y,g_{\theta})$.

$\bullet$ If $J$ is a a choice of an almost complex structure on
$\xi$ then for any $v,w\in\xi$, $g_{\theta}(v,w)=d\theta(v,Jw)$.
\end{assumption}

The contact structure $\xi$ determines a canonical spin-c structure
$\mathfrak{s}_{\xi}$: \textcolor{black}{define the spinor bundle
$S$ as the rank-2 vector bundle $S=\underline{\mathbb{C}}\oplus\xi$
where $\underline{\mathbb{C}}$ is the trivial vector bundle and we
are considering $\xi$ as a complex line bundle. Moreover, there is
a Clifford map $\rho_{Y}:TY\rightarrow\hom(S,S)$ which identifies
$TY$ isometrically with the subbundle $\mathfrak{su}(S)$ of traceless,
skew-adjoint endomorphisms equipped with the inner product $\frac{1}{2}\text{tr}(a^{*}b)$}
\cite[Section 1.1]{MR2388043}\textcolor{black}{. Using $(Y,g_{\theta},\mathfrak{s}_{\xi})$
we can write the configuration space on which the Seiberg-Witten equations
are defined} \cite[Section 9.1]{MR2388043}\textcolor{black}{: }for
any integer or half integer $k\geq0$ define 
\[
\mathcal{C}_{k}(Y,\mathfrak{s}_{\xi})=(B_{ref},0)+L_{k}^{2}(M;iT^{*}Y\oplus S)=\mathcal{A}_{k}(Y,\mathfrak{s}_{\xi})\times L_{k}^{2}(Y;S)
\]
where $B_{ref}$ is a reference smooth connection on the spinor bundle
$S$ compatible with the Levi-Civita connection defined on $TY$ and
$\mathcal{A}_{k}(Y,\mathfrak{s}_{\xi})$ denotes the (affine) space
of spin-c connections of $S$ with Sobolev regularity $L_{k}^{2}$.
We will always assume whenever needed that $k\geq5$, but by elliptic
regularity the constructions end up being independent of $k$ because
one can always find a smooth representative in each gauge equivalence
class of solutions to the Seiberg-Witten equations so will not dwell
a lot on the actual value of $k$ being used. 

The gauge group $\mathcal{G}_{k+1}(Y)$ is 
\[
\mathcal{G}_{k+1}(Y)=\{u\in L_{k+1}^{2}(Y;\mathbb{C})\mid|u|=1\text{ pointwise}\}
\]
It acts on the configuration space via 
\[
u\cdot(B,\varPsi)=(B-u^{-1}du,u\varPsi)
\]
The action is not free at the \textbf{reducible configurations}, that
is, the configurations $(B,0)$ with the spinor component identically
zero. The stabilizer at those configurations consists of the constant
maps $u:Y\rightarrow S^{1}$ which we can identify with $S^{1}$.
To handle reducible configurations Kronheimer and Mrowka introduced
the \textbf{blown-up configuration space} \cite[Section 6.1]{MR2388043}
\[
\mathcal{C}_{k}^{\sigma}(Y,\mathfrak{s}_{\xi})=\{(B,s,\phi)\mid\|\phi\|_{L^{2}(Y)=1},s\geq0\}=\mathcal{A}_{k}(Y,\mathfrak{s}_{\xi})\times\mathbb{R}^{\geq}\times\mathbb{S}(L_{k}^{2}(Y;S))
\]
Here $\mathbb{S}(L_{k}^{2}(Y;S)$ denotes those elements $\phi$ in
$L_{k}^{2}(Y;S)$ whose $L^{2}$ norm (not $L_{k}^{2}$ norm!) is
equal to $1$. In this case the gauge action is 
\[
u\cdot(B,s,\phi)=(B-u^{-1}du,s,u\phi)
\]
and it is easy to check that the gauge group acts freely on this space.
In fact, Lemma 9.1.1 in \cite{MR2388043} shows that the space $\mathcal{C}_{k}^{\sigma}(Y,\mathfrak{s}_{\xi})$
is naturally a Hilbert manifold with boundary and when $k\geq1$,
the space $\mathcal{G}_{k+1}(Y)$ is a Hilbert Lie group which acts
smoothly and freely on $\mathcal{C}_{k}^{\sigma}(Y,\mathfrak{s}_{\xi})$.

We are interested in triples $(B,s,\phi)$ which satisfy a perturbed
version of the Seiberg-Witten equations. At this point the nature
of the perturbations is not that important. For now it suffices to
say that we will take them to be \textbf{strongly tame perturbations
}as in definition 3.6 of \cite{Zhang[2016]}. As a technical point
it is useful to note that the cylindrical functions constructed in
section 11.1 of \cite{MR2388043} are strongly tame perturbations
so the theorems from \cite{MR2388043} which used this class of perturbations
continue to work in this context. We will denote such a perturbation
by $\mathfrak{q}_{Y,g_{\theta},\mathfrak{s}_{\xi}}$. In general a
strongly tame perturbation $\mathfrak{q}$ can be regarded as a map
$\mathfrak{q}:\mathcal{C}_{k}(Y,\mathfrak{s}_{\xi})\rightarrow L_{k}^{2}(Y;iT^{*}Y\oplus S)$,
where one thinks of the codomain as a copy of the tangent space $T_{(B,\varPsi)}\mathcal{C}_{k}(Y,\mathfrak{s}_{\xi})$
for each configuration $(B,\varPsi)\in\mathcal{C}_{k}(Y,\mathfrak{s}_{\xi})$.
Since the codomain naturally splits one can write $\mathfrak{q}=(\mathfrak{q}^{0},\mathfrak{q}^{1})$
and in section 10.2 of \cite{MR2388043} it is explained how $\mathfrak{q}$
gives rise to a perturbation on the blown-up configuration space $\mathfrak{q}^{\sigma}=(\mathfrak{q}^{0},\hat{\mathfrak{q}}^{1,\sigma})$
(notice that only the second component is modified).

The corresponding equations $(B,s,\phi)$ satisfy are \cite[Section 10.3]{MR2388043}
\begin{equation}
\begin{cases}
\frac{1}{2}*F_{B^{t}}+s^{2}\rho_{Y}^{-1}(\phi\phi^{*})_{0}+\mathfrak{q}_{Y,g_{\theta},\mathfrak{s}_{\xi}}^{0}(B,s\phi)=0\\
\varLambda_{\mathfrak{q}_{Y,g_{\theta},\mathfrak{s}_{\xi}}}(B,s,\phi)s=0\\
D_{B}\phi-\varLambda_{\mathfrak{q}_{Y,g_{\theta},\mathfrak{s}_{\xi}}}(B,s,\phi)\phi+\tilde{\mathfrak{q}}_{Y,g_{\theta},\mathfrak{s}_{\xi}}^{1}(B,s,\phi)=0
\end{cases}\label{equations for critical points}
\end{equation}
where:

$\bullet$ $F_{B^{t}}$ denotes the curvature of the connection $B^{t}$
on $\det(S)$.

$\bullet$ $(\phi\phi^{*})_{0}$ denotes the trace-free part of the
hermitian endomorphism $\phi\phi^{*}$: $(\phi\phi^{*})_{0}=\phi\phi^{*}-\frac{1}{2}|\phi|^{2}1_{S}$.

$\bullet$ $D_{B}$ is the Dirac operator corresponding to the connection
$B$.

$\bullet$ $\varLambda_{\mathfrak{q}_{Y,g_{\theta},\mathfrak{s}_{\xi}}}(B,s,\phi)=\text{Re}\left\langle \phi,D_{B}\phi+\tilde{\mathfrak{q}}_{Y,g_{\theta},\mathfrak{s}_{\xi}}^{1}(B,s,\phi)\right\rangle _{L^{2}(Y)}$
and $\tilde{\mathfrak{q}}^{1}(B,r,\psi)=\int_{0}^{1}\mathcal{D}_{(B,sr\psi)}\mathfrak{q}^{1}(0,\psi)ds$
(here $\mathcal{D}$ denotes the linearization of the map $\mathfrak{q}^{1}$).

Using the equations (\ref{equations for critical points}) we can
distinguish three types of solutions (or critical points) $\mathfrak{c}=(B,s,\phi)$
\cite[Definition 4.4]{MR2299739}, the \textbf{irreducible critical
point, }the \textbf{boundary stable reducible critical point }and
the \textbf{boundary unstable reducible critical point. }What is important
about this classification for us is that solutions of the four dimensional
Seiberg Witten equations on $\mathbb{R}\times Y$ for which the spinor
does not vanish identically can only be asymptotic as $t\rightarrow\infty$
to irreducible critical points or boundary stable reducible critical
points. The gauge equivalence class of any of these points will be
denoted as $[\mathfrak{c}]$. 

The triple $(Y,g_{\theta},\mathfrak{s}_{\xi})$ induces a spin-c structure
on $(-Y,g_{\theta})$ given by the same spinor bundle $S_{\xi}$ and
changing the Clifford multiplication from $\rho_{\xi}$ to $-\rho_{\xi}$
\cite[Section 22.5]{MR2388043}. We will continue to denote this spin-c
structure by $\mathfrak{s}_{\xi}$. Given this structure we can use
the cylindrical metric and the spin-c structure induced by $-Y$ on
the cylinder $\mathbb{R}^{+}\times-Y$ \cite[Section 4.3]{MR2388043}.
We use the perturbation $-\mathfrak{q}_{Y,g_{\theta},\mathfrak{s}_{\xi}}$
on $-Y$. 

Consider now the manifold
\[
W_{\xi',Y}^{+}=\left([1,\infty)\times Y'\right)\cup W^{\dagger}\cup\left(\mathbb{R}^{+}\times-Y\right)
\]
We will define the appropriate geometric structures needed on each
piece together with the perturbations we will be using.

$\bullet$ On $\mathbb{R}^{+}\times-Y$, we use the cylindrical metric
and the canonical spin-c structure induced by $\mathfrak{s}_{\xi}$
on the cylinder. As explained on section 10.1 of \cite{MR2388043},
we have a four dimensional perturbation $-\hat{\mathfrak{q}}_{Y,g_{\theta},\mathfrak{s}_{\xi}}:\mathcal{C}_{k}(\mathbb{R}^{+}\times-Y,\mathfrak{s}_{\xi})\rightarrow L_{k}^{2}(\mathbb{R}^{+}\times-Y;iT^{*}(\mathbb{R}^{+}\times-Y)\oplus S)$
on the half-cylinder $\mathbb{R}^{+}\times-Y$, defined by restriction
to each slice.

$\bullet$ On $W^{\dagger}$ we choose a metric $g_{W}$ on $W^{\dagger}$
such that the metric $g_{W}$ is cylindrical in collar neighborhoods
of the boundary components. To define the perturbation on $W^{\dagger}$
we follow section 24.1 in \cite{MR2388043}. Since the Riemannian
metric is cylindrical in the neighborhood of the boundary it contains
on each boundary component an isometric copy of $I_{1}\times-Y$ and
$I_{2}\times Y'$ where $I_{1}=(-C_{1},0]$, $I_{2}=(-C_{2},0]$.
Since the argument is the same for both ends we will use generic notation.
Let $\beta$ be a cut-off function, equal to $1$ near $t=0$ and
equal to $0$ near $t=-C$. Let $\beta_{0}$ be a bump function with
compact support in $(-C,0)$, equal to one on a compact subset inside
$(-C,0)$, for example, the compact subset $\left[-C/2,-C/4\right]$.
Choose another perturbation $\mathfrak{p}_{0}$ of the three dimensional
equations and consider the perturbation 
\[
\hat{\mathfrak{p}}_{W}=\beta\hat{\mathfrak{q}}+\beta_{0}\hat{\mathfrak{p}}_{0}
\]
It is useful to note that the reason why we use two perturbations
is so that one can be varied when we use a transversality argument. 

$\bullet$ On $[1,\infty)\times Y'$ we assume that the metric is
cylindrical in a collar neighborhood $[1,C_{K})\times Y'$ and on
a complement of this neighborhood (like $N_{K}=[C_{K}+1,\infty)\times Y'$
for instance) it is given by the metric 
\[
g_{K,\theta'}=dt\otimes dt+t^{2}g_{\theta'}
\]
with symplectic form 
\[
\omega_{\theta'}=\frac{1}{2}d(t^{2}\theta')
\]
Here $K$ stands for Kahler, although in most cases the cone will
not be a Kahler manifold (in fact occurs only when $(Y,\xi)$ is a
Sasakian manifold \cite{MR2382957}. The form is self-dual with respect
to $g_{K,\theta'}$ and $|\omega_{\theta'}|_{g_{K,\theta'}}=\sqrt{2}$
pointwise. By Lemma 2.1 in \cite{MR1474156}, on the symplectic cone
we have a unit length section $\varPhi_{0}$ associated to the canonical
spinor bundle $S_{\omega_{\theta'}}$. For this section $\varPhi_{0}$
we have a corresponding connection $A_{0}$ such that $D_{A_{0}}\varPhi_{0}=0$.
Choose a smooth extension of $(A_{0},\varPhi_{0})$ to all of $W_{\xi',Y}^{+}$
in such a way that $(A_{0},\varPhi_{0})$ is translation invariant
on the cylindrical end $\mathbb{R}^{+}\times-Y$. Define 
\[
\mathfrak{p}_{K}=\left(-\frac{1}{2}\rho(F_{A_{0}^{t}}^{+})+(\varPhi_{0}\varPhi_{0}^{*})_{0},0\right)
\]
and choose a cutoff function $\beta_{K}$ which is supported on $N_{K}$
and identically equal to $1$ on $[C_{K}+2,\infty)$. Choose also
a cutoff function $\beta_{N_{K}}$ which is supported on $[1,C_{K})\times Y'$
and identically equal to $1$ near the boundary $\partial\left([1,C_{K})\times Y'\right)$. 

Our global perturbation will be 
\begin{equation}
\mathfrak{p}_{W_{\xi',Y}^{+}}=-\hat{\mathfrak{q}}_{Y,g_{\theta},\mathfrak{s}_{\xi}}+\left(\beta\hat{\mathfrak{q}}_{Y,g_{\theta},\mathfrak{s}_{\xi}}+\beta_{0}'\hat{\mathfrak{p}}_{0}\right)+\left(\beta_{0}'\hat{\mathfrak{p}}_{0}'+\beta'\hat{\mathfrak{q}}_{Y',g_{\theta'},\mathfrak{s}_{\xi'}}\right)+\left(\beta_{N_{K}}\hat{\mathfrak{q}}_{Y',g_{\theta'},\mathfrak{s}_{\xi'}}+\beta_{K}\mathfrak{p}_{K}\right)\label{eq:glued perturbation}
\end{equation}
where $\beta_{0}',\beta'$ are cutoff functions defined analogously
for the other cylindrical neighborhood $I_{2}\times Y'$.

In words the previous perturbation behaves as follows: if we start
on the cylindrical end $\mathbb{R}^{+}\times-Y$ we will see the translation
invariant perturbation $-\hat{\mathfrak{q}}_{Y,g_{\theta},\mathfrak{s}_{\xi}}$.
As we enter the cobordism through the boundary $-Y\subset W^{\dagger}$
(recall that $\partial W^{\dagger}=-Y\sqcup Y'$) this perturbation
is modified into a combined perturbation $\beta\hat{\mathfrak{q}}_{Y,g_{\theta},\mathfrak{s}_{\xi}}+\beta_{0}'\hat{\mathfrak{p}}_{0}$,
which is supported on a collar neighborhood of this end. After we
exit this collar neighborhood we will see no perturbations until we
reach again the collar neighborhood of the end $Y'\subset W^{\dagger}$
, where the perturbation is $\beta_{0}'\hat{\mathfrak{p}}_{0}'+\beta'\hat{\mathfrak{q}}_{Y',g_{\theta'},\mathfrak{s}_{\xi'}}$.
Finally, as we exit the cobordism we will see a perturbation identically
equal to $\hat{\mathfrak{q}}_{Y',g_{\theta'},\mathfrak{s}_{\xi'}}$
for a small time until it becomes zero again and then it will eventually
be changed into the perturbation identically equal to $\mathfrak{p}_{K}$.
We will explain the reason why the perturbations were chosen in this
way near the end of this section. 

Now we must define the corresponding configuration space that we want
to use in order to analyze the Seiberg-Witten equations. In general
one needs to define the ordinary configuration space and its blow-up
(see sections 13 and 24.2 of \cite{MR2388043} for some motivation
behind this construction). Due to the asymptotic condition we will
impose, our solutions will always be irreducible so the gauge group
action will be free without having to blow up the configuration space.
Therefore, most of the time we will simply use the ordinary configuration
space. However, if one wants to describe the compactification of the
moduli spaces in terms of the space of broken trajectories then the
blow up model is more convenient so for completeness sake we will
write the equations in the blow up model (but we will switch to the
ordinary configuration space when some computations become more transparent
there).

We are interested in the configurations that solve the following perturbed
version of the Seiberg-Witten equations:
\begin{equation}
\mathfrak{F}_{\mathfrak{p}}=\mathfrak{F}+\mathfrak{p}_{W_{\xi',Y}^{+}}=0\label{solutions perturbed}
\end{equation}
where the \textbf{unperturbed Seiberg Witten map} is \cite[Eq. 4.12]{MR2388043}
\[
\mathfrak{F}(A,\varPhi)=\left(\frac{1}{2}\rho(F_{A^{t}}^{+})-(\varPhi\varPhi^{*})_{0},D_{A}\varPhi\right)
\]

Both the perturbed and unperturbed maps are defined on elements of
the following configuration space (def. 3.5 in \cite{Zhang[2016]}
and def. 13.1 in \cite{MR2388043}):
\begin{defn}
\label{configuration space}Define the \textbf{configuration space}
(without blow-up) $\mathcal{C}_{k,loc}(W_{\xi',Y}^{+},\mathfrak{s}_{\omega})$
as follows. It will consist of pairs $(A,\varPhi)$ such that: 

1) $A$ is a locally $L_{k}^{2}$ spin-c connection for $S$ and $\varPhi$
is a locally $L_{k}^{2}$ section of $S^{+}$.

2) It is $L_{k}^{2}$ close to the canonical solution on the conical
end, that is,
\begin{align*}
A-A_{0}\in L_{k}^{2}([1,\infty)\times Y',iT^{*}([1,\infty)\times Y'))\\
\varPhi-\varPhi_{0}\in L_{k,A_{0}}^{2}([1,\infty)\times Y',S^{+})
\end{align*}
\end{defn}

\begin{rem}
a) Recall that we chose an extension of $A_{0}$ to the cylindrical
end in such a way that it was translation invariant so the condition
that $A$ is a locally $L_{k}^{2}$ spin-c connection means that $A-A_{0}\in L_{k,loc}^{2}(W_{\xi',Y}^{+};iT^{*}W_{\xi',Y}^{+})$.

b) Notice that the second condition implies that $\varPhi$ cannot
be identically $0$, i.e, $\mathcal{C}_{k,loc}(W_{\xi',Y}^{+},\mathfrak{s})$
contains no reducible configurations. In the notation of \cite{MR2388043},
we would write $\mathcal{C}_{k,loc}(W_{\xi',Y}^{+},\mathfrak{s})=\mathcal{C}_{k,loc}^{*}(W_{\xi',Y'}^{+},\mathfrak{s})$. 

c) Due to the lack of a norm the space $\mathcal{C}_{k,loc}(W_{\xi',Y}^{+},\mathfrak{s})$
is not a Banach space unless we impose some asymptotic condition on
the cylindrical end. 
\end{rem}

The \textbf{blown-up configuration} \textbf{space} $\mathcal{C}_{k,loc}^{\sigma}(W_{\xi',Y}^{+},\mathfrak{s}_{\omega})$
is defined as follows:
\begin{defn}
If $S$ denotes the spinor bundle, define the sphere $\mathbb{S}$
as the topological quotient of $L_{k,loc}^{2}(W_{\xi',Y}^{+};S^{+})\backslash0$\textbf{
}by the action of $\mathbb{R}^{+}$ \cite[Section 6.1]{MR2388043}.
The blown-up configuration space associated to $\mathcal{C}_{k,loc}(W_{\xi',Y}^{+},\mathfrak{s}_{\omega})$
is
\[
\mathcal{C}_{k,loc}^{\sigma}(W_{\xi',Y}^{+},\mathfrak{s}_{\omega})=\{(A,\mathbb{R}^{+}\phi,\varPhi)\mid\varPhi\in\mathbb{R}^{\geq0}\phi,\;\;\phi\in\mathbb{S}\text{ and }(A,\varPhi)\in\mathcal{C}_{k,loc}(W_{\xi',Y}^{+},\mathfrak{s}_{\omega})\}
\]
 
\end{defn}

Just as its blown-down version, $\mathcal{C}_{k,loc}^{\sigma}(W_{\xi',Y}^{+},\mathfrak{s}_{\omega})$
is not a Banach manifold, much less a Hilbert manifold, so we will
not try to find useful slices on this space. These slices would have
been ``orthogonal'' in some suitable sense to the gauge group action,
which we will take to be 
\begin{equation}
\mathcal{G}_{k+1}(W_{\xi',Y}^{+})=\{u:W_{\xi',Y}^{+}\rightarrow\mathbb{C}^{*}\mid|u|=1\text{ and }1-u\in L_{k+1}^{2}([1,\infty)\times Y')\}\label{definition gauge-1}
\end{equation}
where the action of $u\in\mathcal{G}_{k+1}$ on a triple $(A,\mathbb{R}^{+}\phi,\varPhi)\in\mathcal{C}_{k,loc}^{\sigma}(W_{\xi',Y}^{+},\mathfrak{s}_{\omega})$
is given by 
\begin{equation}
u\cdot(A,\mathbb{R}^{+}\phi,\varPhi)=(A-u^{-1}du,\mathbb{R}^{+}(u\phi),u\varPhi)\label{action gauge group}
\end{equation}

Using the Sobolev multiplication theorems on manifolds with bounded
geometry \cite[Chapter 1]{MR1175322} it is not difficult to verify
that $\mathcal{G}_{k+1}(W_{\xi',Y}^{+})$ is a Hilbert Lie group and
that the previous formula indeed gives an action on the configuration
space $\mathcal{C}_{k,loc}(W_{\xi',Y}^{+},\mathfrak{s}_{\omega})$,
that is:
\begin{lem}
\label{lem: Hilbert Group} Suppose that $k\geq4$. Then $\mathcal{G}_{k+1}(W_{\xi',Y}^{+})$
is a Hilbert Lie group. Moreover, the action of $\mathcal{G}_{k+1}(W_{\xi',Y}^{+})$
on $\mathcal{C}_{k,loc}(W_{\xi',Y}^{+},\mathfrak{s}_{\omega})$ is
well defined in that:

i) if $(A,\varPhi)\in\mathcal{C}_{k}(W_{\xi',Y}^{+},\mathfrak{s}_{\omega})$
and $u\in\mathcal{G}_{k+1}(W_{\xi',Y}^{+})$ then $u\cdot(A,\varPhi)\in\mathcal{C}_{k}(W_{\xi',Y}^{+},\mathfrak{s}_{\omega})$
and similarly, 

ii) if $u\cdot(A,\varPhi)=(\tilde{A},\tilde{\varPhi})$ for two configurations
$(A,\varPhi),(\tilde{A},\tilde{\varPhi})\in\mathcal{C}_{k}(W_{\xi',Y}^{+},\mathfrak{s}_{\omega})$
and $u$ is a $L_{k+1,loc}^{2}(W_{\xi',Y}^{+})$ gauge transformation,
then $1-u\in L_{k+1}^{2}([1,\infty)\times Y')$.
\end{lem}

\begin{rem}
A proof of this lemma can be found in the author's thesis, \cite[Lemma 17]{Echeverria[Thesis]}.
\end{rem}

Therefore it makes sense to define 
\[
\mathcal{B}_{k,loc}^{\sigma}(W_{\xi',Y}^{+},\mathfrak{s}_{\omega})=\mathcal{C}_{k,loc}^{\sigma}(W_{\xi',Y}^{+},\mathfrak{s}_{\omega})/\mathcal{G}_{k+1}(W_{\xi',Y}^{+})
\]
Again, since the original space $\mathcal{C}_{k,loc}^{\sigma}(W_{\xi',Y}^{+},\mathfrak{s}_{\omega})$
is not a Banach manifold, we will not be interested in studying directly
$\mathcal{B}_{k,loc}^{\sigma}(W_{\xi',Y}^{+},\mathfrak{s}_{\omega})$,
although this is the space where the solutions to the Seiberg-Witten
equations live. 

To define the moduli space to the Seiberg-Witten equations, we need
to introduce the $\tau$ model first. Let 
\[
[\mathfrak{c}]\in\mathfrak{C}^{o}(-Y,g_{\theta},\mathfrak{s}_{\xi},-\mathfrak{q}_{Y,g_{\theta},\mathfrak{s}_{\xi}})\cup\mathfrak{C}^{s}(-Y,g_{\theta},\mathfrak{s}_{\xi},-\mathfrak{q}_{Y,g_{\theta},\mathfrak{s}_{\xi}})
\]
 be a critical point \cite[Proposition 12.2.5]{MR2388043} to the
blown -up three dimensional Seiberg Witten equations on $-Y$ (\ref{equations for critical points}).
Write $[\mathfrak{c}]=[(B,s,\phi)]$ and let $\mathfrak{c}=(B,s,\phi)$
be a smooth representative in $\mathcal{C}_{k}^{\sigma}(-Y,\mathfrak{s}_{\xi})$.
The critical point $\mathfrak{c}$ gives rise to a translation invariant
configuration $\gamma_{\mathfrak{c}}$ on the half-infinite cylinder
$\mathbb{R}^{+}\times-Y$. 
\begin{defn}
\label{tau model}Define on $\mathbb{R}^{+}\times-Y$ the $\tau$
model $\mathcal{C}_{k,loc}^{\tau}(\mathbb{R}^{+}\times-Y,\mathfrak{s}_{\xi},\mathfrak{c})$
associated to $\mathfrak{c}$ as the space of triples \cite[Section 13.3]{MR2388043}
\[
\gamma=(A,r(t),\phi(t))\in\mathcal{A}_{k,loc}(\mathbb{R}^{+}\times-Y,\mathfrak{s}_{\xi})\times L_{k,loc}^{2}(\mathbb{R}^{+};\mathbb{R})\times L_{k,loc}^{2}(\mathbb{R}^{+}\times-Y;S^{+})
\]
such that 

$i)$ $\gamma-\gamma_{\mathfrak{c}}\in L_{k,loc}^{2}(iT^{*}(\mathbb{R}^{+}\times-Y))\times L_{k,loc}^{2}(\mathbb{R}^{+};\mathbb{R})\times L_{k,loc}^{2}(\mathbb{R}^{+}\times-Y;S^{+})$,
i.e, $\gamma$ is $L_{k,loc}^{2}$ close to $\gamma_{\mathfrak{c}}$.

$ii)$ For all $t\in\mathbb{R}^{+}$, we have that $r(t)\geq0$. 

$iii)$ For all $t\in\mathbb{R}^{+}$ , we have that $\|\phi(t)\|_{L^{2}(-Y)}=1$,
i.e, on each slice the $L^{2}$ norm (not the $L_{k}^{2}$ norm) is
one.
\end{defn}

There is a natural restriction of the gauge group $\mathcal{G}_{k+1}(W_{\xi',Y}^{+})$
to $\mathbb{R}^{+}\times-Y$ which acts on $\mathcal{C}_{k,loc}^{\tau}(\mathbb{R}^{+}\times-Y,\mathfrak{s}_{\xi},\mathfrak{c})$
via 
\[
u\cdot(A,r(t),\phi(t))=(A-u^{-1}du,r(t),u\phi(t))
\]
The gauge equivalence classes of configurations under this gauge group
action will be denoted as 
\[
\mathcal{B}_{k,loc}^{\tau}(\mathbb{R}^{+}\times-Y,\mathfrak{s}_{\xi},[\mathfrak{c}])=\mathcal{C}_{k,loc}^{\tau}(\mathbb{R}^{+}\times-Y,\mathfrak{s}_{\xi},\mathfrak{c})/\mathcal{G}_{k+1,loc}(\mathbb{R}^{+}\times-Y)
\]

We will also use the \textbf{unique continuation principle, }which
will essentially allow us for the most part to avoid working with
the blow-up model. The versions most convenient to us are Proposition
7.1.4 and Proposition 10.8.1 in \cite{MR2388043}, which can still
be used in our context because the perturbation $\mathfrak{p}_{K}$
used on the conical (symplectic) end involves no spinor component.

These imply that if a solution of the perturbed Dirac equation vanishes
on a slice $\{t\}\times-Y$ of the cylindrical end $\mathbb{R}^{+}\times-Y$,
then it would have to vanish on the entire half-cylinder $\mathbb{R}^{+}\times-Y$
and then on the entire four manifold $W_{\xi',Y}^{+}$. However, since
we will be interested in solutions which are asymptotic on the conical
end to the spinor $\varPhi_{0}$ (which is non-vanishing), this cannot
be the case so we can safely conclude that no such solutions will
exist, that is, our spinor $\varPhi$ will never vanish on an open
set or a cylindrical slice. Thanks to this, the following definition
makes sense (compare with definition 24.2.1 of \cite{MR2388043}):
\begin{defn}
The moduli space $\mathcal{M}(W_{\xi',Y}^{+},\mathfrak{s}_{\omega},[\mathfrak{c}])$
for a critical point 
\[
[\mathfrak{c}]\in\mathfrak{C}^{o}(-Y,g_{\theta},\mathfrak{s}_{\xi},-\mathfrak{q}_{Y,g_{\theta},\mathfrak{s}_{\xi}})\cup\mathfrak{C}^{s}(-Y,g_{\theta},\mathfrak{s}_{\xi},-\mathfrak{q}_{Y,g_{\theta},\mathfrak{s}_{\xi}})
\]
 consists gauge equivalence classes of triples 
\[
[A,\mathbb{R}^{+}\phi,\varPhi]\in\mathcal{B}_{k,loc}^{\sigma}(W_{\xi',Y}^{+},\mathfrak{s}_{\omega})
\]
 such that: 

1) $(A,\mathbb{R}^{+}\phi,\varPhi)\in\mathcal{C}_{k,loc}^{\sigma}(W_{\xi',Y}^{+},\mathfrak{s}_{\omega})$
and $(A,\varPhi)$ satisfies the perturbed Seiberg-Witten equations
$\mathfrak{F}_{\mathfrak{p}}(A,\varPhi)=0$ on $W_{\xi',Y}^{+}$.
Here $\mathfrak{p}$ refers to the perturbation explained before equation
(\ref{solutions perturbed}).

2) Because of the unique continuation principle, $\varPhi$ can not
be identically zero on each of the cylindrical slices. Therefore we
can define for each $t$ \cite[Sections 6.2, 13.1]{MR2388043}:
\[
(r(t),\psi(t))=\left(\|\check{\varPhi}(t)\|_{L^{2}(-Y)},\frac{\check{\varPhi}(t)}{\|\check{\varPhi}(t)\|_{L^{2}(-Y)}}\right)
\]
 Also, if we decompose the covariant derivative $\nabla_{A}$ in the
$\frac{d}{dt}$ direction as 
\[
\nabla_{A,\frac{d}{dt}}=\frac{d}{dt}+a_{t}\otimes1_{S}
\]
we require that $\gamma=(A,r(t),\psi(t))$ be an element of $\mathcal{C}_{k,loc}^{\tau}(\mathbb{R}^{+}\times-Y,\mathfrak{s}_{\xi},\mathfrak{c})$
and that it solves the following Seiberg-Witten equations on the cylinder
\cite[Eq. 10.9]{MR2388043} 
\begin{align*}
\frac{1}{2}\frac{d}{dt}\check{A}^{t}=-\frac{1}{2}*_{-Y}F_{\check{A}^{t}}+da_{t}-r^{2}\rho^{-1}(\psi\psi^{*})_{0}-\mathfrak{q}^{0}(\check{A},r\psi)\\
\frac{d}{dt}r=-\varLambda_{\mathfrak{q}}(\check{A},r,\psi)r\\
\frac{d}{dt}\psi=-D_{\check{A}}\psi-a_{t}\psi-\tilde{\mathfrak{q}}^{1}(\check{A},r,\psi)+\varLambda_{\mathfrak{q}}(\check{A},r,\psi)\psi
\end{align*}
where $\check{A}(t)$ denotes the restriction of $A$ to the $t$
slice. Moreover, we require that the gauge equivalence class $[\gamma]$
of $\gamma$ be asymptotic as $t\rightarrow\infty$ to $[\mathfrak{c}]$
in the sense of Definition 13.1.1 in \cite{MR2388043}.
\end{defn}

The moduli space $\mathcal{M}(W_{\xi',Y}^{+},\mathfrak{s}_{\omega},[\mathfrak{c}])$
is naturally a subset of $\mathcal{B}_{k,loc}^{\sigma}(W_{\xi',Y}^{+},\mathfrak{s}_{\omega})$.
However, since the latter space is not in any natural way a Hilbert
manifold we will use a fiber product description of $\mathcal{M}(W_{\xi',Y}^{+},\mathfrak{s}_{\omega},[\mathfrak{c}])$
instead \cite[Lemma 24.2.2, Lemma 19.1.1]{MR2388043}. The idea is
that we can ``break'' the moduli space $\mathcal{M}(W_{\xi',Y}^{+},\mathfrak{s}_{\omega},[\mathfrak{c}])$
into three moduli spaces which we will show are Hilbert manifolds.
These moduli spaces are the moduli space on the cobordism $\mathcal{M}(W^{\dagger},\mathfrak{s}_{\omega})$,
the moduli space on the half cylinder $\mathcal{M}^{\tau}(\mathbb{R}^{+}\times-Y,\mathfrak{s}_{\xi})$
and the moduli space on the conical end $\mathcal{M}([1,\infty)\times Y',\mathfrak{s}')$.
We included a superscript $\tau$ for the second one to emphasize
that its definition uses the $\tau$ model, which described in definition
(\ref{tau model}). The fiber product description will then also allows
us to show that $\mathcal{M}(W_{\xi',Y}^{+},\mathfrak{s}_{\omega},[\mathfrak{c}])$
has a Hilbert manifold structure but in order to explain this we need
we first begin with a lemma.
\begin{lem}
The moduli spaces $\mathcal{M}(W^{\dagger},\mathfrak{s}_{\omega})$
and $\mathcal{M}^{\tau}(\mathbb{R}^{+}\times-Y,\mathfrak{s}_{\xi})$
are Hilbert manifolds.
\end{lem}

\begin{rem}
Section 4.1, 4.2 and 4.3 of the author's thesis \cite{Echeverria[Thesis]}
gives more details on the construction of these manifolds and the
ideas behind the proof of this lemma.
\end{rem}

\begin{proof}
The piece corresponding to the moduli space $\mathcal{M}(W^{\dagger},\mathfrak{s}_{\omega})$
is described in \cite[Proposition 24.3.1]{MR2388043} where it is
shown to be a Hilbert manifold.

Likewise, the second moduli space $\mathcal{M}^{\tau}(\mathbb{R}^{+}\times-Y,\mathfrak{s}_{\xi})$
is described in \cite[Theorem 14.4.2]{MR2388043} where it is shown
that it is a Hilbert manifold (strictly speaking they analyzed an
entire cylinder $\mathbb{R}\times Y$ rather than a half cylinder
$\mathbb{R}^{+}\times-Y$ but the analysis is essentially the same).
\end{proof}
We will start the next section showing that $\mathcal{M}([1,\infty)\times Y',\mathfrak{s}')$
is a Hilbert manifold, following the arguments in section 24 of \cite{MR2388043}.
At the end of the day, we obtain restriction (or trace) maps 
\begin{align*}
R_{\tau}:\mathcal{M}^{\tau}(\mathbb{R}^{+}\times-Y,\mathfrak{s}_{\xi},[\mathfrak{c}])\rightarrow\mathcal{B}_{k-1/2}^{\sigma}(Y,\mathfrak{s}_{\xi})\\
R_{W}^{-}:\mathcal{M}(W^{\dagger},\mathfrak{s}_{\omega})\rightarrow\mathcal{B}_{k-1/2}^{\sigma}(-Y,\mathfrak{s}_{\xi})\\
R_{W}^{+}:\mathcal{M}(W^{\dagger},\mathfrak{s}_{\omega})\rightarrow\mathcal{B}_{k-1/2}^{\sigma}(Y',\mathfrak{s}_{\xi'})\\
R_{K}:\mathcal{M}([1,\infty)\times Y',\mathfrak{s}')\rightarrow\mathcal{B}_{k-1/2}^{\sigma}(-Y',\mathfrak{s}_{\xi})
\end{align*}
given by restricting the (gauge equivalence class of a) solution to
the boundary of each of the corresponding manifolds. We should point
out that there is an identification between $\mathcal{B}_{k}^{\sigma}(-Y,\mathfrak{s})$
and $\mathcal{B}_{k}^{\sigma}(Y,\mathfrak{s})$ \cite[Section 22.5]{MR2388043}
and we can identify $\mathcal{M}(W_{\xi',Y}^{+},\mathfrak{s},[\mathfrak{c}])$
with the fiber product $\text{Fib}(R_{\tau},R_{W}^{-},R_{W}^{+},R_{K})$
given by 
\begin{equation}
\left\{ \left.\left([\gamma_{\mathbb{R}^{+}\times-Y}],[\gamma_{W}],[\gamma_{[1,\infty)\times Y'}]\right)\;\right|\;R_{\tau}[\gamma_{\mathbb{R}^{+}\times-Y}]=R_{W}^{-}[\gamma_{W}]\text{ and }R_{W}^{+}[\gamma_{W}]=R_{K}[\gamma_{[1,\infty)\times Y'}]\right\} \label{eq:Fiber Product}
\end{equation}
Now we can explain how to give $\mathcal{M}(W_{\xi',Y}^{+},\mathfrak{s},[\mathfrak{c}])$
a Hilbert manifold structure, which is stated precisely in Definition
\ref{def:regularity} and Theorem \ref{Theorem perturbations} of
the next section.

For convenience write $R=(R_{\tau},R_{W}^{-},R_{W}^{+},R_{K})$ and
suppose that 
\[
[\gamma]=\left([\gamma_{\mathbb{R}^{+}\times-Y}],[\gamma_{W}],[\gamma_{[1,\infty)\times Y'}]\right)\in\text{Fib}(R_{\tau},R_{W}^{-},R_{W}^{+},R_{K})
\]
 is such that the map $R$ is transverse at $([\mathfrak{b}],[\mathfrak{b}'])$,
where $[\mathfrak{b}]=R_{\tau}([\gamma_{\mathbb{R}^{+}\times-Y}])=R_{W}^{-}([\gamma_{W}])$
and $[\mathfrak{b}']=R_{W}^{+}([\gamma_{W}])=R_{K}([\gamma_{[1,\infty)\times Y'}])$.
In other words, we want the linearized map $\mathcal{D}_{[\gamma]}R$
to be Fredholm and surjective. If this can be achieved, then near
$[\gamma]$ the space $\text{Fib}(R_{\tau},R_{W}^{-},R_{W}^{+},R_{K})$
will have the structure of smooth manifold of dimension $\dim\ker\mathcal{D}_{[\gamma]}R$.
The Fredholm property is proven in Lemma (\ref{Lemma parametrix})
of our paper in the next section. The surjectivity of the map $\mathcal{D}_{[\gamma]}R$
may not be true for an arbitrary perturbation of the form described
in equation (\ref{eq:glued perturbation}) earlier, however, an application
of Sard's theorem shows that one can choose generic perturbations
such that the surjectivity is achieved as well. In fact, achieving
the surjectivity is essentially the same as the proof Kronheimer and
Mrowka gave for the case of a manifold $X^{*}$ with cylindrical ends
\cite[Proposition 24.4.7]{MR2388043}. By choosing a perturbation
from this generic set, one can then guarantee that 
\[
\text{Fib}(R_{\tau},R_{W}^{-},R_{W}^{+},R_{K})=\mathcal{M}(W_{\xi',Y}^{+},\mathfrak{s}_{\omega},[\mathfrak{c}])
\]
 has the structure of a smooth manifold (possibly disconnected with
components of different dimensions).

\section{4. Transversality and Fiber Products}

\subsection{4.1 The moduli space on the conical end $\mathcal{M}([1,\infty)\times Y',\mathfrak{s}')$:}

\ 

We want to regard $\mathcal{M}([1,\infty)\times Y',\mathfrak{s}')$
as a Hilbert submanifold of $\mathcal{B}_{k}([1,\infty)\times Y',\mathfrak{s}')$,
which will be a Hilbert manifold. Denote for simplicity $K_{Y'}=[1,\infty)\times Y'$
and define
\[
\begin{array}{c}
\mathcal{C}_{k}(K_{Y'},\mathfrak{s})=\{(A,\varPhi)\mid A-A_{0}\in L_{k}^{2}(iT^{*}K_{Y'})\;,\;\varPhi-\varPhi_{0}\in L_{k,A_{0}}^{2}(S^{+})\}\end{array}
\]
We take the gauge group to be 
\[
\mathcal{G}_{k+1}(K_{Y'})=\{u:K_{Y'}\rightarrow\mathbb{C}\mid|u|=1,\;\;u\in L_{k+1,loc}^{2}(K_{Y'})\;\;,1-u\in L_{k+1}^{2}(K_{Y'})
\]
Clearly $\mathcal{C}_{k}(K_{Y'},\mathfrak{s}')$ will be a Hilbert
manifold because of the $L_{k}^{2}$ asymptotic conditions. It is
also easy to see that $\mathcal{G}_{k+1}(K_{Y'})$ will be a Hilbert
Lie group. Therefore, to show that 
\[
\mathcal{B}_{k}(K_{Y'},\mathfrak{s}')=\mathcal{C}_{k}(K_{Y'},\mathfrak{s}')/\mathcal{G}_{k+1}(K_{Y'})
\]
is a Hilbert manifold we can use Lemma 9.3.2 in \cite{MR2388043}
which we quote for convenience:
\begin{lem}
\label{lem: slices-1}Suppose we have a Hilbert Lie group $G$ acting
smoothly and freely on a Hilbert manifold $C$ with Hausdorff quotient.
Suppose that at each $c\in C$, the map $d_{0}:T_{e}G\rightarrow T_{c}C$
(obtained from the derivative of the action) has closed range. Then
the quotient $C/G$ is also a Hilbert manifold.
\end{lem}

The Hilbert manifold structure is given as follows. If $S\subset C$
is any locally closed submanifold containing $c$, satisfying 
\[
T_{c}C=\text{im}(d_{0})\oplus T_{c}S
\]
 then the restriction of the quotient map $S\rightarrow C/G$ is a
diffeomorphism from a neighborhood of $c$ in $S$ to a neighborhood
of $Gc$ in $C/G$. Therefore, we need to verify first that $\mathcal{B}_{k}(K_{Y'},\mathfrak{s})$
is a Hausdorff space which is the content of the next lemma.
\begin{lem}
The quotient space $\mathcal{B}_{k}(K_{Y'},\mathfrak{s}')$ is Hausdorff. 
\end{lem}

\begin{proof}
We suppose that we have two gauge equivalent sequences $\gamma_{n}=(A_{n},\varPhi_{n})$
and $\tilde{\gamma}_{n}=(\tilde{A}_{n},\tilde{\varPhi}_{n})$ and
want to show that the limits $\gamma=(A_{\infty},\varPhi_{\infty})$
and $\tilde{\gamma}=(\tilde{A}_{\infty},\tilde{\varPhi}_{\infty})$
they converge to are gauge equivalent as well. By an exhaustion argument
and an application of Proposition 9.3.1 in \cite{MR2388043} we can
find a gauge transformation $u_{\infty}$ which establishes the gauge
equivalence, i.e $u_{\infty}\cdot\gamma=\tilde{\gamma}$. A priori
we only have $u_{\infty}\in L_{k+1,loc}^{2}(K_{Y'})$ and to show
that $1-u_{\infty}\in L_{k+1}^{2}(K_{Y'})$ we can now apply condition
$ii)$ in Lemma (\ref{lem: Hilbert Group}) of our paper.
\end{proof}
As is usually the case for Seiberg-Witten or Yang-Mills moduli spaces,
we do not want any random slice to the gauge group action. Rather,
we want to use the so-called Coulomb-Neumann slice \cite[Section 9.3]{MR2388043}.
A tangent vector to $\mathcal{C}_{k}(K_{Y'},\mathfrak{s}')$ at $\gamma=(A,\varPhi)$
can be written as 
\[
(a,\varPsi)\in L_{k}^{2}(K_{Y'},iT^{*}K_{Y'})\oplus L_{k,A_{0}}^{2}(K_{Y'},S^{+})
\]

while the derivative of the gauge group action is 
\begin{equation}
\begin{array}{c}
\mathbf{d}_{\gamma}:L_{k+1}^{2}(K_{Y'};i\mathbb{R})\rightarrow\mathcal{T}_{k}=L_{k}^{2}(K_{Y'},iT^{*}K_{Y'})\oplus L_{k,A_{0}}^{2}(K_{Y'},S^{+})\\
\\
\begin{array}{c}
\mathbf{d}_{(A,\varPhi)}\end{array}(\zeta)=(-d\zeta,\zeta\varPhi)
\end{array}\label{derivative gauge}
\end{equation}
We use the inner product 
\[
\left\langle (a_{1},\varPsi_{1}),(a_{2},\varPsi_{2})\right\rangle _{L^{2}}=\int\left\langle a_{1},a_{2}\right\rangle +\text{Re}\left\langle \varPsi_{1},\varPsi_{2}\right\rangle 
\]
to define the formal adjoint of $\mathbf{d}_{(A,\varPhi)}$, which
is given by \cite[Lemma 9.3.3]{MR2388043}
\begin{equation}
\mathbf{d}_{(A,\varPhi)}^{*}(a,\varPsi)=-d^{*}a+i\text{Re}\left\langle i\varPhi,\varPsi\right\rangle \label{eq:dual derivative gauge}
\end{equation}

To use Lemma (\ref{lem: slices-1}) we just need to show that $\mathbf{d}_{\gamma}$
has closed range. In order to this we will rely on Theorem 3.3 in
\cite{MR1474156} and Proposition 4.1 in \cite{Zhang[2016]}. 

First we need to define another map which will be used soon to show
that $\mathcal{M}([1,\infty)\times Y',\mathfrak{s}')$ is a Hilbert
manifold (for the proof of lemma \ref{Hilbert manifold}). The linearization
of the unperturbed Seiberg-Witten map is \cite[Eq. 8]{MR1474156}
\begin{align*}
\mathcal{D}_{(A,\varPhi)}\mathfrak{F}:L_{k}^{2}(K_{Y'},iT^{*}K_{Y'})\oplus L_{k,A_{0}}^{2}(K_{Y'},S^{+})\mapsto L_{k-1}^{2}(K_{Y'},i\mathfrak{su}(S^{+}))\oplus L_{k-1,A_{0}}^{2}(K_{Y'},S^{-})\\
(a,\varPsi)\mapsto\left(\rho(d^{+}a)-\{\varPhi\varPsi^{*}+\varPsi\varPhi^{*}\}_{0},D_{A}\varPsi+\rho(a)\varPhi\right)
\end{align*}
where
\[
\{\varPhi\varPsi^{*}+\varPsi\varPhi^{*}\}_{0}=\varPhi\varPsi^{*}+\varPsi\varPhi^{*}-\frac{1}{2}\left\langle \varPhi,\varPsi\right\rangle -\frac{1}{2}\left\langle \varPsi,\varPhi\right\rangle =\varPhi\varPsi^{*}+\varPsi\varPhi^{*}-\text{Re}\left\langle \varPhi,\varPsi\right\rangle 
\]
Define the elliptic operator (in \cite{Zhang[2016],MR2199446,MR1474156}
this is the operator $\mathcal{D}$)
\begin{align}
Q_{(A,\varPhi)}=\mathcal{D}_{(A,\varPhi)}\mathfrak{F}\oplus\mathbf{d}_{(A,\varPhi)}^{*}\label{derivative SW gauge-1}\\
(a,\varPsi)\rightarrow(\rho(d^{+}a)-\{\varPhi\varPsi^{*}+\varPsi\varPhi^{*}\}_{0},D_{A}\varPsi+\rho(a)\varPhi,-d^{*}a+i\text{Re}\left\langle i\varPhi,\varPsi\right\rangle )\nonumber 
\end{align}
We also want a formula for the formal adjoint: $Q_{(A,\varPhi)}^{*}$:
this is essentially eq. 24.10 in \cite{MR2388043}. Modulo notational
differences, we obtain
\begin{equation}
Q_{(A,\varPhi)}^{*}(\eta,\psi,\vartheta)=\left((d^{+})^{*}\rho^{*}\eta+\rho^{*}(\psi\varPhi^{*})-d\vartheta,D_{A}^{*}\psi-\eta\varPhi+\vartheta\varPhi\right)\label{dual derivative SW}
\end{equation}
In particular, taking $\eta=0$ and $\psi=0$ one obtains: 
\begin{equation}
Q_{(A,\varPhi)}^{*}(0,0,\vartheta)=(-d\vartheta,\vartheta\varPhi)=\mathbf{d}_{(A,\varPhi)}(\vartheta)\label{derivative gauge and Q}
\end{equation}
Now we are finally ready to prove that $\mathcal{B}_{k}(K_{Y'},\mathfrak{s})$
is Hausdorff. 
\begin{lem}
Define at a configuration $\gamma=(A,\varPhi)$ the subspaces
\[
\begin{array}{c}
\mathcal{K}_{k,\gamma}=\{(a,\varPsi)\mid\mathbf{d}_{(A,\varPhi)}^{*}(a,\varPsi)=0,\;\;\;\left\langle a\mid_{\partial K_{Y'}},n\right\rangle =0\;\;\;\text{at }\partial K_{Y'}\}\\
\\
\mathcal{J}_{k,\gamma}=\text{im}\;\;\mathbf{d}_{\gamma}
\end{array}
\]
As $\gamma$ varies over $\mathcal{C}_{k}(K_{Y'},\mathfrak{s})$,
the subspaces $\mathcal{J}_{k,\gamma}$ and $\mathcal{K}_{k,\gamma}$
define complementary closed sub-bundles of $\mathcal{T}_{k,\gamma}$
and we have a smooth decomposition
\begin{equation}
T\mathcal{C}_{k}(K_{Y'},\mathfrak{s})=\mathcal{J}_{k}\oplus\mathcal{K}_{k}\label{smooth decomposition}
\end{equation}
\end{lem}

\begin{proof}

In order to show the smooth decomposition $T\mathcal{C}_{k}(K_{Y'},\mathfrak{s})=\mathcal{J}_{k}\oplus\mathcal{K}_{k}$
we can follow again the proof of Proposition 9.3.4 in \cite{MR2388043}
and reduce this to the invertibility of the ``Laplacian'' 
\begin{equation}
L_{k+1}^{2}(K_{Y'},iT^{*}K_{Y'})\ni\vartheta\mapsto\triangle\vartheta+|\varPhi|^{2}\vartheta\in L_{k-1}^{2}(K_{Y'},iT^{*}K_{Y'})\label{laplacian coulomb-1}
\end{equation}
This property can be proved using a parametrix argument (which is
essentially the same as Lemma \ref{Lemma parametrix} and Theorem
\ref{thm invertibility triangle-1} in this paper): choosing a compact
subset large enough for which $|\varPhi|^{2}$ is not identically
zero, one knows from Proposition 9.3.4 in \cite{MR2388043} that the
operator (\ref{laplacian coulomb-1}) is invertible. On the other
hand, Lemma 2.3.2 in \cite{MR2199446} says that on any four manifold
with conical end (like the manifold $X^{+}$ we just used), the operator
\ref{laplacian coulomb-1} is invertible. Notice that their lemma
requires a solution to the Seiberg Witten equations but this is only
because this section was trying to find uniform bounds (independent
of the solution used). At this stage this is not our concern so the
proof they give near the end of that section can be adapted to any
configuration. Therefore, we can splice these two inverses to get
an approximate inverse to \ref{laplacian coulomb-1} on our domain
of interest $K_{Y'}$. By choosing appropriate cutoff functions one
can then guarantee that \ref{laplacian coulomb-1} will be invertible
(again, the proof of Theorem \ref{thm invertibility triangle-1} provides
more details).

Finally, notice that $\mathcal{J}_{k}$ is a closed subspace since
it is the kernel of a continuous projection $T\mathcal{C}_{k}\rightarrow\mathcal{J}_{k}$.
\end{proof}

Continuing with our analysis of our moduli space, we will now show:
\begin{lem}
\label{Hilbert manifold}$\mathcal{M}(K_{Y'},\mathfrak{s}')$ is a
Hilbert submanifold of $\mathcal{B}_{k}(K_{Y'},\mathfrak{s}')$.
\end{lem}

\begin{proof}
We seek for an analogue of proposition 24.3.1 in \cite{MR2388043}.
The main point in that proof was to show that the operator $Q_{(A,\varPhi)}$
introduced before in (\ref{derivative gauge and Q}) is surjective.

To show surjectivity, the idea in the book was to apply Corollary
17.1.5 in \cite{MR2388043}. We will not use directly the corollary
but rather its proof.

a) First one needs to check that $Q_{(A,\varPhi)}:L_{k}^{2}(K_{Y'};iT^{*}K_{Y'}\oplus S^{+})\rightarrow L_{k-1}^{2}(K_{Y'};i\mathfrak{su}(S^{+})\oplus S^{-})$
has closed range. The proof in Corollary 17.1.5 for the closed range
property was based on Theorem 17.1.3 in \cite{MR2388043}, which showed
that on a compact manifold with boundary, the operator 
\begin{equation}
Q_{(A,\varPhi)}\oplus(\pi_{0}\circ r):L_{k}^{2}(X;iT^{*}K_{Y'}\oplus S^{+})\rightarrow L_{k-1}^{2}(X;i\mathfrak{su}(S^{+})\oplus S^{-})\label{Fredholm APS}
\end{equation}
 is Fredholm, where $X$ is denotes the compact manifold with boundary
and $\pi_{0}\circ r$ denotes Atiyah-Patodi-Singer spectral conditions.
If we can show that (\ref{Fredholm APS}) continues to be Fredholm
when $X$ is replaced with $K_{Y'}$ we would thus be done. But this
just follows from a parametrix argument, since we already know that
on a manifold without boundary and with a symplectic end the operator
$Q_{(A,\varPhi)}$ is Fredholm \cite[Theorem 3.3]{MR1474156}, while
on a compact manifold with boundary $Q_{(A,\varPhi)}\oplus(\pi_{0}\circ r)$
is Fredholm. Again, this type of parametrix argument can be found
in our proof of Lemma (\ref{Lemma patching}) for example.

b) The next step is to show that $Q_{(A,\varPhi)}^{*}$ has the property
that every non-zero solution of $Q_{(A,\varPhi)}^{*}v=0$ for $v=(\eta,\psi,\vartheta)$
has non-zero restriction to the boundary $\partial K_{Y'}$.

Using the equation (\ref{derivative gauge and Q}) for the adjoint
$Q_{(A,\varPhi)}^{*}$, we can see that the equation $Q_{(A,\varPhi)}^{*}(\eta,\psi,\vartheta)=0$
becomes in the coordinates $(a,\varPsi)$ of $Q_{(A,\varPhi)}$ (compare
with eq 24.10 in \cite{MR2388043})
\begin{align}
(d^{+})^{*}\rho^{*}\eta+\rho^{*}(\psi\varPhi^{*})-d\vartheta=0\nonumber \\
D_{A}^{*}\psi-\eta\varPhi+\vartheta\varPhi=0\label{calculation Q*}
\end{align}
As in eq. (24.15) of \cite{MR2388043}, the equations (\ref{calculation Q*})
have the shape
\[
\frac{d}{dt}v+(L_{0}+h(t))v=0
\]
where $L_{0}$ is a self-adjoint elliptic operator on $Y'$ and $h$
is a time dependent operator on $Y'$ satisfying the conditions of
the unique continuation lemma. Since $v$ vanishes on the boundary,
it vanishes on the collar too and therefore on the cone $K_{Y'}$.

With $a)$ and $b)$ the surjectivity of $Q_{(A,\varPhi)}$ follows
so the proof that the moduli space $\mathcal{M}([1,\infty)\times Y',\mathfrak{s}')$
is a Hilbert sub-manifold of $\mathcal{B}_{k}(K_{Y'},\mathfrak{s}')$
follows from our initial remarks.
\end{proof}
\begin{rem}
Notice that as in the other cases we have a restriction map
\[
R_{K}:\mathcal{M}([1,\infty)\times Y',\mathfrak{s}')\rightarrow\mathcal{B}_{k-1/2}^{\sigma}(-Y',\mathfrak{s}_{\xi'})
\]
\end{rem}

\subsection{4.2 Gluing the Moduli Spaces}

\ 

Now that we know that each moduli space appearing in the fiber product
description (\ref{eq:Fiber Product}) is a Hilbert manifold, we need
to show that their fiber product is a finite dimensional manifold,
possibly with components of different dimensions. As mentioned before,
we have the restrictions maps
\begin{align*}
R_{\tau}:\mathcal{M}^{\tau}(\mathbb{R}^{+}\times-Y,\mathfrak{s}_{\xi},[\mathfrak{c}])\rightarrow\mathcal{B}_{k-1/2}^{\sigma}(Y,\mathfrak{s}_{\xi})\\
R_{W}^{-}:\mathcal{M}(W^{\dagger},\mathfrak{s}_{\omega})\rightarrow\mathcal{B}_{k-1/2}^{\sigma}(-Y,\mathfrak{s}_{\xi})\\
R_{W}^{+}:\mathcal{M}(W^{\dagger},\mathfrak{s}_{\omega})\rightarrow\mathcal{B}_{k-1/2}^{\sigma}(Y',\mathfrak{s}_{\xi'})\\
R_{K}:\mathcal{M}([1,\infty)\times Y',\mathfrak{s}')\rightarrow\mathcal{B}_{k-1/2}^{\sigma}(-Y',\mathfrak{s}_{\xi'})
\end{align*}
If we write as before an element $[\gamma]\in\mathcal{M}(W_{\xi',Y}^{+},\mathfrak{s}_{\omega},[\mathfrak{c}])$
as 
\[
\text{Fib}(R_{\tau},R_{W}^{-},R_{W}^{+},R_{K})\ni[\gamma]=([\gamma_{\mathbb{R}^{+}\times-Y}],[\gamma_{W}],[\gamma_{[1,\infty)\times Y'}])
\]
and define 
\[
\begin{cases}
\mathcal{B}_{k-1/2}^{\sigma}(-Y,\mathfrak{s}_{\xi})\ni\mathfrak{b}=R_{W}^{-}\left(\gamma_{W}\right)\\
\mathcal{B}_{k-1/2}^{\sigma}(Y',\mathfrak{s}_{\xi'})\ni\mathfrak{b}'=R_{W}^{+}(\gamma_{W})
\end{cases}
\]
then the derivatives of our restriction maps can be written as 
\begin{align*}
\mathcal{D}R_{[\gamma_{\mathbb{R}^{+}\times-Y]}}^{\tau}:T_{[\gamma_{\mathbb{R}^{+}\times-Y]}}\mathcal{M}^{\tau}(\mathbb{R}^{+}\times-Y,\mathfrak{s}_{\xi},[\mathfrak{c}])\rightarrow\mathcal{K}_{k-1/2,\mathfrak{b}}^{\sigma}(Y,\mathfrak{s}_{\xi})\\
\mathcal{D}R_{W,[\gamma_{W}]}^{-}:T_{[\gamma_{W}]}\mathcal{M}(W^{\dagger},\mathfrak{s}_{\omega})\rightarrow\mathcal{K}_{k-1/2,\mathfrak{b}}^{\sigma}(-Y,\mathfrak{s}_{\xi})\\
\mathcal{D}R_{W,[\gamma_{W}]}^{+}:T_{[\gamma_{W}]}\mathcal{M}(W^{\dagger},\mathfrak{s}_{\omega})\rightarrow\mathcal{K}_{k-1/2,\mathfrak{b}'}^{\sigma}(Y',\mathfrak{s}_{\xi'})\\
\mathcal{D}R_{K,[\gamma_{[1,\infty)\times Y'}]}:T\mathcal{M}_{[\gamma_{[1,\infty)\times Y'}]}([1,\infty)\times Y',\mathfrak{s})\rightarrow\mathcal{K}_{k-1/2,\mathfrak{b}'}^{\sigma}(-Y',\mathfrak{s}_{\xi'})
\end{align*}
where the right hand side is the corresponding Couloumb slice at each
configuration $\mathfrak{b},\mathfrak{b}'$. The next definition is
the analogue of definition 24.4.2 in \cite{MR2388043}:
\begin{defn}
\label{def:regularity}Let $[\gamma]\in\mathcal{M}(W_{\xi',Y}^{+},\mathfrak{s}_{\omega},[\mathfrak{c}])$
and 
\[
\rho:\mathcal{M}(W_{\xi',Y}^{+},\mathfrak{s}_{\omega},[\mathfrak{c}])\rightarrow\mathcal{M}^{\tau}(\mathbb{R}^{+}\times-Y,\mathfrak{s}_{\xi},[\mathfrak{c}])\times\mathcal{M}(W^{\dagger},\mathfrak{s}_{\omega})\times\mathcal{M}([1,\infty)\times Y',\mathfrak{s}')
\]
 the restriction map. Write 
\[
\rho([\gamma])=([\gamma_{1}],[\gamma_{2}],[\gamma_{3}])=\left([\gamma_{\mathbb{R}^{+}\times-Y}],[\gamma_{W}],[\gamma_{[1,\infty)\times Y'}]\right)\in\text{Fib}(R_{\tau},R_{W}^{-},R_{W}^{+},R_{K})
\]
 and 
\begin{align*}
[\mathfrak{b}]=R_{\tau}([\gamma_{\mathbb{R}^{+}\times-Y}])=R_{W}^{-}\left([\gamma_{W}]\right)\in\mathcal{B}_{k-1/2}^{\sigma}(-Y,\mathfrak{s}_{\xi})\\{}
[\mathfrak{b}']=R_{W}^{+}([\gamma_{W}])=R_{K}([\gamma_{[1,\infty)\times Y'}])\in\mathcal{B}_{k-1/2}^{\sigma}(Y',\mathfrak{s}_{\xi'})
\end{align*}
We say that the moduli space $\mathcal{M}(W_{\xi',Y}^{+},\mathfrak{s}_{\omega},[\mathfrak{c}])$
is \textbf{regular} at $[\gamma]$ if the map
\[
R=\left(\left(R_{\tau},R_{W}^{-}\right),\left(R_{W}^{+},R_{K}\right)\right):\text{Fib}(R_{\tau},R_{W}^{-},R_{W}^{+},R_{K})\rightarrow\mathcal{B}_{k-1/2}^{\sigma}(-Y,\mathfrak{s}_{\xi})\times\mathcal{B}_{k-1/2}^{\sigma}(Y',\mathfrak{s}_{\xi'})
\]
is transverse at $\rho[\gamma]$. That is, $(R_{\tau},R_{W}^{-})$
is transverse at $[\mathfrak{b}]$ while $\left(R_{W}^{+},R_{K}\right)$
is transverse at $[\mathfrak{b}']$.
\end{defn}

Following the strategy in section 24.4 of \cite{MR2388043}, to show
regularity what we really need is an analogue of Lemma 24.4.1 (which
is our next lemma). The other pieces used by \cite{MR2388043} do
not change so we can conclude the following transversality result
(compare with Proposition 24.4.7 \cite{MR2388043}):
\begin{thm}
\label{Theorem perturbations}Let $\mathfrak{q}_{-Y},\mathfrak{q}_{Y'}$
be fixed perturbations for $-Y,Y'$ respectively such that for all
critical points $[\mathfrak{a}],[\mathfrak{b}]\in\mathcal{B}_{k-1/2}^{\sigma}(-Y,\mathfrak{s}_{\xi})$
and $[\mathfrak{a}'],[\mathfrak{b}']\in\mathcal{B}_{k-1/2}^{\sigma}(Y',\mathfrak{s}_{\xi})$,
the moduli spaces of trajectories $\mathcal{M}([\mathfrak{a}],\mathbb{R}\times-Y,\mathfrak{s}_{\xi},[\mathfrak{b}])$
and $\mathcal{M}([\mathfrak{a}'],\mathbb{R}\times Y',\mathfrak{s}_{\xi'},[\mathfrak{b}'])$
are cut out transversely. Then there is a residual subset $\mathcal{P}_{0}$
of the large space of perturbations $\mathcal{P}(-Y,\mathfrak{s}_{\xi})\times\mathcal{P}(Y',\mathfrak{s}_{\xi'})$
defined in section 11.6 of \cite{MR2388043} for which the following
holds: if for any $(\mathfrak{p}_{0},\mathfrak{p}_{0}')\in\mathcal{P}_{0}\subset\mathcal{P}(-Y,\mathfrak{s}_{\xi})\times\mathcal{P}(Y',\mathfrak{s}_{\xi'})$
one forms perturbation 
\[
\mathfrak{p}_{W_{\xi',Y}^{+}}=-\hat{\mathfrak{q}}_{Y,g_{\theta},\mathfrak{s}_{\xi}}+\left(\beta\hat{\mathfrak{q}}_{Y,g_{\theta},\mathfrak{s}_{\xi}}+\beta_{0}'\hat{\mathfrak{p}}_{0}\right)+\left(\beta_{0}'\hat{\mathfrak{p}}_{0}'+\beta'\hat{\mathfrak{q}}_{Y',g_{\theta'},\mathfrak{s}_{\xi'}}\right)+\left(\beta_{N_{K}}\hat{\mathfrak{q}}_{Y',g_{\theta'},\mathfrak{s}_{\xi'}}+\beta_{K}\mathfrak{p}_{K}\right)
\]
described in equation (\ref{eq:glued perturbation}) , then the moduli
space $\mathcal{M}(W_{\xi',Y}^{+},\mathfrak{s}_{\omega},[\mathfrak{c}],\mathfrak{p}_{W_{\xi',Y}^{+}})$
defined using the perturbation $\mathfrak{p}_{W_{\xi',Y}^{+}}$ is
regular, in other words, we have transversality at $\rho[\gamma]$
for all $[\gamma]\in\mathcal{M}(W_{\xi',Y}^{+},\mathfrak{s}_{\omega},[\mathfrak{c}],\mathfrak{p}_{W_{\xi',Y}^{+}})$.

In particular, for any perturbation belonging to this residual set,
the moduli space 
\[
\mathcal{M}(W_{\xi',Y}^{+},\mathfrak{s}_{\omega},[\mathfrak{c}],\mathfrak{p}_{W_{\xi',Y}^{+}})
\]
 will be a manifold whose components have dimensions equal to $\text{ind}\mathcal{D}_{\rho[\gamma]}R=\dim\ker\mathcal{D}_{\rho[\gamma]}R$. 
\end{thm}

Again, the proof of this theorem is a consequence of the following
lemma:
\begin{lem}
\label{Lemma parametrix}Let $[\gamma]\in\mathcal{M}(W_{\xi',Y}^{+},\mathfrak{s}_{\omega},[\mathfrak{c}])$
. Then the sum of the derivatives 
\begin{align*}
\mathcal{D}_{\rho[\gamma]}R=\left(\mathcal{D}_{[\gamma_{1}]}R_{\tau}+\mathcal{D}_{[\gamma_{2}]}R_{W}^{-}\right)\oplus\left(\mathcal{D}_{[\gamma_{2}]}R_{W}^{+}+\mathcal{D}_{[\gamma_{3}]}R_{K}\right):\\
T_{[\gamma_{1}]}\mathcal{M}^{\tau}(\mathbb{R}^{+}\times-Y,\mathfrak{s}_{\xi},[\mathfrak{c}])\oplus T_{[\gamma_{2}]}\mathcal{M}(W^{\dagger},\mathfrak{s}_{\omega})\oplus T_{[\gamma_{3}]}\mathcal{M}([1,\infty)\times Y',\mathfrak{s})\rightarrow\mathcal{K}_{k-1/2,\mathfrak{b}}^{\sigma}(-Y)\oplus\mathcal{K}_{k-1/2,\mathfrak{b}'}(Y')
\end{align*}
is a Fredholm map.
\end{lem}

\begin{proof}
We will begin showing that the following maps are Fredholm and compact:
\begin{equation}
\begin{array}{ccccccccc}
\pi_{\mathfrak{b}}\circ\mathcal{D}_{[\gamma_{1}]}R_{\tau} &  & \text{is compact} & (1) &  & (1-\pi_{\mathfrak{b}})\circ\mathcal{D}_{[\gamma_{2}]}R_{W}^{-} & \text{is compact} & (5)\\
(1-\pi_{\mathfrak{b}})\circ\mathcal{D}_{[\gamma_{1}]}R_{\tau} &  & \text{is Fredholm} & (2) &  & \pi_{\mathfrak{b}}\circ\mathcal{D}_{[\gamma_{2}]}R_{W}^{-} & \text{is Fredholm} & (6)\\
\\
(1-\pi_{\mathfrak{b}'})\circ\mathcal{D}_{[\gamma_{2}]}R_{W}^{+} &  & \text{is compact} & (3) &  & \pi_{\mathfrak{b}'}\circ\mathcal{D}_{[\gamma_{3}]}R_{K} & \text{is compact} & (7)\\
\pi_{\mathfrak{b}'}\circ\mathcal{D}_{[\gamma_{2}]}R_{W}^{+} &  & \text{is Fredholm} & (4) &  & (1-\pi_{\mathfrak{b}'})\circ\mathcal{D}_{[\gamma_{3}]}R_{K} & \text{is Fredholm} & (8)
\end{array}\label{eight}
\end{equation}
Here $\pi_{\mathfrak{b}},\pi_{\mathfrak{b}'}$ are defined as follows
\cite[Sections 12.4, 17.3]{MR2388043}. We have a Hessian operator
$\text{Hess}_{\mathfrak{q}}^{\sigma}:\mathcal{K}_{k}^{\sigma}\rightarrow\mathcal{K}_{k-1}^{\sigma}$
obtained by projecting $\mathcal{D}(\text{grad}\cancel{\mathcal{L}})^{\sigma}$
onto the subspace $\mathcal{K}_{k-1}^{\sigma}$. The spectrum of $\text{Hess}_{\mathfrak{q}}^{\sigma}$
is real, discrete and with finite dimensional generalized eigenspaces.
If the operator is hyperbolic (that is, zero is not an eigenvalue)
we have a spectral decomposition 
\[
\mathcal{K}_{k-1/2,\mathfrak{b}}^{\sigma}=\mathcal{K}_{\mathfrak{b}}^{+}\oplus\mathcal{K}_{\mathfrak{b}}^{-}
\]
where $\mathcal{K}_{\mathfrak{b}}^{+}$ is the closure of the span
of the positive eigenspaces and $\mathcal{K}_{\mathfrak{b}}^{-}$
of the negative eigenspaces. In the non-hyperbolic case, we choose
$\epsilon$ sufficiently small that there are no eigenvalues in $(0,\epsilon)$
and then define $\mathcal{K}_{k-1/2,\mathfrak{b}}^{\pm}$ using the
spectral decomposition of the operator $\text{Hess}_{\mathfrak{q},\mathfrak{b}}^{\sigma}-\epsilon$.
The effect is that the generalized $0$ eigenspace belongs to $\mathcal{K}_{\mathfrak{b}}^{-}$. 

Also, notice that the roles of the different operators are sometimes
opposite because of the different orientations on the manifolds, namely
\begin{align*}
\mathcal{K}_{\mathfrak{b}}^{-}(-Y)=\mathcal{K}_{\mathfrak{b}}^{+}(Y)\\
\mathcal{K}_{\mathfrak{b}'}^{-}(-Y')=\mathcal{K}_{\mathfrak{b}'}^{+}(Y')
\end{align*}
We will now break the proof this lemma and the verification of the
previous eight identities into three mini lemmas.
\end{proof}
\begin{lem}
The identities $(1),(2),(3),(4),(5),(6)$ stated in Lemma (\ref{Lemma parametrix})
hold.
\end{lem}

\begin{proof}
By Proposition 24.3.2 in \cite{MR2388043}, $(3),(4),(5),(6)$ are
true (remember that in this section of the book the boundary is the
compact four manifold is allowed to be disconnected. In our case the
boundary is simply $-Y\cup Y'$).

By the discussion in Lemma 24.4.1 in \cite{MR2388043}, $(1)$ and
$(2)$ are true.
\end{proof}
Now we turn to verifying \textbf{$(7)$} and $(8)$. To explain what
we need to do we will chase through some theorems of \cite{MR2388043}
\cite[Proposition 2.18, Lemma 3.17]{LinJ[2016]}. It is also useful
to observe that $(7)$ will be true because the proof in \cite{MR2388043}
is essentially a local argument near the boundary.
\begin{lem}
\label{Lemma patching}The identities $(7),(8)$ stated in Lemma (\ref{Lemma parametrix})
hold.
\end{lem}

\begin{proof}
Assertions $(7)$ and $(8)$ are the ``conical'' versions of Proposition
24.3.2 in the book. The proof of this theorem in turn refers to Theorem
17.3.2, which at the same time requires Proposition 17.2.6, which
depends at the same time on Proposition 17.2.5. The latter uses essentially
Theorem 17.1.3 and the only part that is not proven explicitly is
part $a)$, which depends on a parametrix argument (modeled on Proposition
14.2.1) of Theorem 17.1.4. 

In a nutshell, we must do the following. Decompose $Q_{(A,\varPhi)}$
as 
\begin{align*}
Q_{(A,\varPhi)} & =D_{0}+K\\
D_{0}(a,\varPsi) & =\left(\rho(d^{+}a),D_{A_{0}}\varPsi,-d^{*}a\right)\\
K(a,\varPsi) & =\left(-\{\varPhi\varPsi^{*}+\varPsi\varPhi^{*}\}_{0},\rho(A-A_{0})\varPsi+\rho(a)\varPhi+i\text{Re}\left\langle i\varPhi,\varPsi\right\rangle \right)
\end{align*}
On the collar of $\partial K_{Y'}$, $D_{0}$ can be written in the
form 
\[
\frac{d}{dt}+L_{0}
\]
where $L_{0}:C^{\infty}(-Y';E_{0})\rightarrow C^{\infty}(-Y';E_{0})$
is a first order, self-adjoint elliptic differential operator. We
will not write the exact formula for the domain and codomain since
they would rather cumbersome. Rather we will denote the bundles involved
by the letter $E_{0}$ when referring to the three manifolds and by
$E$ for the four manifolds just as the book does.

If $H_{0}^{+}$ and $H_{0}^{-}$ are the closures in $L_{1/2}^{2}(Y;E_{0})$
of the spans of the eigenvectors belonging to positive and non-positive
eigenvalues of $L_{0}$ and 
\[
\varPi_{0}:L_{1/2}^{2}(Y;E_{0})\rightarrow L_{1/2}^{2}(Y;E_{0})
\]
is the projection with image $H_{0}^{-}$ and kernel $H_{0}^{+}$,
we need to show that the operator 
\[
Q_{(A,\varPhi)}\oplus(\varPi_{0}\circ r_{-Y'}):L_{k}^{2}(K_{Y'};E)\rightarrow L_{k-1}^{2}(K_{Y'};E)\oplus(H_{0}^{-}\cap L_{k-1/2}^{2})
\]
is Fredholm. First, for notational purposes take the collar neighborhood
of $\partial K_{Y'}$ to be $(-5,0]\times-Y'$, where $\partial K_{Y'}$
has now been identified with $\{0\}\times-Y'$. Also denote for simplicity
\[
Q_{K_{Y'}}=Q_{(A,\varPhi)}:L_{k}^{2}(K_{Y'};E)\rightarrow L_{k-1}^{2}(K_{Y'};E)
\]
To show the Fredholm property mentioned above we will give a parametrix
argument, which is essentially the same as the one used in Proposition
14.2.1 of \cite{MR2388043}. Namely, we modify the manifold $K_{Y'}$
in two different ways.

For the first modification we close up $K_{Y'}$ first by extending
the collar neighborhood a little bit (to the left in our figure) and
then finding a four manifold $X$ (dots on the left side of the figure)
bounding $Y'$. For the second modification, we forget about the part
of the cone $K_{Y'}$ which does not have a product structure, in
other words, we take the collar neighborhood of $K_{Y'}$ and extend
it into a half-infinite cylinder which extends indefinitely to the
right in our figure. In particular, notice that we superimposed both
modifications in our image to save some space but they do not interact
with each other. Each modification provides a parametrix as follows.

\begin{figure}[H]
\begin{centering}
\includegraphics[scale=0.65]{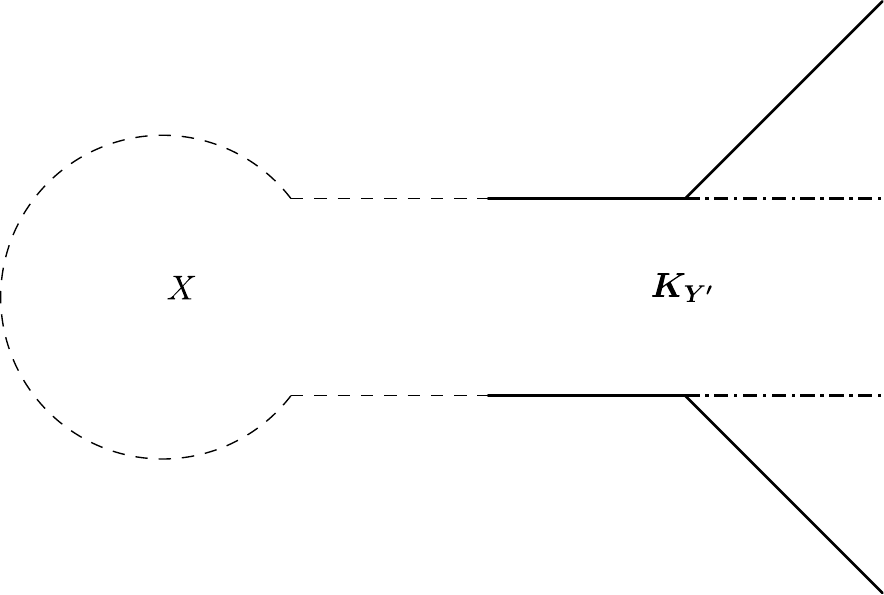}
\par\end{centering}
\caption{Closing up the cone $K_{Y'}$ into the manifold $X\cup K_{Y'}$. Simultaneously,
we extend the product neighborhood $(-5,0]\times-Y'$ of $K_{Y'}$
into a half-infinite cylinder $Z=(-\infty,0]\times-Y'$. }
\end{figure}

Regarding the first modification, we can define the manifold $X^{+}=X\cup\text{cylinder}\cup K_{Y'}$
and extend $Q_{K_{Y}'}$ to an operator 
\[
Q_{X^{+}}:L_{k}^{2}(X^{+};E)\rightarrow L_{k-1}^{2}(X^{+};E)
\]
and by Theorem 3.3 in \cite{MR1474156} there is a parametrix (that
is, $Q_{X^{+}}P_{X^{+}}-I$ and $P_{X^{+}}Q_{X^{+}}-I$ are compact
operators) which we denote 
\[
P_{X^{+}}:L_{k-1}^{2}(X^{+};E)\rightarrow L_{k}^{2}(X^{+};E)
\]
Similarly, for the second modification we define the half-cylinder
$Z=(-\infty,0]\times-Y'$. By Theorem 17.1.4 in \cite{MR2388043},
the operator 
\[
Q_{Z}\oplus(\varPi_{0}\circ r_{-Y'}):L_{k}^{2}(Z;E)\rightarrow L_{k-1}^{2}(Z;E)\oplus(H_{0}^{-}\cap L_{k-1/2}^{2}(-Y';E_{0}))
\]
has a parametrix 
\[
P_{Z}:L_{k-1}^{2}(Z;E)\oplus(H_{0}^{-}\cap L_{k-1/2}^{2}(-Y';E_{0}))\rightarrow L_{k}^{2}(Z;E)
\]
Finally, to define the parametrix corresponding to $Q_{K_{Y'}}\oplus(\varPi_{0}\circ r_{-Y'})$
, let $1=\eta_{1}+\eta_{2}$ be a partition of unity subordinate to
a covering of $K_{Y'}$ by the open sets $U_{1}=K_{Y'}\backslash([-2,0]\times-Y')$
and $U_{2}=(-3,0]\times-Y'$. Let $\gamma_{1}$ be a function which
is $1$ on the support of $\eta_{1}$ and vanishes on $(-1,0]\times-Y'$.
Similarly, let $\gamma_{2}$ be $1$ on the support of $\eta_{2}$
and vanishing outside $[-4,0]\times Y'$. Define 
\begin{align*}
P_{K_{Y'}}:L_{k-1}^{2}(K_{Y'};E)\oplus(H_{0}^{-}\cap L_{k-1/2}^{2}) & \mapsto & L_{k}^{2}(K_{Y'};E)\\
e & \mapsto & \gamma_{1}P_{X^{+}}(\eta_{1}e)+\gamma_{2}P_{Z}(\eta_{2}e)
\end{align*}
Notice that thanks to how the supports of the functions where chosen,
the function is actually well defined. A similar computation to Proposition
14.2.1 in \cite{MR2388043} shows that $P_{K_{Y'}}$ is a parametrix
for $Q_{K_{Y'}}\oplus(\varPi_{0}\circ r_{-Y'})$. 
\end{proof}
Now we can finish the proof of Lemma (\ref{Lemma parametrix}).
\begin{lem}
$\mathcal{D}_{\rho[\gamma]}R$ is a Fredholm map.
\end{lem}

\begin{proof}
Thanks to the eight identities (\ref{eight}) we can see that 
\[
\begin{array}{cc}
 & \mathcal{D}_{[\gamma_{1}]}R_{\tau}+\mathcal{D}_{[\gamma_{2}]}R_{W}^{-}\\
\\
= & \underbrace{(1-\pi_{\mathfrak{b}})\circ\mathcal{D}_{[\gamma_{1}]}R_{\tau}}_{\text{Fredholm}}+\underbrace{\pi_{\mathfrak{b}}\circ\mathcal{D}_{[\gamma_{2}]}R_{W}^{-}}_{\text{Fredholm}}+\underbrace{\pi_{\mathfrak{b}}\circ\mathcal{D}_{[\gamma_{1}]}R_{\tau}}_{\text{compact}}+\underbrace{(1-\pi_{\mathfrak{b}})\circ\mathcal{D}_{[\gamma_{2}]}R_{W}^{-}}_{\text{compact}}
\end{array}
\]
Likewise,
\[
\begin{array}{cc}
 & \mathcal{D}_{[\gamma_{3}]}R_{K}+\mathcal{D}_{[\gamma_{2}]}R_{W}^{+}\\
\\
= & \underbrace{(1-\pi_{\mathfrak{b}'})\circ\mathcal{D}_{[\gamma_{3}]}R_{K}}_{\text{Fredholm}}+\underbrace{\pi_{\mathfrak{b}'}\circ\mathcal{D}_{[\gamma_{2}]}R_{W}^{+}}_{\text{Fredholm}}+\underbrace{\pi_{\mathfrak{b}'}\circ\mathcal{D}_{[\gamma_{3}]}R_{K}}_{\text{compact}}+\underbrace{(1-\pi_{\mathfrak{b}'})\circ\mathcal{D}_{[\gamma_{2}]}R_{W}^{+}}_{\text{compact}}
\end{array}
\]
Therefore, 
\[
\left(\mathcal{D}_{[\gamma_{1}]}R_{\tau}+\mathcal{D}_{[\gamma_{2}]}R_{W}^{-}\right)\oplus\left(\mathcal{D}_{[\gamma_{2}]}R_{W}^{+}+\mathcal{D}_{[\gamma_{3}]}R_{K}\right)
\]
 differs by the compact operator 
\[
\left(\pi_{\mathfrak{b}}\circ\mathcal{D}_{[\gamma_{1}]}R_{\tau}+(1-\pi_{\mathfrak{b}})\circ\mathcal{D}_{[\gamma_{2}]}R_{W}^{-}\right)\oplus\left(\pi_{\mathfrak{b}'}\circ\mathcal{D}_{[\gamma_{3}]}R_{K}+(1-\pi_{\mathfrak{b}'})\circ\mathcal{D}_{[\gamma_{2}]}R_{W}^{+}\right)
\]
 from the direct sum of the Fredholm operators 
\[
\left((1-\pi_{\mathfrak{b}})\circ\mathcal{D}_{[\gamma_{1}]}R_{\tau}\oplus\pi_{\mathfrak{b}}\circ\mathcal{D}_{[\gamma_{2}]}R_{W}^{-}\right)\oplus\left((1-\pi_{\mathfrak{b}'})\circ\mathcal{D}_{[\gamma_{3}]}R_{K}\oplus\pi_{\mathfrak{b}'}\circ\mathcal{D}_{[\gamma_{2}]}R_{W}^{+}\right)
\]
and so the result follows.
\end{proof}

\section{5. Stretching the Neck}

As promised when we explained our strategy for proving naturality,
we will consider a parametrized moduli space following the ideas used
in sections 4.9, 4.10, 6.3 of \cite{MR2299739} and sections 24.6,
26.1 and 27.4 of \cite{MR2388043}. Thanks to the computations done
in sections 5.5 and 6 of \cite{Zhang[2016]}, formally our situation
$\text{cylinder+compact+cone}$ behaves in the same way as if we were
working in the context of $\text{cylinder+compact}$, which is where
the theorems just mentioned strictly speaking apply. 

Recall that we want to show that $\widecheck{HM}_{\bullet}(W^{\dagger},\mathfrak{s}_{\omega})\mathbf{c}(\xi')=\mathbf{c}(\xi',Y)$,
in other words, at the chain-level we must have
\[
\check{m}c(\xi')-c(\xi',Y)\in\text{im}\check{\partial}_{-Y}
\]
The strategy we spelled out consisted in attaching a cylinder of length
$L$ to $W_{\xi',Y}^{+}$ and studying the Seiberg-Witten equations
on 
\[
W_{\xi',Y}^{+}(L)=([1,\infty)\times Y')\cup([0,L]\times-Y')\cup W^{\dagger}\cup(\mathbb{R}^{+}\times-Y)
\]
Equivalently, as explained in section 24.6 of \cite{MR2388043}, we
can consider a family of metrics $g_{L}$ and perturbations on $W^{\dagger}$,
all of which are equal near $Y'$. For example, we can choose a fraction
of the collar neighborhood near $Y'$ and instead of using the product
metric $dt\otimes dt+g_{Y'}$, we use a smoothed out version of the
metric $g_{L}=L^{2}dt\otimes dt+g_{Y'}$, which agrees with the old
metric outside this region. In any case, we obtain a parametrized
configuration space 
\[
\mathcal{\mathfrak{M}}_{z}(W_{\xi',Y}^{+},\mathfrak{s}_{\omega},[\mathfrak{c}])=\bigcup_{L\in[0,\infty)}\{L\}\times\mathcal{\mathcal{M}}_{z}(W_{\xi',Y}^{+}(L),\mathfrak{s}_{\omega},[\mathfrak{c}])
\]
 which we can identify with a subset of $[0,\infty)\times\mathcal{B}_{k,loc}^{\sigma}(W_{\xi',Y}^{+},\mathfrak{s}_{\omega})$
as follows (see the remark before definition 24.4.9 in \cite{MR2388043}
and section 2.3 in \cite{MR2838269}):

For any $t\in[0,\infty)$ there is a unique automorphism $b_{t}:TW_{\xi',Y}^{+}\rightarrow TW_{\xi',Y}^{+}$
that is positive, symmetric with respect to $g_{0}$ and has the property
that $g_{0}(u,v)=g_{t}(b_{t}(u),b_{t}(v))$. The map induced by $b_{t}$
on orthonormal frames gives rise to a map of spinor bundles $\bar{b}_{t}:S_{0}^{\pm}\rightarrow S_{t}^{\pm}$
associated to the metrics $g_{0}$ and $g_{t}$ . This map is an isomorphism
preserving the fiberwise length of spinors. The identification 
\[
[0,\infty)\times\mathcal{B}_{k,loc}^{\sigma}(W_{\xi',Y}^{+},\mathfrak{s}_{\omega})\rightarrow\bigcup_{L\in[0,\infty)}\{L\}\times\mathcal{B}_{k,loc}^{\sigma}(W_{\xi',Y}^{+}(L),\mathfrak{s}_{\omega})
\]
is then given by 
\begin{equation}
(L,A,\mathbb{R}^{+}\phi,\varPhi)\rightarrow(L,A,\mathbb{R}^{+}\bar{b}_{L}(\phi),\bar{b}_{L}(\varPhi))\label{identification metric}
\end{equation}
Just as in proposition 26.1.3 in \cite{MR2388043}, the moduli space
$\mathcal{\mathfrak{M}}{}_{z}(W_{\xi',Y}^{+},\mathfrak{s}_{\omega},[\mathfrak{c}])$
is a smooth manifold with boundary. The boundary is the fiber over
$L=0$, that is, the original moduli space $\mathcal{\mathcal{M}}_{z}(W_{\xi',Y}^{+},\mathfrak{s},[\mathfrak{c}])$.
Each individual moduli space $\mathcal{\mathcal{M}}{}_{z}(W_{\xi',Y}^{+}(L),\mathfrak{s}_{\omega},[\mathfrak{c}])$
can be compactified into $\mathcal{\mathcal{M}}_{z}^{+}(W_{\xi',Y}^{+}(L),\mathfrak{s}_{\omega},[\mathfrak{c}])$
by adding broken trajectories as in definition 24.6.1 of \cite{MR2388043}\footnote{More precisely, for us a broken trajectory asymptotic to $[\mathfrak{c}]$
consists of an element $[\gamma_{0}]$ in a moduli space $\mathcal{M}_{z_{0}}(W_{\xi',Y}^{+}(L),\mathfrak{s}_{\omega},[\mathfrak{c}])$
and an unparametrized broken trajectory $[\check{\boldsymbol{\gamma}}]$
in a moduli space $\check{\mathcal{M}}_{z}([\mathfrak{c}_{0}],\mathfrak{s}_{\xi},[\mathfrak{c}])$.} and to compactify
\[
\bigcup_{L\in[0,\infty)}\{L\}\times\mathcal{\mathcal{M}}_{z}^{+}(W_{\xi',Y}^{+}(L),\mathfrak{s},[\mathfrak{c}])
\]
 we add a fiber over $L=\infty$, which is denoted $\mathcal{\mathcal{M}}_{z}^{+}(W_{\xi',Y}^{+}(\infty),\mathfrak{s},[\mathfrak{c}])$
, where 
\begin{equation}
W_{\xi',Y}^{+}(\infty)=\left(K_{Y'}\cup\left[\mathbb{R}^{+}\times-Y'\right]\right)\cup\left(\left[\mathbb{R}^{-}\times-Y'\right]\cup W^{\dagger}\cup\left[\mathbb{R}^{+}\times-Y\right]\right)\label{infinite stretch}
\end{equation}
 An element in this space consists (at most) of a quadruple $([\gamma_{K'}],[\check{\boldsymbol{\gamma}}_{Y'}],[\gamma_{W^{\dagger}}],[\check{\boldsymbol{\gamma}}_{Y}])$
where :

$\bullet$ $[\gamma_{K'}]\in\mathcal{\mathcal{M}}(Z_{Y',\xi'}^{+},\mathfrak{s}',[\mathfrak{a}_{Y'}])$
is a solution on $\left[\mathbb{R}^{+}\times-Y\right]\cup K_{Y'}$.

$\bullet$ $[\check{\boldsymbol{\gamma}}_{Y'}]\in\check{\mathcal{\mathcal{M}}}^{+}([\mathfrak{a}_{Y'}],\mathfrak{s}_{\xi'},[\mathfrak{b}_{Y'}])$
is an unparametrized trajectory on the cylinder $\mathbb{R}\times-Y'$.

$\bullet$ $[\gamma_{W^{\dagger}}]\in\mathcal{\mathcal{M}}([\mathfrak{b}_{Y'}],W_{*}^{\dagger},\mathfrak{s}_{\omega},[\mathfrak{b}_{Y}])$
is a solution on $W_{*}^{\dagger}$, that is, $W^{\dagger}$ with
two cylindrical ends attached to it. 

$\bullet$ $[\check{\boldsymbol{\gamma}}_{Y}]\in\check{M}^{+}([\mathfrak{b}_{Y}],\mathfrak{s}_{\xi},[\mathfrak{c}])$
is an unparametrized trajectory on the cylinder $\mathbb{R}\times-Y$.

Just as in proposition 26.1.4 in \cite{MR2388043}, the space 

\[
\mathfrak{M}_{z}^{+}(W_{\xi',Y}^{+},\mathfrak{s},[\mathfrak{c}])=\bigcup_{L\in[0,\infty]}\{L\}\times\mathcal{\mathcal{M}}_{z}^{+}(W_{\xi',Y}^{+}(L),\mathfrak{s},[\mathfrak{c}])
\]
 is compact and when it is of dimension 1 the $0$ dimensional strata
over $L=\infty$ are of the following types (compare with proposition
26.1.6 \cite{MR2388043}):

i) $\mathcal{\mathcal{M}}_{Z_{Y',\xi'}^{+}}\times\mathcal{\mathcal{M}}_{W_{*}^{\dagger}}$

ii) $\mathcal{\mathcal{M}}_{Z_{Y',\xi'}^{+}}\times\mathcal{\mathcal{M}}_{W_{*}^{\dagger}}\times\check{\mathcal{\mathcal{M}}}_{-Y}$ 

iii) $\mathcal{\mathcal{M}}_{Z_{Y',\xi'}^{+}}\times\check{\mathcal{\mathcal{M}}}_{-Y'}\times\mathcal{\mathcal{M}}_{W_{*}^{\dagger}}$

Here $\check{\mathcal{\mathcal{M}}}$ denotes an unparametrized moduli
space. Also, in the last two cases the middle space denotes a boundary-obstructed
moduli space, i.e, it denotes trajectories which connect a boundary
stable point (as $t\rightarrow-\infty$) with a boundary unstable
point (as $t\rightarrow\infty$).

The following theorem shows that up to a boundary term, $\sum_{z}m_{z}(W_{\xi',Y}^{+},\mathfrak{s}_{\omega},[\mathfrak{c}])$
equals either of the sums (\ref{sum naturality 1}), (\ref{sum naturality 2}).
It can be seen as the analogue of Lemma 4.15 in \cite{MR2299739}
and Proposition 24.6.10 in \cite{MR2388043} (in fact, it was used
implicitly in the proof of the pairing formula in Proposition 6.8
of \cite{MR2299739} and Theorem 6.2 in \cite{Zhang[2016]}):
\begin{prop}
\label{boundary parametrized}If $\mathfrak{M}_{z}(W_{\xi',Y}^{+},\mathfrak{s},[\mathfrak{c}])$
is zero-dimensional, it is compact. If $\mathcal{\mathfrak{M}}_{z}(W_{\xi',Y}^{+},\mathfrak{s}_{\omega},[\mathfrak{c}])$
is one-dimensional and contains irreducible trajectories, then the
compactification $\mathcal{\mathfrak{M}}_{z}^{+}(W_{\xi',Y}^{+},\mathfrak{s}_{\omega},[\mathfrak{c}])$
is a 1-dimensional manifold whose boundary points are of the following
types:

1) The fiber over $L=0$, namely the space $\mathcal{M}_{z}(W_{\xi',Y}^{+},\mathfrak{s}_{\omega},[\mathfrak{c}])$.

2) The fiber over $L=\infty$, namely the three products described
previously.

3) Products of the form 
\[
\mathcal{\mathfrak{M}}(W_{\xi',Y}^{+},\mathfrak{s}_{\omega},[\mathfrak{b}])\times\check{\mathcal{M}}([\mathfrak{b}],\mathfrak{s}_{\xi},[\mathfrak{c}])
\]
 or 
\[
\mathcal{\mathfrak{M}}(W_{\xi',Y}^{+},\mathfrak{s}_{\omega},[\mathfrak{a}])\times\check{\mathcal{M}}([\mathfrak{a}],\mathfrak{s}_{\xi},[\mathfrak{b}])\times\check{\mathcal{M}}([\mathfrak{b}],\mathfrak{s}_{\xi},[\mathfrak{c}])
\]
 where the middle one is boundary obstructed. 
\end{prop}

In order to apply the proposition define $P=[0,\infty)$ and the numbers
\[
m_{z}(W_{\xi',Y}^{+},\mathfrak{s}_{\omega},[\mathfrak{a}])_{P}=\begin{cases}
|\mathcal{\mathfrak{M}}_{z}(W_{\xi',Y}^{+},\mathfrak{s}_{\omega},[\mathfrak{a}])|\;\;\;\mod2 & \text{if }\dim\mathcal{\mathfrak{M}}_{z}(W_{\xi',Y}^{+},\mathfrak{s},[\mathfrak{a}])=0\\
0 & \text{otherwise}
\end{cases}
\]

Recall also that the differential on $\check{C}_{\bullet}(-Y,\mathfrak{s}_{\xi})=\mathfrak{C}^{o}(-Y,\mathfrak{s}_{\xi})\oplus\mathfrak{C}^{s}(-Y,\mathfrak{s}_{\xi})$
is \cite[Definition 22.1.3]{MR2388043}
\[
\check{\partial}=\left(\begin{array}{cc}
\partial_{o}^{o} & -\partial_{o}^{u}\bar{\partial}_{u}^{s}\\
\partial_{s}^{o} & \bar{\partial}_{s}^{s}-\partial_{u}^{u}\bar{\partial}_{u}^{s}
\end{array}\right)
\]
Suppose now that $\mathcal{\mathfrak{M}}_{z}(W_{\xi',Y}^{+},\mathfrak{s}_{\omega},[\mathfrak{c}])$
is one dimensional. We use the previous proposition to count the endpoints
of $\mathcal{\mathfrak{M}}_{z}^{+}(W_{\xi',Y}^{+},\mathfrak{s}_{\omega},[\mathfrak{c}])$
by making cases on $[\mathfrak{c}]$.

\subsection{Case $[\mathfrak{c}]\in\mathfrak{C}^{o}(-Y,\mathfrak{s}_{\xi})$
{[}irreducible critical point{]}}
\begin{enumerate}
\item The fiber over $L=0$, gives the contributions 
\begin{equation}
\sum_{z}m_{z}(W_{\xi',Y}^{+},\mathfrak{s}_{\omega},[\mathfrak{c}])\label{cont1}
\end{equation}
These numbers were used in the chain-level definition of $c(\xi',Y)$.
\item The fiber over $L=\infty$ gives the contributions (\ref{sum naturality 1})
\begin{align}
 & \sum_{[\mathfrak{a}]\in\mathfrak{C}^{o}(-Y')}\;\;\sum_{z_{1},z_{2}}m_{z_{1}}(Z_{Y',\xi'}^{+},\mathfrak{s}',[\mathfrak{a}])n_{z_{2}}([\mathfrak{a}],W_{*}^{\dagger},\mathfrak{s}_{\omega},[\mathfrak{c}])\nonumber \\
+ & \sum_{[\mathfrak{a}]\in\mathfrak{C}^{s}(-Y'),[\mathfrak{b}]\in\mathfrak{C}^{u}(-Y')}\;\;\sum_{z_{1},z_{2},z_{3}}m_{z_{1}}(Z_{Y',\xi'}^{+},\mathfrak{s}',[\mathfrak{a}])\bar{n}_{z_{2}}([\mathfrak{a}],\mathfrak{s}_{\xi'},[\mathfrak{b}])n_{z_{3}}([\mathfrak{b}],W_{*}^{\dagger},\mathfrak{s}_{\omega},[\mathfrak{c}])\label{cont2}\\
+ & \sum_{[\mathfrak{a}]\in\mathfrak{C}^{s}(-Y'),[\mathfrak{b}]\in\mathfrak{C}^{u}(-Y)}\;\;\sum_{z_{1},z_{2},z_{3}}m_{z_{1}}(Z_{Y',\xi'}^{+},\mathfrak{s}',[\mathfrak{a}])\bar{n}_{z_{2}}([\mathfrak{a}],W_{*}^{\dagger},\mathfrak{s}_{\omega},[\mathfrak{b}])n_{z_{3}}([\mathfrak{b}],\mathfrak{s}_{\xi},[\mathfrak{c}])\nonumber 
\end{align}
These numbers were used in the chain-level definition of $\check{m}c(\xi')$. 
\item We obtain contributions of the form 
\begin{align}
 & \sum_{[\mathfrak{a}]\in\mathfrak{C}^{o}(-Y)}\;\;\sum_{w_{1},w_{2}}m_{w_{1}}(W_{\xi',Y}^{+},\mathfrak{s}_{\omega},[\mathfrak{a}])_{P}n_{w_{2}}([\mathfrak{a}],\mathfrak{s}_{\xi},[\mathfrak{c}])\label{cont3}\\
+ & \sum_{[\mathfrak{a}]\in\mathfrak{C}^{s}(-Y),[\mathfrak{b}]\in\mathfrak{C}^{u}(-Y')}\;\;\sum_{w_{1},w_{2},w_{3}}m_{w_{1}}(W_{\xi',Y}^{+},\mathfrak{s}_{\omega},[\mathfrak{a}])_{P}\bar{n}_{w_{2}}([\mathfrak{a}],\mathfrak{s}_{\xi},[\mathfrak{b}])n_{w_{3}}([\mathfrak{b}],\mathfrak{s}_{\xi},[\mathfrak{c}])\nonumber 
\end{align}
These numbers will be used momentarily to define the boundary term. 
\end{enumerate}
By Proposition (\ref{boundary parametrized}) the sum of (\ref{cont1}),
(\ref{cont2}) and (\ref{cont3}) correspond to the number of points
in the boundary of a one dimensional compact manifold, hence it must
equal $0$.

\subsection{Case $[\mathfrak{c}]\in\mathfrak{C}^{s}(-Y,\mathfrak{s}_{\xi})$
{[}boundary stable critical point{]}}
\begin{enumerate}
\item The fiber over $L=0$, gives the contributions 
\begin{equation}
\sum_{z}m_{z}(W_{\xi',Y}^{+},\mathfrak{s}_{\omega},[\mathfrak{c}])\label{cont1-1}
\end{equation}
These numbers were used in the chain-level definition of $c(\xi',Y)$.
\item The fiber over $L=\infty$ gives the contributions (\ref{sum naturality 1})
\begin{align}
 & \sum_{[\mathfrak{a}]\in\mathfrak{C}^{o}(-Y')}\;\;\sum_{z_{1},z_{2}}m_{z_{1}}(Z_{Y',\xi'}^{+},\mathfrak{s}',[\mathfrak{a}])n_{z_{2}}([\mathfrak{a}],W_{*}^{\dagger},\mathfrak{s}_{\omega},[\mathfrak{c}])\label{cont2-1}\\
+ & \sum_{[\mathfrak{a}]\in\mathfrak{C}^{s}(-Y')}\;\;\sum_{z_{1},z_{2}}m_{z}(Z_{Y',\xi'}^{+},\mathfrak{s}',[\mathfrak{a}])\bar{n}_{z_{2}}([\mathfrak{a}],W_{*}^{\dagger},\mathfrak{s}_{\omega},[\mathfrak{c}])\nonumber \\
+ & \sum_{[\mathfrak{a}]\in\mathfrak{C}^{s}(-Y'),[\mathfrak{b}]\in\mathfrak{C}^{u}(-Y')}\;\;\sum_{z_{1},z_{2},z_{3}}m_{z_{1}}(Z_{Y',\xi'}^{+},\mathfrak{s}',[\mathfrak{a}])\bar{n}_{z_{2}}([\mathfrak{a}],\mathfrak{s}_{\xi},[\mathfrak{b}])n_{z_{3}}([\mathfrak{b}],W_{*}^{\dagger},\mathfrak{s}_{\omega},[\mathfrak{c}])\nonumber \\
+ & \sum_{[\mathfrak{a}]\in\mathfrak{C}^{s}(-Y'),[\mathfrak{b}]\in\mathfrak{C}^{u}(-Y)}\;\;\sum_{z_{1},z_{2},z_{3}}m_{z_{1}}(Z_{Y',\xi'}^{+},\mathfrak{s}',[\mathfrak{a}])\bar{n}_{z_{2}}([\mathfrak{a}],W_{*}^{\dagger},\mathfrak{s}_{\omega},[\mathfrak{b}])n_{z_{3}}([\mathfrak{b}],\mathfrak{s}_{\xi},[\mathfrak{c}])\nonumber 
\end{align}
These numbers were used in the chain-level definition of $\check{m}c(\xi')$. 
\item We obtain contributions of the form 
\begin{align}
 & \sum_{[\mathfrak{a}]\in\mathfrak{C}^{o}(-Y)}\sum_{w_{1},w_{2}}m_{w_{1}}(W_{\xi',Y}^{+},\mathfrak{s}_{\omega},[\mathfrak{a}])_{P}n_{w_{2}}([\mathfrak{a}],\mathfrak{s}_{\xi},[\mathfrak{c}])\label{cont3-1}\\
+ & \sum_{[\mathfrak{a}]\in\mathfrak{C}^{s}(-Y)}\sum_{w_{1},w_{2}}m_{w_{1}}(W_{\xi',Y}^{+},\mathfrak{s}_{\omega},[\mathfrak{a}])_{P}\bar{n}_{w_{2}}([\mathfrak{a}],\mathfrak{s}_{\xi},[\mathfrak{c}])\nonumber \\
+ & \sum_{[\mathfrak{a}]\in\mathfrak{C}^{s}(-Y),[\mathfrak{b}]\in\mathfrak{C}^{u}(-Y)}\sum_{w_{1},w_{2},w_{3}}m_{w_{1}}(W_{\xi',Y}^{+},\mathfrak{s}_{\omega},[\mathfrak{a}])_{P}\bar{n}_{w_{2}}([\mathfrak{a}],\mathfrak{s}_{\xi},[\mathfrak{b}])n_{w_{3}}([\mathfrak{b}],\mathfrak{s}_{\xi},[\mathfrak{c}])\nonumber 
\end{align}
These numbers will be used momentarily to define the boundary term. 
\end{enumerate}
As before the sum of (\ref{cont1-1}), (\ref{cont2-1}) and (\ref{cont3-1})
equals $0$.

Define the chain element $\psi\in\mathfrak{C}^{o}(-Y,\mathfrak{s}_{\xi})\oplus\mathfrak{C}^{s}(-Y,\mathfrak{s}_{\xi})$
via the formula 
\[
\psi=\left(\sum_{[\mathfrak{a}]\in\mathfrak{C}^{o}(-Y)}\sum_{w_{1}}m_{w_{1}}(W_{\xi',Y}^{+},\mathfrak{s}_{\omega},[\mathfrak{a}])_{P}e_{[\mathfrak{a}]},\sum_{[\mathfrak{a}]\in\mathfrak{C}^{s}(-Y)}\sum_{w_{1}}m_{w_{1}}(W_{\xi',Y}^{+},\mathfrak{s}_{\omega},[\mathfrak{a}])_{P}e_{[\mathfrak{a}]}\right)
\]
It is not hard to see that 
\[
\check{\partial}\psi=\left(\sum_{[\mathfrak{c}]\in\mathfrak{C}^{o}(-Y)}C_{o}e_{[\mathfrak{c}]},\sum_{[\mathfrak{c}]\in\mathfrak{C}^{s}(-Y)}C_{s}e_{[\mathfrak{c}]}\right)
\]
where $C_{o}$ equals (\ref{cont3}) and $C_{s}$ equals (\ref{cont3-1}).
In other words, we have the chain-level identity 
\[
\check{m}c(\xi')-c(\xi',Y)=\check{\partial}\psi
\]
which gives us the desired identity 
\[
\widecheck{HM}_{\bullet}(W^{\dagger},\mathfrak{s}_{\omega})\mathbf{c}(\xi')=\mathbf{c}(\xi',Y)
\]
concluding the first phase in the proof for the naturality of the
contact invariant under strong symplectic cobordisms. Now we proceed
to address the second part of the proof (as explained at the beginning
of the paper). Namely, we will show that $\mathbf{c}(\xi',Y)$ equals
$\mathbf{c}(\xi)$ by adapting Mrowka and Rollin's ``dilating the
cone'' technique to the case of a manifold with cylindrical end.

\section{6. Generalized Gluing-Excision Theorem}

\subsection{\label{subsec:Gluing-and-Identifying}6.1 Gluing and Identifying
Spin-c Structures}

\ 

Before describing the modified gluing argument why will say very quickly
why the ``special'' condition can be dropped for the symplectic
cobordisms we are working with. More details can be found in section
6.1 of \cite{Echeverria[Thesis]}. The definition Mrowka and Rollin
used for a special symplectic corbordism appears near formula (1.1)
of \cite{MR2199446}, which we repeat for convenience: 
\begin{defn}
\label{def:special cobordism-1-1-1}A cobordism $(W,\omega):(Y,\xi)\rightarrow(Y',\xi')$
is said to be a \textbf{special symplectic cobordism} if:

1) With the symplectic orientation, $\partial W=-Y\sqcup Y'$ and
$\omega$ is strictly positive on $\xi$ and $\xi'$ with their induced
orientations.

2) The symplectic form is given in a collar neighborhood of the concave
boundary by a symplectization of $(Y,\xi)$.

3) The map induced by the inclusion $i^{*}:H^{1}(W,Y';\mathbb{Z})\rightarrow H^{1}(Y;\mathbb{Z})$
is the zero map. 
\end{defn}

Notice that it is the last condition the one that makes the symplectic
cobordism ``special''. We want to work with strong cobordisms, which
in particular means that the convex end is also given by a symplectization
of $(Y,\xi)$ and that the special condition does not appear. The
reason why Mrowka and Rollin introduced this condition is that they
were interested in guaranteeing the injectivity of a certain map 
\[
\jmath:\text{Spin}^{c}(X,\xi)\rightarrow\text{Spin}^{c}(X\cup W,\xi')
\]
where $X$ was a compact manifold with boundary a contact manifold
$(Y,\xi)$ and $\text{Spin}^{c}(X,\xi)$ denotes the isomorphism classes
of (relative) spin-c structures on $X$ whose restriction to $Y$
induce the spin-c structure determined by $\xi$. A similar definition
applies to $\text{Spin}^{c}(X\cup W,\xi')$, where now $X\cup W$
bounds $(Y',\xi')$. In the same way in which for a manifold without
boundary $Z$ the set of spin-c structures $\text{Spin}^{c}(Z)$ is
an affine space over $H^{2}(Z;\mathbb{Z})$, the set $\text{Spin}^{c}(X,\xi)$
is an affine space over $H^{2}(X;Y;\mathbb{Z})$ (this is discussed
in the first three pages of \cite{MR1474156}).

We are interested in the situation when $X=\mathbb{R}^{+}\times-Y$
and so the affine space $H^{2}(\mathbb{R}^{+}\times-Y;Y;\mathbb{Z})$
reduces automatically to a singleton, so regardless of how the map
\[
\jmath:\text{Spin}^{c}(\mathbb{R}^{+}\times-Y,\xi)\rightarrow\text{Spin}^{c}(\mathbb{R}^{+}\times-Y\cup W,\xi')
\]
 is defined, it will automatically be injective. In fact, the definition
of such a map is not difficult to give: as we already mentioned the
contact structure $\xi$ gives rise to a canonical spinor model $S_{\xi}$
on $Y$ (also on $-Y$, \cite[Section 22.5]{MR2388043}) and on $\mathbb{R}^{+}\times-Y$
\cite[Section 4.3]{MR2388043}.

This canonical spinor bundle model over $\mathbb{R}^{+}\times-Y$
represents the (unique) isomorphism class $\mathfrak{s}(\mathbb{R}^{+}\times-Y,\xi)$
of relative spin-c structure inside $\text{Spin}^{c}(\mathbb{R}^{+}\times-Y,\xi)$.
We define $\jmath\left[\mathfrak{s}(\mathbb{R}^{+}\times-Y,\xi)\right]$
by specifying a relative spin-c structure over $\mathbb{R}^{+}\times-Y\cup W$
as follows.

Using the symplectic structure on $W$ we have a canonical spinor
bundle $S_{\omega}$ as well. This induces spinor bundles on $\partial W$
as explained in Section 4.5 of \cite{MR2388043}. Since the symplectic
structure is specified near the boundary by the corresponding contact
structure because of the strong condition in our cobordism it is not
difficult to identify in this way $S_{\xi}$ with $S_{\omega}\mid_{-Y}$
and hence we produce a total spinor bundle over $(\mathbb{R}^{+}\times-Y)\cup W$,
which is representing $\jmath\left[\mathfrak{s}(\mathbb{R}^{+}\times-Y,\xi)\right]$
(more details can be found in the author's thesis cited before). 

Another way to explain why the special condition was needed in the
paper \cite{MR2199446} is to say that $X$ could have interesting
topology, so there was an obstruction problem when trying to extend
certain data defined on the complement of $X$ (for example gauge
transformations) to the entire manifold. However, in our case these
obstructions disappear since we have replaced $X$ with a half-cylinder.

\subsection{6.2 Connected Sum Along $Y$}

\ 

We will now adapt the gluing/ excision theorem in \cite{MR2199446}
to our situation. More precisely we want an analogue of their corollary
3.2.2. The following construction is based on sections 4.1 and 2.1.5
from that paper. There they proved a gluing result for a class of
manifolds with a so called \textit{AFAK end $Z$, }that is, an asymptotically
flat almost Kahler end, the idea being that this class of manifolds
behave sufficiently nice near the symplectic end so all the necessary
analysis goes through. We recall the definition of an AFAK end, which
is Definition 2.1.2 of \cite{MR2199446}.
\begin{defn}
\textbf{\label{def:AKAK end }($AFAK$ end} \textbf{manifold) }An
asymptotically flat almost Kahler end is a manifold $Z$ which admits
a decomposition $C_{Z}\cup_{Y}N$, where $N$ is a not necessarily
compact 4-dimensional manifold, with contact boundary $Y$, endowed
with a fixed contact form $\theta$, and $C_{Z}=(0,T]\times Y$ for
some $T>0$.

In addition, $Z$ is endowed with an almost Kahler structure $(\omega_{Z},J_{Z})$
and a proper function $\sigma_{Z}:Z\rightarrow(0,\infty)$ satisfying:

$a)$ On $(0,T]\times Y\subset Z$ we have $\sigma_{Z}(t,y)=t$.

$b)$ The almost Kahler structure on $C_{Z}$ is the one of an almost
Kahler cone on $(Y,\theta)$.

$c)$ There is a constant $\kappa>0$, such that the injectivity radius
satisfies $\kappa\text{inj}(x)>\sigma(x)$ for all $x\in Z$.

$d)$ For each $x\in Z$, let $e_{x}$ be the map $e_{x}:v\rightarrow\exp_{x}(\sigma_{Z}(x)v/\kappa)$
, and $\gamma_{x}$ be the metric on the unit ball in $T_{x}Z$ defined
as $e_{x}^{*}\gamma_{x}/\sigma_{Z}(x)^{2}$. Then these metrics have
bounded geometry in the sense that all covariant derivatives of the
curvature are bounded by some constants independent of $x$.

$e)$ For each $x\in Z$ , let $o_{x}$ denote the symplectic form
$e_{x}^{*}\omega_{Z}/\sigma_{Z}(x)^{2}$ on the unit ball, then $o_{x}$
similarly approximates the translation invariant symplectic form,
with all its derivatives.

$f)$ For all $\epsilon>0$, the function $e^{-\epsilon\sigma_{Z}}$
is integrable on $Z$.

$g)$ The map induced by the inclusion $Y=\partial N\subset N$ ,
$H_{c}^{1}(N)\rightarrow H^{1}(Y)$ where $H_{c}^{*}(N)$ is the compactly
supported de Rham cohomology, is identically $0$.
\end{defn}

The important things that we need to point out regarding this definition
is that:

$\bullet$ The last condition $g)$ regarding the vanishing of map
between de Rham cohomologies mimics the special condition for a symplectic
cobordism that we already discussed before. Therefore, in our context
this condition is not needed.

$\bullet$ To our cobordism $(W,\omega)$ one can associate an AFAK
end $(Z,\omega_{Z})$ as explained in section 4.1 of \cite{MR2199446}
. We can simplify their construction in our case because our cobordism
is strong so in fact we can exploit the fact that near the convex
end $\omega$ is also determined by a symplectization of the contact
structure. We start be using a collar neighborhood $[T_{0},T_{1})\times Y$
of $Y\subset\partial W$ (with $T_{0}>1$) and a contact form $\theta$
such that the symplectic form $\omega$ near the concave end of that
neighborhood is given by $\frac{1}{2}d(t^{2}\theta)$. We then glue
a sharp cone on the boundary $Y$ by extending the collar neighborhood
into $(0,T_{1})\times Y$ with its symplectic form. Likewise, we have
a similar collar neighborhood near the convex end and we can therefore
glue (after some reparameterizations) the half-infinite cone $[1,\infty)\times Y'$
with the symplectic form $\frac{1}{2}d(t'^{2}\theta')$ where $t'$
denotes the time coordinate on $[1,\infty)\times Y'$. Therefore we
take $Z$ to be
\[
Z=((0,T_{0})\times Y)\cup W\cup\left([1,\infty)\times Y'\right)
\]
Moreover, we can find a ``time coordinate'' $\sigma_{Z}$ on $Z$
as they described in our definition (\ref{def:AKAK end }), i.e, properties
$a),b),c),d),e)$ are satisfied (in fact, after reparametrization
in can be taken to agree with the natural time coordinate on the third
factor $[1,\infty)\times Y'$ of $Z$).

$\bullet$ Notice in particular that our symplectic form $\omega_{Z}$
has the property that it is exact except for a compact set (which
is contained in $W$). Hence the class of manifolds we are using could
be called \textit{AFAKAE ends} (where \textit{AE} stands for almost
exact) but for convenience we will keep calling this manifold an AFAK
end. After choosing a metric $g_{Z}$ and almost complex structure
$J_{Z}$ on $Z$ so that $\omega_{Z}$ is self-dual and of pointwise
norm $\sqrt{2}$ the data $(Z,\omega_{Z},J_{Z},g_{Z},\sigma_{Z})$
will represent an AFAK end with the caveats mentioned above.

This is the class of manifolds to which the generalized excision/gluing
theorem will apply, though the theorem will only be used for this
particular $Z$. The idea will be to glue $Z$ to the cylindrical
end $\mathbb{R}^{+}\times-Y$ using an operation that Mrowka and Rollin
named \textbf{connected sum along $Y$}.

To be more precise, consider as before the symplectic cone $[1,\infty)\times Y$
for the contact form $\theta$ with metric 
\[
g_{K,\theta}=dt\otimes dt+t^{2}g_{\theta}
\]
 and symplectic form 
\[
\omega_{\theta}=\frac{1}{2}d(t^{2}\theta)
\]

Choose a number $\tau>1$ \footnote{It goes without saying that the $\tau$ is completely unrelated to
the $\tau$ used in the $\tau$-model of the configuration space. } and identify an annulus $(1,\tau)\times Y$ in $[1,\infty)\times Y$
with an annulus $(1/\tau,1)\times Y\subset Z$ using the dilation
map 
\begin{align*}
(1,\tau)\times Y & \overset{\nu_{\tau}}{\longrightarrow}(1/\tau,1)\times Y\\
(t,y) & \rightarrow(t/\tau,y)
\end{align*}

Define $M_{\tau}$ as the union of $\left(\mathbb{R}^{+}\times-Y\right)\cup[1,\tau)\times Y$
and $Z\cap\{\sigma_{Z}>1/\tau\}$\textbf{ }
\begin{equation}
M_{\tau}=(\left(\mathbb{R}^{+}\times-Y\right)\cup[1,\tau)\times Y))\cup\left(Z\cap\{\sigma_{Z}>1/\tau\}\right)\label{manifolds Mtau}
\end{equation}
glued along the previous annuli via the dilation map $\nu_{\tau}$.

\begin{figure}[H]
\begin{centering}
\includegraphics[scale=0.5]{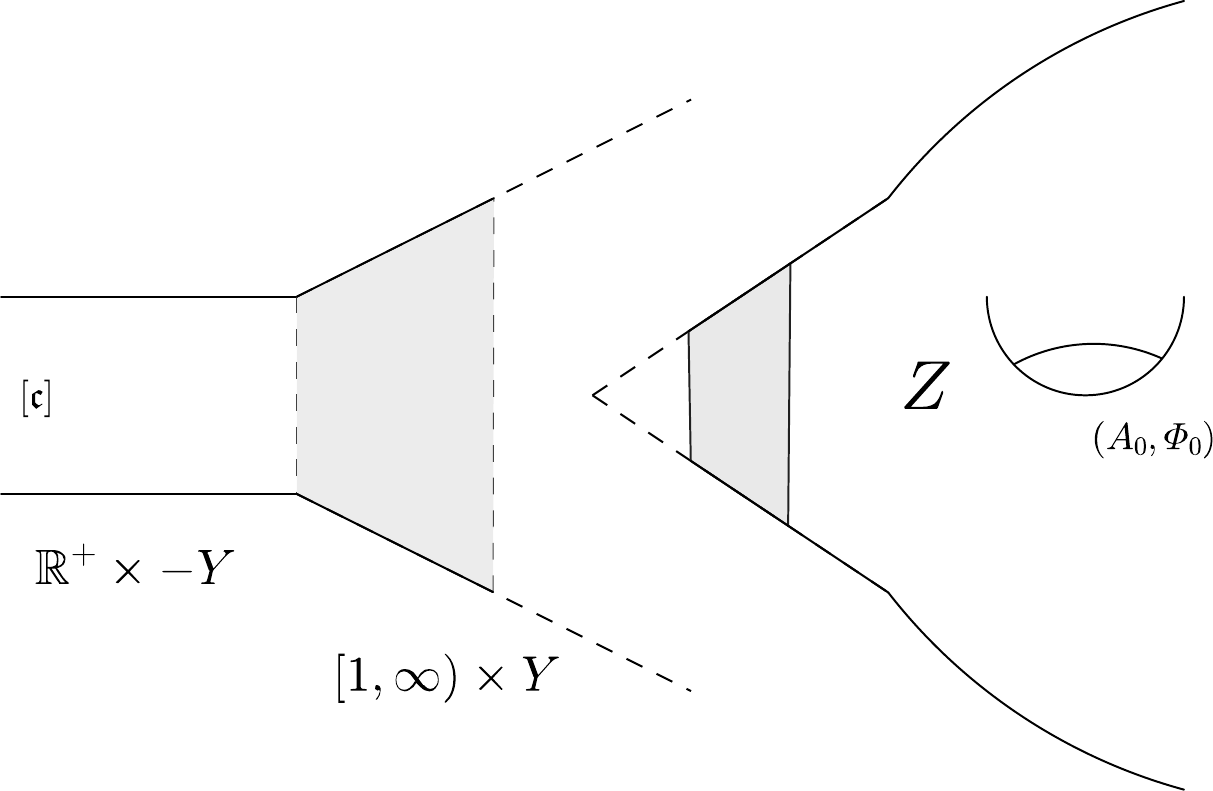}
\par\end{centering}
\caption{\label{fig: Mtau}Using the ``connected sum along $Y$'' operation
to obtain the family of manifolds $M_{\tau}$}
\end{figure}

In the figure, the gray regions represent the annuli that are identified
and the dashed regions are the parts of the cone and $Z$ that are
taken off in the construction. We need to say how to redefine the
geometric structures we had in place (metric, symplectic form, etc)
so that they agree under the identification operation. The symplectic
form can be taken as 
\[
\omega_{Z,\tau}=\tau^{2}\omega_{Z}
\]
and the new ``time coordinate'' becomes 
\[
\sigma_{\tau,Z}=\tau\sigma_{Z}
\]
The metric is a dilation of the original metric, that is,
\[
g_{\tau,Z}=\tau^{2}g_{Z}
\]
 In this way, with respect to $g_{\tau,Z}$, $\omega_{Z,\tau}$ is
self-dual with norm $\sqrt{2}$. As usual, $g_{Z}$ and $\omega_{Z}$
determine a compatible almost complex structure $J_{Z,\tau}$ which
in fact is independent of $\tau$, i.e,
\[
J_{Z,\tau}=J_{Z}
\]
The natural Clifford multiplication is 
\[
\rho_{\tau,Z}(\eta)=\frac{\rho_{Z}(\eta)}{\tau},\;\;\;\;\eta\text{ a one form }
\]
while the spinor bundle remains the same, i.e
\[
S_{\tau,Z}=S_{Z}
\]

We will specify the spin-c structure $\mathfrak{s}_{\tau}$ on $M_{\tau}$
as the isomorphism class of the following spinor bundle $(S_{\tau},\rho_{\tau})$: 

$\bullet$ On $\mathbb{R}^{+}\times-Y$ we use the spinor bundle $(S_{\mathbb{R}^{+}\times-Y},\rho_{\mathbb{R}^{+}\times-Y})$
that the canonical spinor bundle $(S_{\theta},\rho_{\theta})$ on
$Y$ induces on $\mathbb{R}^{+}\times-Y$. 

$\bullet$ Along the boundary, we identify\textbf{ }$(S_{\mathbb{R}^{+}\times-Y,\theta},\rho_{\mathbb{R}^{+}\times-Y,\theta})\mid_{Y}$
with $(S_{K_{Y}},\rho_{K_{Y}})\mid_{\{1\}\times Y}$ where $(S_{K_{Y}},\rho_{K_{Y}})$
denotes the canonical spinor bundle associated to the symplectic cone
$K_{Y}=[1,\infty)\times Y$.

$\bullet$ Over $M_{\tau}\cap\{\sigma_{\tau,Z}<\tau\}=M_{\tau}\cap\{\sigma_{Z}<1\}=M_{\tau}\cap\{(1,\tau)\times Y\}$
we use the spinor bundle $(S_{K_{Y}},\rho_{K_{Y}})$ . 

$\bullet$ Over $M_{\tau}\cap\{\sigma_{\tau,Z}>1\}=M_{\tau}\cap\{\sigma_{Z}>1/\tau\}$
we use the spinor bundle $(S_{\tau,Z},\rho_{\tau,Z})=\left(S_{\tau,Z},\frac{\rho_{Z}}{\tau}\right)$. 

To write the transition map from $(S_{K_{Y}},\rho_{K_{Y}})$ to $(S_{\tau,Z},\rho_{\tau,Z})$
over $M_{\tau}\cap\{1/\tau<\sigma_{Z}<1\}$ observe that if $e_{Y}^{1},e_{Y}^{2},e_{Y}^{3}$
is a coframe at the slice $\{1\}\times Y\simeq Y$ then $dt,te_{Y}^{1},te_{Y}^{2},te_{Y}^{3}$
is a coframe on $(1,\tau)\times Y\subset K_{Y}$ while $\tau dt,\tau te_{Y}^{1},\tau te_{Y}^{2},\tau te_{Y}^{3}$
is a coframe on $\{1/\tau<\sigma_{Z}<1\}\subset Z$. Therefore we
can define as $\bar{\epsilon}_{t}^{01}=\frac{1}{\sqrt{2}}(dt-ite_{Y}^{1})$,
$\bar{\epsilon}_{t}^{23}=\frac{t}{\sqrt{2}}(e_{Y}^{2}-ie_{Y}^{3})$
and the identification map 
\[
\begin{array}{c}
\mathfrak{G}_{\tau}:S_{K_{Y}}\rightarrow S_{\tau,Z}\\
\alpha_{K_{Y}}+\beta_{K_{Y}}\bar{\epsilon}_{t}^{01}\wedge\bar{\epsilon}_{t}^{23}\rightarrow\alpha_{K_{Y}}+\tau^{2}\beta_{K_{Y}}\bar{\epsilon}_{t}^{01}\wedge\bar{\epsilon}_{t}^{23}
\end{array}
\]
 
\begin{rem}
In the case of \cite{MR2199446} , their construction required (in
their notation) the choice of an element $(\mathfrak{s},h)\in\text{Spin}^{c}(M,\omega)$
\cite[Section 2.1.7]{MR2199446}. As we explained before, by using
a half infinite cylinder instead of a compact piece, all of our constructions
can be done in a canonical way, which is why our description is more
simple sense and we can drop the explicit reference to $h$.
\end{rem}

Our (unperturbed) Seiberg Witten map continues to be 
\[
\mathfrak{F}(A,\varPhi)=\left(\frac{1}{2}\rho(F_{A^{t}}^{+})-(\varPhi\varPhi^{*})_{0},D_{A}\varPhi\right)
\]
To define the perturbations, write the half-infinite cylinder as 
\[
\mathbb{R}^{+}\times-Y=\left([0,1]\times-Y\right)\cup\left([1,\infty)\times-Y\right)
\]
where $[0,1]\times-Y$ is going to play the role of a trivial cobordism.
By that we simply mean that the perturbations we use on $[0,1]\times-Y$
are of the form $\hat{\mathfrak{p}}=\beta\hat{\mathfrak{q}}+\beta_{0}\hat{\mathfrak{p}}_{0}$
where $\hat{\mathfrak{p}}$ coincides near $\{1\}\times-Y$ with a
strongly tame perturbation $-\mathfrak{\hat{q}}_{Y,g_{\theta},\mathfrak{s}_{\xi}}$
on $[1,\infty)\times-Y$ and near $\{0\}\times-Y$ it vanishes. On
\[
Z_{\tau}=[1,\tau)\times Y\cup\left(Z\cap\{\sigma_{Z}>1/\tau\}\right)
\]
 consider the perturbation
\[
\mathfrak{p}_{Z_{\tau}}=-\frac{1}{2}\rho_{\tau}(F_{A_{0,\tau}^{t}}^{+})+(\varPhi_{\tau,0}\varPhi_{\tau,0}^{*})_{0}
\]
 where $(A_{0,\tau}^{t},\varPhi_{\tau,0})$ denotes the canonical
solution. Again, similar to the perturbation $\mathfrak{p}_{W_{\xi',Y}^{+}}$
defined in equation  (\ref{eq:glued perturbation}) we can produce
a perturbation 
\begin{equation}
\mathfrak{p}_{M_{\tau}}=-\mathfrak{\hat{q}}_{Y,g_{\theta},\mathfrak{s}_{\xi}}+(\beta\mathfrak{\hat{q}}_{Y,g_{\theta},\mathfrak{s}_{\xi}}+\beta_{0}'\hat{\mathfrak{p}}_{0})+\beta_{K}\mathfrak{p}_{Z_{\tau}}\label{perturbed tau}
\end{equation}
 It is also useful to think of the manifold $Z_{Y,\xi}^{+}$ {[}where
the contact invariant $\mathbf{c}(\xi)$ of $(Y,\xi)$ is defined{]}
as the manifold $M_{\tau}$ obtained by taking ``$\tau=\infty$''.
In other words, we will write 
\[
M_{\infty}\equiv Z_{Y,\xi}^{+}
\]
Notice that on this manifold we can also define a perturbation $\mathfrak{p}_{M_{\infty}}$
in exactly the same way as for $\mathfrak{p}_{M_{\tau}}$ (so it agrees
with $-\mathfrak{\hat{q}}_{Y,g_{\theta},\mathfrak{s}_{\xi}}$ on half-cylinder
$[1,\infty)\times-Y$, it agrees with $\mathfrak{p}_{K}$ on the cone
$[1,\infty)\times Y$ and it is interpolated between these two perturbations
on the finite cylinder $[0,1]\times-Y$ through a perturbation $\beta\mathfrak{\hat{q}}_{Y,g_{\theta},\mathfrak{s}_{\xi}}+\beta_{0}'\hat{\mathfrak{p}}_{0}$
). 

Our previous transversality Theorem (\ref{Theorem perturbations})
now reads as follows:
\begin{lem}
For all critical points $[\mathfrak{c}]\in\mathfrak{C}^{o}(-Y,\mathfrak{s}_{\xi})\oplus\mathfrak{C}^{s}(-Y,\mathfrak{s}_{\xi})$
and for each $0<\tau\leq\infty$ there is a residual subset $\mathcal{P}_{\tau}$
of the large space of perturbations $\mathcal{P}(Y,\mathfrak{s}_{\xi})$
such that for any $\mathfrak{p}_{\tau}\in\mathcal{P}_{\tau}$ the
corresponding perturbation $\mathfrak{p}_{M_{\tau}}$ satisfies the
property that all the moduli spaces $\mathcal{M}(M_{\tau},\mathfrak{s}_{\tau},[\mathfrak{c}],\mathfrak{p}_{M_{\tau}})$
are cut out transversely. 
\end{lem}

When we study the properties of the gluing map it will become clear
that we want to be able to choose a single perturbation $\mathfrak{p}_{all}$
such that when we plug it in the formula for $\mathfrak{p}_{M_{\tau}}$
it guarantees transversality \textit{simultaneously} for all moduli
spaces $\mathcal{M}(M_{\tau},\mathfrak{s}_{\tau},[\mathfrak{c}],\mathfrak{p}_{M_{\tau}})$.
In other words, we would like to be able to choose a perturbation
$\mathfrak{p}_{all}\in\bigcap_{0<\tau\leq\infty}\mathcal{P}_{\tau}$.
However, notice that without further restrictions $\bigcap_{0<\tau\leq\infty}\mathcal{P}_{\tau}$
might be empty.

Fortunately, since we are ultimately interested in the case when $\tau$
is sufficiently large we can choose an increasing sequence $\tau_{n}$
with $\tau_{n}\rightarrow\infty$ and then use the fact that the countable
intersection of residual sets is residual \cite[Theorem 1.4]{MR584443}
so that $\left(\cap_{n}\mathcal{P}_{\tau_{n}}\right)\cap\mathcal{P}_{\infty}$
is residual as well. In particular this means that we can take $\mathfrak{p}_{all}\in\left(\cap_{n}\mathcal{P}_{\tau_{n}}\right)\cap\mathcal{P}_{\infty}$
, which we will assume from now on.

Strictly speaking, since we will work with an additional family $M_{\tau}'$
obtained by using another connected sum operation with another AFAK
end $Z'$ we should really take $\mathfrak{p}_{all}\in\left(\cap_{n}\mathcal{P}_{\tau_{n}}\right)\cap\mathcal{P}_{\infty}\cap\left(\cap_{n}\mathcal{P}_{\tau_{n}}^{\prime}\right)$,
where $\mathcal{P}_{\tau_{n}}^{\prime}$ denotes the residual space
of perturbations for the manifold $M_{\tau}'$. However, for the proof
of the gluing theorem we will end up taking $Z'=(0,\infty)\times Y$
(as in section 4.1 of \cite{MR2199446}), in which case one can check
that \textit{all} the $M_{\tau}'$ end up coinciding with $Z_{Y,\xi}^{+}=M_{\infty}$.
Hence, this point does not make much of a difference. Also, for notational
convenience, we will keep writing the moduli spaces typically as $\mathcal{M}(M_{\tau},\mathfrak{s}_{\tau},[\mathfrak{c}])$
instead of $\mathcal{M}(M_{\tau_{n}},\mathfrak{s}_{\tau_{n}},[\mathfrak{c}])$.

\subsection{6.3 Gluing Map}

\ 

Our main objective in this section is to adapt Theorem 3.1.9 in \cite{MR2199446}
to our situation. First we need to define a pre-gluing map that allows
us to compare solutions in the moduli spaces corresponding to the
manifolds $M_{\tau}$ and $M_{\tau}^{\prime}$. This will then be
promoted to an actual gluing map which basically says that once $\tau$
becomes sufficiently large the Seiberg-Witten solutions on $M_{\tau}$
are in bijective correspondence with the Seiberg-Witten solutions
on $M_{\tau}'$ (the precise statement is Theorem (\ref{fig:Hybrid invariant})). 

As can be seen from Figure (\ref{fig: Mtau}), one should think of
the manifolds $M_{\tau}$ as being diffeomorphic versions of the manifold
$W_{\xi',Y}^{+}$ described in Figure (\ref{fig:Hybrid invariant}).
The moduli space of Seiberg-Witten equations over each of the $M_{\tau}$
gives rise to a ``$\tau$-hybrid'' invariant $c(\xi',Y,\tau)\in\check{C}_{*}(-Y,\mathfrak{s}_{\xi})$,
but a standard deformation of metrics and perturbations argument which
is explained at the end of the paper tells us that in fact they all
define the same homology class $\mathbf{c}(\xi',Y,\tau)=\mathbf{c}(\xi',Y)$,
where the right hand side denotes our original ``hybrid'' invariant.
On the other hand, when we take $Z'=(0,\infty)\times Y$, the resulting
manifolds $M_{\tau}'$ agree with $Z_{Y,\xi}^{+}$ as mentioned at
the end of the previous section. Therefore, from the moduli space
of Seiberg-Witten equations over $M_{\tau}'$ we obtain the ordinary
contact invariant $\mathbf{c}(\xi)$ and then the our gluing argument
will imply that these two invariants agree.

We write $(M_{\tau},g_{\tau},\omega_{\tau},J_{\tau},\sigma_{\tau})$
and $(M_{\tau}',g_{\tau}',\omega_{\tau}',J_{\tau}',\sigma_{\tau}')$
to make explicit the data required in our construction. Notice that
on the domains $\{\sigma_{\tau}\leq\tau\}\subset M_{\tau}$ and $\{\sigma_{\tau}'\leq\tau\}\subset M_{\tau}'$
all the previous structures agree (including the spinor bundles and
the canonical solutions). In fact, we can regard these regions as
subsets of $Z_{Y,\xi}^{+}$. 

Let $(A,\varPhi)$ be a solution of the Seiberg-Witten equations on
$M_{\tau}$. We want to transport $(A,\varPhi)$ into an approximate
solution on $M_{\tau}'$. 

First we need to construct a spinor bundle $S_{(A,\varPhi)}'$ associated
to $(A,\varPhi)$ on $M_{\tau}'$. To be more precise, the isomorphism
class of the spin-c structure is independent of the solution $(A,\varPhi)$
that we use, but the particular model will depend on the solution
since it will be used to define a transition function.

Lemma (\ref{C0 control}) in the Appendix states that we can find
a compact set $C$ with the following significance: for every $\tau$
large enough and for every solution to the Seiberg-Witten equations
on $M_{\tau}$ we have $|\alpha|\geq\frac{1}{2}$ on $M_{\tau}\backslash[(\mathbb{R}^{+}\times-Y)\cup C]$
(recall that $\varPhi=(\alpha,\beta)$ and that the paper \cite{MR2199446}
writes it instead as $(\beta,\gamma)$). We may write $C$ as $C=[1,T]\times Y\subset Z_{Y,\xi}^{+}$
where $T$ is large enough and independent of $\tau$ and the solution
$(A,\varPhi)$. \textbf{From now on we will assume that $\tau$ is
chosen so that it is larger than $T$}.

For $\tau>T$, we construct the spinor bundle $S_{(A,\varPhi)}'$
on $M_{\tau}'$ as follows (\cite{MR2199446} named this spinor bundle
$S_{(A,\varPhi)}$): 
\begin{enumerate}
\item Over the region $M_{\tau}'\cap\{\sigma_{\tau}'\leq\tau\}\subset Z_{Y,\xi}^{+}$,
we use the spinor bundle $S_{\xi}$ determined by $\xi$. Over the
region $M_{\tau}'\cap\{\sigma_{\tau}'\geq T\}$, we use the spinor
bundle determined by the almost complex structure $J_{\tau}'$, i.e,
$S_{J_{\tau}^{\prime}}'$. In other words 
\[
S_{(A,\varPhi)}^{\prime}=\begin{cases}
S_{\xi} & \text{over }M_{\tau}'\cap\{\sigma_{\tau}'\leq\tau\}\\
S_{J_{\tau}'}^{\prime} & \text{over }M_{\tau}'\cap\{\sigma_{\tau}'\geq T\}
\end{cases}
\]
\item To specify what happens over the annulus $\{T\leq\sigma_{\tau}'\leq\tau\}$
define the map (gauge transformation) \footnote{Here we do not use the notation $h'_{(A,\varPhi)}$ that can be found
in \cite{MR2199446} , since our isomorphism $h$ is already canonical
and so there is no need to distinguish $h'_{(A,\varPhi)}$ from $h_{(A,\varPhi)}$.}
\[
\begin{array}{c}
h_{(A,\varPhi)}:M_{\tau}\backslash[(\mathbb{R}^{+}\times-Y)\cup C]\rightarrow S^{1}\\
h_{(A,\varPhi)}=\frac{|\alpha|}{\alpha}
\end{array}
\]
Identifying both of $S_{\xi}$ and $S'_{J_{\tau}'}$ canonically with
$\underline{\mathbb{C}}\oplus\varLambda_{2}^{+}$, then the transition
map $S_{\xi}\rightarrow S'_{J_{\tau}'}$ becomes multiplication by
$h_{(A,\varPhi)}$.
\end{enumerate}

\begin{figure}[H]
\begin{centering}
\includegraphics[scale=0.5]{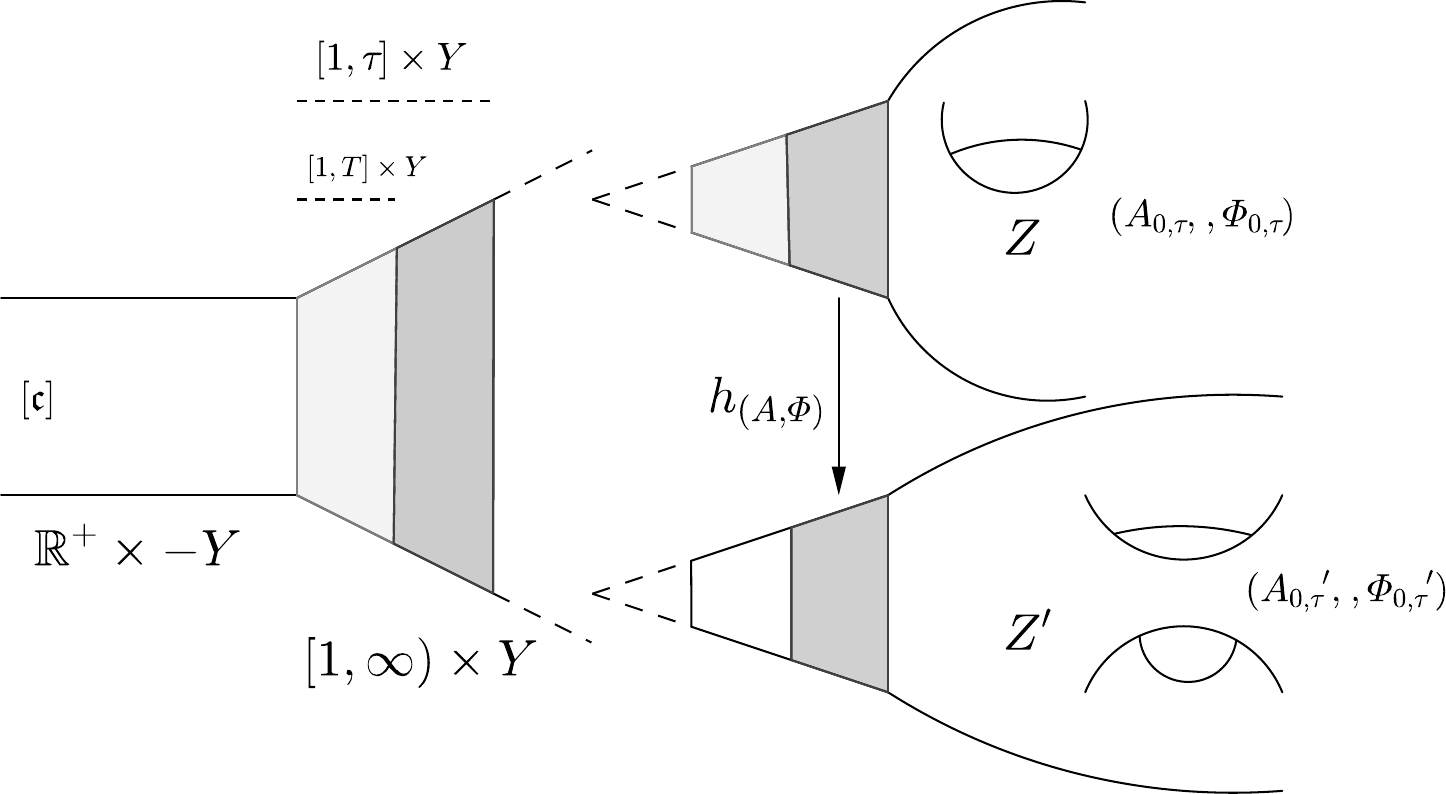}
\par\end{centering}
\caption{\label{fig:Defining-the-spinor}Defining the spinor bundle $S_{(A,\varPhi)}^{\prime}$
over $M_{\tau}'$.}
\end{figure}

Notice that if $u\in\mathcal{G}(M_{\tau})$ then $h_{u\cdot(A,\varPhi)}=u^{-1}h_{(A,\varPhi)}$
and so we can easily build an isomorphism 
\[
u^{\#}:S_{(A,\varPhi)}^{\prime}\rightarrow S_{u\cdot(A,\varPhi)}^{\prime}
\]
 Our next job is to construct a configuration on the spinor bundle
$S_{(A,\varPhi)}'$ over $M_{\tau}'$. For this we recall a family
of cut-off functions described in section 2.2.1 of \cite{MR2199446}
. Let $\chi(t)$ be a smooth decreasing function such that 
\[
\chi(t)=\begin{cases}
0 & t\geq1\\
1 & t\leq0
\end{cases}
\]
and define 
\[
\chi_{\tau}(t)=\chi\left(\frac{t-\tau}{N_{0}}+1\right)=\begin{cases}
0 & t\geq\tau\\
1 & t\leq\tau-N_{0}
\end{cases}
\]
 where $N_{0}$ is a number that is fixed later to control the derivatives
of $\chi_{\tau}$. With the help of this function define $S_{(A,\varPhi)}'$
as follows:

$\bullet$ On the region $M_{\tau}'\cap\{\sigma_{\tau}'<\tau\}$ we
can identify the structures on $M_{\tau}$ with $M_{\tau}'$ and so
$(A,\varPhi)\mid_{M_{\tau}\cap\{\sigma_{\tau}<\tau\}}$ defines a
configuration on $S_{(A,\varPhi)}^{\prime}\mid_{M_{\tau}'\cap\{\sigma_{\tau}'\leq\tau\}}$.

$\bullet$ On the region $M_{\tau}'\cap\{\sigma_{\tau}'\geq T\}$
we can write $\varPhi$ as a pair $(\alpha,\beta)$ and $A$ as $A=A_{0,\tau}+a$
so if we regard $h_{(A,\varPhi)}$ as a gauge transformation (the
one that is used to transition on the annulus $M_{\tau}'\cap\{T\leq\sigma_{\tau}'\leq\tau\}$)
then
\begin{align*}
 & h_{(A,\varPhi)}\cdot(A,\varPhi)\\
= & h_{(A,\varPhi)}\cdot(A_{0,\tau}+a,(\alpha,\beta))\\
= & \left(A_{0,\tau}+a-\frac{\alpha}{|\alpha|}d\left(\frac{|\alpha|}{\alpha}\right),\left(|\alpha|,\frac{|\alpha|}{\alpha}\beta\right)\right)\\
\equiv & \left(A_{0,\tau}+\hat{a},(\hat{\alpha},\hat{\beta})\right)
\end{align*}
Notice that $\hat{\alpha}$ is a real function with $\hat{\alpha}\geq\frac{1}{2}$.
Therefore we define on $M_{\tau}'\cap\{\sigma_{\tau}'\geq T\}$ 
\[
(A,\varPhi)^{\#}\equiv(A_{0,\tau}'+(\chi_{\tau}\circ\sigma_{\tau}')\hat{a},(\hat{\alpha}^{\chi_{\tau}\circ\sigma_{\tau}'},(\chi_{\tau}\circ\sigma_{\tau}')\hat{\beta}))
\]

$\bullet$ On the end $\{\sigma_{\tau}'\geq\tau\}$ we set 
\[
(A,\varPhi)^{\#}=(A_{0,\tau}',\varPhi_{0,\tau}')
\]

Since the construction is compatible with the gauge group action in
the sense that 
\[
u^{\#}\cdot(A,\varPhi)^{\#}=(u\cdot(A,\varPhi))^{\#}
\]
 we have constructed our \textbf{pre-gluing map}
\begin{equation}
\#:\mathcal{M}(M_{\tau},\mathfrak{s}_{\tau},[\mathfrak{c}])\rightarrow(\mathcal{C}/\mathcal{G})(M_{\tau}')\label{PREgluing map}
\end{equation}

Lemma (\ref{exponential decay SW map}) of the Appendix guarantees
the following exponential decay estimate: there is a $\delta>0$ and
$T$ large enough such that for every $N_{0}\geq1,k\in\mathbb{N}$
, $\tau$ satisfying $\tau\geq T+N_{0}$ and every solution $(A,\varPhi)$
of the Seiberg-Witten equations on $M_{\tau}$, we have that $(A,\varPhi)^{\#}$
satisfies the Seiberg-Witten equations on $\{\sigma_{\tau}'\leq T\}\subset M_{\tau}'$
and 
\begin{equation}
|\mathfrak{F}_{\mathfrak{p}_{M_{\tau}'}}(A,\varPhi)^{\#}|_{C^{k}(g_{\tau}',A^{\#})}\leq c_{k}e^{-\delta\sigma_{\tau}}\label{exponential decay approximation}
\end{equation}
 on $\{\sigma_{\tau}'\geq T\}\subset M_{\tau}'$.

Our objective now is to modify the pre-gluing map $\#:M(M_{\tau},\mathfrak{s}_{\tau},[\mathfrak{c}])\rightarrow(\mathcal{C}/\mathcal{G})(M_{\tau}')$
to obtain a \textbf{gluing map} (Theorem 3.1.9 \cite{MR2199446})
\[
\mathfrak{G}_{\tau}:\mathcal{M}(M_{\tau},\mathfrak{s}_{\tau},[\mathfrak{c}])\rightarrow\mathcal{M}(M_{\tau}',\mathfrak{s}_{\tau}',[\mathfrak{c}])
\]
We want to define $\mathfrak{G}_{\tau}$ at the level of configuration
spaces in such a way that is gauge equivariant. Our proposal is that
this map should decompose as 
\begin{equation}
\mathfrak{G}_{\tau}(A,\varPhi)=(A,\varPhi)^{\#}+\left(\mathcal{D}_{(A,\varPhi)^{\#}}\mathfrak{F}_{\mathfrak{p}_{M_{\tau}^{\prime}}}\right)^{*}(b',\psi')\label{decomposition glued map}
\end{equation}
 where $(b',\psi')\in L_{k,A}^{2}(i\mathfrak{su}(S_{\tau}^{\prime+})\oplus S_{\tau}^{\prime-})$
is the quantity that needs to be determined.

Here $\mathcal{D}_{(A,\varPhi)^{\#}}\mathfrak{F}_{\mathfrak{p}_{M_{\tau}^{\prime}}}$
denotes the linearization of the perturbed Seiberg-Witten map $\mathfrak{F}_{\mathfrak{p}_{M_{\tau}^{\prime}}}$.
In the old days of Seiberg-Witten theory, where only the curvature
equation was perturbed by some imaginary-valued self-dual two-form,
this linearized map $\mathcal{D}_{(A,\varPhi)^{\#}}\mathfrak{F}_{\mathfrak{p}_{M_{\tau}^{\prime}}}$
would coincide with the linearization of the \textit{unperturbed}
Seiberg-Witten map $\mathcal{D}_{(A,\varPhi)^{\#}}\mathfrak{F}$,
since the perturbations where independent of the configuration $(A,\varPhi)^{\#}$
being used. In fact, analyzing the formula (\ref{perturbed tau}),
we can see that the discrepancy between these two maps is due to the
(abstract) perturbations used on the cylindrical end, so to emphasize
this point we may sometimes write $\mathcal{D}_{\mathfrak{q},(A,\varPhi)^{\#}}\mathfrak{F}$
instead of the more precise notation $\mathcal{D}_{(A,\varPhi)^{\#}}\mathfrak{F}_{\mathfrak{p}_{M_{\tau}^{\prime}}}$. 

Recall that the perturbed Seiberg-Witten equation is 
\[
\mathfrak{F}_{\mathfrak{p}_{M_{\tau}'}}(A,\varPhi)=\mathfrak{F}(A,\varPhi)+\mathfrak{p}_{M_{\tau}'}(A,\varPhi)=0
\]

By definition, we want $\mathfrak{G}_{\mathfrak{\tau}}(A,\varPhi)$
to solve the previous equation which means that 
\[
\mathfrak{F}\mathfrak{G}_{\tau}(A,\varPhi)+\mathfrak{p}_{M_{\tau}'}\mathfrak{G}_{\mathfrak{\tau}}(A,\varPhi)=0
\]
and implicitly we want to think of the previous equation as depending
on $(b',\psi')$ when we write $\mathfrak{G}_{\tau}(A,\varPhi)$ in
an explicit way as in (\ref{decomposition glued map}) . In order
to write this equation in terms of $(b',\psi')$, one needs to perform
many tedious calculations which will be included in the author's thesis
but not here in order to simplify the exposition. These can be found
on pages $88$ through $96$ of \cite{Echeverria[Thesis]}.

To explain the end result, it is good to compare our calculation with
equation 3.2 in \cite{MR2199446}. Following the notation in section
3 of \cite{MR2199446}, we define first the ``\textbf{perturbed Seiberg-Witten
Laplacian}'' 
\begin{equation}
\triangle_{2,\mathfrak{q},(A,\varPhi)^{\#}}=\left(\mathcal{D}_{(A,\varPhi)^{\#}}\mathfrak{F}_{\mathfrak{q}}\right)\circ\left(\mathcal{D}_{(A,\varPhi)^{\#}}\mathfrak{F}_{\mathfrak{q}}\right)^{*}\label{Laplacian}
\end{equation}
and the quadratic map
\[
\mathcal{Q}(a,\phi)=(-(\phi\phi^{*})_{0},\rho(a)\phi)
\]
Finally, define the perturbation term 
\[
P(b',\psi')=\mathfrak{p}_{M_{\tau}'}(\mathfrak{G}(A,\varPhi))-\mathfrak{p}_{M_{\tau}'}(A,\varPhi)^{\#}-\left(\mathcal{D}_{(A,\varPhi)}\hat{\mathfrak{q}}\right)\circ\left(\mathcal{D}_{(A,\varPhi)^{\#}}\mathfrak{F}_{\mathfrak{q}}\right)^{*}
\]
One can then show the following: 
\begin{thm}
\label{GLUING EQUATION}The configuration $\mathfrak{G}_{\mathfrak{q}}(A,\varPhi)=(A,\varPhi)^{\#}+\left(\mathcal{D}_{\mathfrak{q},(A,\varPhi)^{\#}}\mathfrak{F}_{\mathfrak{q}}\right)^{*}(b',\psi')$
is a solution to the perturbed Seiberg Witten equations $\mathfrak{F}_{\mathfrak{p}_{M_{\tau}'}}\mathfrak{G}_{\mathfrak{q}}(A,\varPhi)=0$
if and only if
\begin{equation}
\triangle_{2,\mathfrak{q},(A,\varPhi)^{\#}}(b',\psi')+\mathcal{Q}\circ\left(\mathcal{D}_{(A,\varPhi)^{\#}}\mathfrak{F}_{\mathfrak{q}}\right)^{*}(b',\psi')+P(b',\psi')=-\mathfrak{F}_{\mathfrak{p}_{M_{\tau}'}}(A,\varPhi)^{\#}\label{NEW GLUING-1}
\end{equation}
\end{thm}

Notice that the term $P(b',\psi')$ is a new term that does not appear
in the usual linearization of the Seiberg Witten equations. This appears
solely due to the presence of the abstract perturbations used in \cite{MR2388043}.
To solve this equation we will need a sharp version of the contraction
mapping theorem. 

Namely, the basic idea is to define 
\[
V_{\mathfrak{q}}=\triangle_{2,\mathfrak{q},(A,\varPhi)^{\#}}(b',\psi')
\]
\textbf{Our intention is to show that $\triangle_{2,\mathfrak{q},(A,\varPhi)^{\#}}$
is invertible} so that if we define $S_{\mathfrak{q},(A,\varPhi)^{\#}}(V_{\mathfrak{q}})$
as 
\[
S_{\mathfrak{q},(A,\varPhi)^{\#}}(V_{\mathfrak{q}})\equiv-\mathcal{Q}\circ\left[\left(\mathcal{D}_{(A,\varPhi)^{\#}}\mathfrak{F}_{\mathfrak{q}}\right)^{*}\right]\left(\triangle_{2,\mathfrak{q},(A,\varPhi)^{\#}}^{-1}V_{\mathfrak{q}}\right)-P\left(\triangle_{2,\mathfrak{q},(A,\varPhi)^{\#}}^{-1}V_{\mathfrak{q}}\right)
\]
the gluing equation (\ref{NEW GLUING-1}) that we need to solve can
be written as 
\[
V_{\mathfrak{q}}=S_{\mathfrak{q},(A,\varPhi)^{\#}}(V_{\mathfrak{q}})-\mathfrak{F}_{\mathfrak{p}_{M_{\tau}^{\prime}}}(A,\varPhi)^{\#}
\]
The solution of this equation will be guaranteed once we shows the
hypothesis of Proposition 2.3.5 in \cite{MR2199446} are satisfied.
Therefore, we will show first that $\triangle_{2,\mathfrak{q},(A,\varPhi)^{\#}}$
is indeed invertible.

\subsection{6.4 Invertibility of $\triangle_{2,\mathfrak{q},(A,\varPhi)^{\#}}$ }

\ 

In this section we seek a version of Proposition 3.1.2 and Corollary
3.1.6 in \cite{MR2199446} , namely, we want to show that:
\begin{thm}
\label{thm invertibility triangle-1} For each $k\geq0$ there exists
a constant $c_{k}>0$ such that for every $\tau$ large enough, every
$N_{0}\geq1$ and every solution $(A,\varPhi)$ of the Seiberg-Witten
equations on $M_{\tau}$ belonging to the zero dimensional strata
of $\mathcal{M}(M_{\tau},\mathfrak{s}_{\tau},[\mathfrak{c}])$, the
operator 
\begin{align*}
\triangle_{2,\mathfrak{q},(A,\varPhi)^{\#}}:L_{k+2,A^{\#}}^{2}(M_{\tau}',g_{\tau}')\rightarrow L_{k,A^{\#}}^{2}(M_{\tau}',g_{\tau}')\\
(b',\psi')\rightarrow\left(\mathcal{D}_{\mathfrak{q},(A,\varPhi)^{\#}}\mathfrak{F}\right)\circ(\mathcal{D}_{\mathfrak{q},(A,\varPhi)^{\#}}\mathfrak{F})^{*}(b',\psi')
\end{align*}
is an isomorphism, and moreover, its inverse $\triangle_{2,\mathfrak{q},(A,\varPhi)^{\#}}^{-1}$
satisfies for all $(b',\psi')$
\[
c_{k}\|(b',\psi')\|_{L_{k+1}^{2}(g_{\tau}',A^{\#})}\geq\|\triangle_{2,\mathfrak{q},(A,\varPhi)^{\#}}^{-1}(b',\psi')\|_{L_{k+3}^{2}(g_{\tau}',A^{\#})}
\]
\end{thm}

Before proceeding we make a few clarifications:
\begin{rem}
\label{Remark invertibility}a) The norms used for the gluing arguments
are gauge equivariant norms, which depend on the configuration $(A,\varPhi)^{\#}$
being used, as can be seen from our use of subscript in the formulas
for the Sobolev spaces. 

b) Our hypothesis regarding the fact that the solution $[(A,\varPhi)]$
belongs to the $0$ dimensional strata of the moduli space $\mathcal{M}(M_{\tau},\mathfrak{s}_{\tau},[\mathfrak{c}])$
has to do with the fact that we will need to find uniform bounds which
we will depend (partly) on the norms of these solutions. Since we
are using gauge equivariant norms and the zero dimensional moduli
spaces $\mathcal{M}_{0}(M_{\tau},\mathfrak{s}_{\tau},[\mathfrak{c}])$
are compact, for a fixed $\tau$ there can only be finitely many terms
to worry about. Clearly, a priori the bounds that we get still depend
on the value of $\tau$ chosen, but we will see that a transversality
argument will help us control these quantities in a way that is $\tau$-independent.
It should be pointed out that this assumption regarding the zero dimensional
strata is not that different from the hypothesis used in other gluing
arguments. See for example Theorem 4.17 in \cite{MR1883043} (which
uses a compactness assumption as well) or Theorem 18.3.5 in \cite{MR2388043}
(which describes all small solutions of a moduli space).

c) The strategy that we will use to prove the invertibility of $\triangle_{2,\mathfrak{q},(A,\varPhi)^{\#}}$
differs from the one employed by \cite{MR2199446} mainly because
of the following reasons. The way \cite{MR2199446} controlled the
norm $\triangle_{2,\mathfrak{q},(A,\varPhi)^{\#}}$ was by first controlling
the norm a different operator $\square_{(A,\varPhi)^{\#}}=Q_{\mathfrak{q},(A,\varPhi)}\circ Q_{\mathfrak{q},(A,\varPhi)}^{*}$
(defined in the proof of Proposition 3.1.2) and then relating the
norms of these two operators through equation (3.6) in their paper.
However, these norms were only comparable because their equation (3.7),
which uses the fact that $D_{A^{\#}}\varPhi^{\#}$ is almost zero.
This was true in their case because the usual Seiberg-Witten equations
do not perturb the Dirac equation and since $(A,\varPhi)^{\#}$ is
very close to being a solution this means that $\varPhi^{\#}$ is
very close to being a harmonic spinor with respect to $D_{A^{\#}}$.
However, the abstract perturbations $\mathfrak{q}$ used in \cite{MR2388043}
do modify the Dirac equation, so any clear relationship between $\square_{(A,\varPhi)^{\#}}$
and $D_{A^{\#}}$ is lost. 
\end{rem}

The proof of this theorem will follow a splicing argument similar
to the one used in section 4.2.2 of \cite{Morgan-Mrowka[2014]} (or
section 4.4 in \cite{MR1883043}). Namely, we will separate the manifold
$M_{\tau}'$ into two pieces (see Figure (\ref{fig:Defining-the-spinor}))
and show the following:
\begin{lem}
\label{Lemma inverse of Laplacians}We can find two operators 
\[
\triangle_{2,\mathfrak{q},(A,\varPhi),cyl}^{-1}:L^{2}(M_{\tau})\rightarrow L_{2,A}^{2}(M_{\tau})
\]
and 
\begin{equation}
\triangle_{2,\widetilde{(A,\varPhi)},end}^{-1}:L^{2}(N_{\tau}^{+})\rightarrow L_{2,\tilde{A}}^{2}(N_{\tau}^{+})\label{inverse LAP}
\end{equation}
 with the following properties:

a) $\triangle_{2,\mathfrak{q},(A,\varPhi),cyl}^{-1}$ is the inverse
to the Seiberg-Witten ``Laplacian'' defined as in equation (\ref{Laplacian})
\[
\triangle_{2,\mathfrak{q},(A,\varPhi)}=\left(\mathcal{D}_{(A,\varPhi)}\mathfrak{F}_{\mathfrak{q}}\right)\circ\left(\mathcal{D}_{(A,\varPhi)}\mathfrak{F}_{\mathfrak{q}}\right)^{*}:L_{2,A}^{2}(M_{\tau})\rightarrow L^{2}(M_{\tau})
\]

b) Likewise, there is a Seiberg-Witten Laplacian
\[
\triangle_{2,(A,\varPhi)^{\#},end}=\left(\mathcal{D}_{(A,\varPhi)^{\#}}\mathfrak{F}\right)\circ\left(\mathcal{D}_{(A,\varPhi)^{\#}}\mathfrak{F}\right)^{*}:L_{2,A}^{2}(M_{\tau}'\backslash(\mathbb{R}^{+}\times-Y))\rightarrow L^{2}(M_{\tau}'\backslash(\mathbb{R}^{+}\times-Y))
\]
which can be extended to an invertible Seiberg-Witten Laplacian 
\[
\triangle_{2,\widetilde{(A,\varPhi)},end}=\left(\mathcal{D}_{\widetilde{(A,\varPhi)}}\mathfrak{F}_{\eta}\right)\circ\left(\mathcal{D}_{\widetilde{(A,\varPhi)}}\mathfrak{F}_{\eta}\right)^{*}:L_{2,\tilde{A}}^{2}(N_{\tau}^{+})\rightarrow L^{2}(N_{\tau}^{+})
\]
where \textup{$N_{\tau}^{+}$} is a capped-off version of $M_{\tau}'\backslash(\mathbb{R}^{+}\times-Y)$
defined in the proof of this lemma.

\end{lem}

\begin{rem}
As we will explain in the proof, $\widetilde{(A,\varPhi)}$ denotes
an extension of $(A,\varPhi)^{\#}$ to the manifold $N_{\tau}^{+}$,
which agrees with $(A,\varPhi)^{\#}$ on the region $Z_{\tau}\backslash\left([1,T/2)\times Y\right)$.
Because of this we will write in the following lemmas $\triangle_{2,\widetilde{(A,\varPhi)},end}^{-1}$
as $\triangle_{2,(A,\varPhi)^{\#},end}^{-1}$.

Also, our construction of $\triangle_{2,\widetilde{(A,\varPhi)},end}$
is not very sharp since we are using this operator as a proxy for
a parametrix argument that will be important in the next section.

\end{rem}

\begin{proof}
We will divide $M_{\tau}'$ into two regions, each of which is the
natural location where each of the operators from the statement of
the lemma can be considered.

\textbf{$\bullet$ The unperturbed region} $(\mathbb{R}^{+}\times-Y)\cup[1,T]\times Y$
{[}\textbf{construction of }$\triangle_{2,\mathfrak{q},(A,\varPhi),cyl}^{-1}${]}:

This refers to the region where $(A,\varPhi)^{\#}=(A,\varPhi)$, that
is, the solution $(A,\varPhi)$ was not modified. Notice that this
includes the cylinder $\mathbb{R}^{+}\times-Y$ and the section of
the cone $[1,T]\times Y$ and we can use the fact that the moduli
space on $(M_{\tau},g_{\tau})$ is regular to conclude that $Q_{\mathfrak{q},(A,\varPhi)}$
is surjective on $M_{\tau}$ \cite[Def. 14.5.6]{MR2388043}. Using
that $Q_{\mathfrak{q},(A,\varPhi)}=\mathcal{D}_{(A,\varPhi)}\mathfrak{F}_{\mathfrak{q}}\oplus\mathbf{d}_{(A,\varPhi)}^{*}$
we conclude that $\mathcal{D}_{(A,\varPhi)}\mathfrak{F}_{\mathfrak{q}}$
must be surjective as well.

Now, since $Q_{\mathfrak{q},(A,\varPhi)}$ is a Fredholm operator
we can easily see that $Q_{\mathfrak{q},(A,\varPhi)}^{*}$ must be
injective. Moreover $Q_{\mathfrak{q},(A,\varPhi)}^{*}=\left(\mathcal{D}_{(A,\varPhi)}\mathfrak{F}_{\mathfrak{q}}\right)^{*}\oplus\mathbf{d}_{(A,\varPhi)}$
so $\left(\mathcal{D}_{(A,\varPhi)}\mathfrak{F}_{\mathfrak{q}}\right)^{*}$
must be injective as well.

In particular, it is not difficult to check that because of this $\triangle_{2,\mathfrak{q},(A,\varPhi)}$
will be invertible as an operator on $M_{\tau}$. To emphasize that
we care about this operator when applied to sections supported on
the unperturbed region (which contains the cylinder) we will write
the inverse as a map 
\[
\triangle_{2,\mathfrak{q},(A,\varPhi),cyl}^{-1}:L^{2}(M_{\tau})\rightarrow L_{2,A}^{2}(M_{\tau})
\]

\textbf{$\bullet$ The perturbed region $M_{\tau}'\backslash(\mathbb{R}^{+}\times-Y)$
{[}construction of $\triangle_{2,\widetilde{(A,\varPhi)},end}^{-1}${]}:}

This refers to the region where $(A,\varPhi)^{\#}$ and $(A,\varPhi)$
do not necessarily agree. Notice that this includes the region $[1,\tau]\times Y$
together with remaining piece of the $AFAK$ end $Z'$. Moreover $\triangle_{2,\mathfrak{q},(A,\varPhi)^{\#}}=\triangle_{2,(A,\varPhi)^{\#}}$
on this part of the manifold.

We now follow section 3.3 in \cite{MR2199446}, which uses a similar
idea to what we are about to do, albeit for dealing with the orientability
of the moduli spaces.

Recall that we can find an almost complex manifold $N$ with boundary
$\partial N=Y$ \cite[Lemma 4.4]{MR1668563}. We can glue this piece
to $M_{\tau}'\backslash(\mathbb{R}^{+}\times-Y)$ along the cone $[1,\tau]\times Y$
to obtain a manifold $N_{\tau}^{+}$ without boundary and with an
AFAK end (so these are precisely the class of manifolds that were
studied in \cite{MR1668563}.

In particular, since an almost complex structure now exists globally
on $N_{\tau}^{+}$ we can talk about the canonical pair $(A_{0,\tau}',\varPhi_{0,\tau}')$
globally, which we will denote as $\widetilde{(A_{0,\tau}',\varPhi'_{0,\tau}})$.
This configuration is almost a solution to the Seiberg-Witten equations
\[
\begin{cases}
\frac{1}{2}\rho(F_{A^{t}}^{+})-(\varPhi\varPhi^{*})_{0}=\frac{1}{2}\rho(F_{\tilde{A}_{0,\tau}^{\prime}}^{+})-(\tilde{\varPhi}_{0,\tau}^{\prime}\tilde{\varPhi}_{0,\tau}^{\prime*})_{0}\\
D_{A}\varPhi=0
\end{cases}
\]
on $N_{\tau}^{+}$. The only reason why we say almost is that on the
almost complex manifold $N$ we attached, the nondegenerate two form
$\omega_{N}$ induced by the metric and almost complex structure on
$N$ may not give rise to a symplectic form. In general, there is
torsion (a Lee-Gauduchon) form $\eta_{N}$ characterized by 
\[
d\omega_{N}=\omega_{N}\wedge\eta
\]
In particular, $\widetilde{(A_{0,\tau}',\varPhi'_{0,\tau}})$ will
be a solution to the equations
\[
\begin{cases}
\frac{1}{2}\rho(F_{A^{t}}^{+})-(\varPhi\varPhi^{*})_{0}=\frac{1}{2}\rho(F_{\tilde{A}_{0,\tau}^{\prime}}^{+})-(\tilde{\varPhi}_{0,\tau}^{\prime}\tilde{\varPhi}_{0,\tau}^{\prime*})_{0}\\
D_{A}\varPhi=\frac{1}{4}\rho(\eta)\varPhi
\end{cases}
\]
on $N_{\tau}^{+}$. Associated to these Seiberg-Witten equations,
there is a corresponding operator $Q_{\eta,(\tilde{A}_{0,\tau}^{\prime},\tilde{\varPhi}_{0,\tau}^{\prime})}$.
The discussion in \cite[Section 3.3.1]{MR2199446} in fact shows that
the moduli space of solutions is unobstructed at $\widetilde{(A_{0,\tau}',\varPhi'_{0,\tau}})$,
that is, both of the operators $Q_{\eta,(\tilde{A}_{0,\tau}^{\prime},\tilde{\varPhi}_{0,\tau}^{\prime})}$
and $Q_{\eta,(\tilde{A}_{0,\tau}^{\prime},\tilde{\varPhi}_{0,\tau}^{\prime})}^{*}$
are invertible. Hence, as in the argument from the first part of this
lemma we can conclude that the corresponding operator 
\[
\triangle_{2,\widetilde{(A_{0,\tau}',\varPhi'_{0,\tau}}),end}=\left(\mathcal{D}_{\widetilde{(A_{0,\tau}',\varPhi'_{0,\tau}})}\mathfrak{F}_{\eta}\right)\circ\left(\mathcal{D}_{\widetilde{(A_{0,\tau}',\varPhi'_{0,\tau}})}\mathfrak{F}_{\eta}\right)^{*}:L_{2,\tilde{A}_{0,\tau}^{\prime}}^{2}(N_{\tau}^{+})\rightarrow L^{2}(N_{\tau}^{+})
\]
 is invertible.

Now we explain what to do with the other configurations $(A,\varPhi)^{\#}$.
We define $\widetilde{(A,\varPhi)}$ as an extension to $N_{\tau}^{+}$
in such a way that:

a) $\widetilde{(A,\varPhi)}$ agrees with $\widetilde{(A_{0,\tau}',\varPhi'_{0,\tau}})$
on the region $N\cup\left([1,T/4]\times Y\right)$.

b) $\widetilde{(A,\varPhi)}$ agrees with $(A,\varPhi)^{\#}$ on the
region $N_{\tau}^{+}\backslash\left(N\cup\left([1,T/2)\times Y\right)\right)=Z_{\tau}\backslash\left([1,T/2)\times Y\right)$.

c) $\widetilde{(A,\varPhi)}$ interpolates between $\widetilde{(A_{0,\tau}',\varPhi'_{0,\tau}})$
and $(A,\varPhi)^{\#}$ on the remaining region $[T/4,T/2]\times Y$.

In particular, since $(A,\varPhi)^{\#}$ also agrees $\widetilde{(A_{0,\tau}',\varPhi'_{0,\tau}})$
whenever $t>\tau$, we find that with $\widetilde{(A,\varPhi)}$ differs
from $\widetilde{(A_{0,\tau}',\varPhi'_{0,\tau}})$ at most on the
region $[T/4,\tau]\times Y$ . Since being an invertible operator
is an open condition, if we wanted to show that the corresponding
operator $\triangle_{2,(\tilde{A},\tilde{\varPhi}),end}$ is invertible,
it would suffice to show that it is close in the operator norm to
$\triangle_{2,\widetilde{(A_{0,\tau}',\varPhi'_{0,\tau}}),end}$.

In fact, one can compare in a very explicit fashion both operators
(we do something very similar in the next lemma for example hence
we omit some of the details for now), so we just need to have chosen
$T$ from the beginning of our gluing construction in section $6.3$
in such a way that this norm condition is satisfied, so that $\triangle_{2,(\tilde{A},\tilde{\varPhi}),end}$
ends up being invertible (again, a similar argument appears in the
next lemma).
\end{proof}
\begin{lem}
\label{Uniform parametrix}For any solution $(A,\varPhi)$ to the
Seiberg Witten equations on the manifold $M_{\tau}$, the operator
$\triangle_{2,\mathfrak{q},(A,\varPhi)^{\#}}$ is a Fredholm operator
on $M_{\tau}'$.
\end{lem}

\begin{proof}
We will introduce some cutoff functions that will allow us to follow
the usual splicing arguments: these will be denoted $\eta_{cyl}$
and $\eta_{end}$. They satisfy the following properties:

$\bullet$ $0\leq\eta_{cyl},\eta_{end}\leq1$ and $\eta_{cyl}^{2}+\eta_{end}^{2}=1$.

$\bullet$ $\eta_{cyl}$ is supported on the unperturbed region. Moreover,
$\eta_{cyl}\equiv1$ on a small neighborhood of the region $\mathbb{R}^{+}\times-Y\cup[1,T/2]\times Y$
. In particular, \textbf{the gradient of $\eta_{cyl}$ is supported
on the fixed region $[1,T]\times Y$.}

$\bullet$ $\eta_{end}$ is supported on the perturbed region. Moreover,
$\eta_{end}\equiv1$ on a small neighborhood of $([T,\tau]\times Y)\cup\{Z'\cap\{\sigma_{Z'}>1/\tau\}$.
In particular, \textbf{the gradient of $\eta_{end}$ is supported
on the fixed region $[1,T]\times Y$. }Notice that also $\eta_{end}$
will vanish in a small neighborhood of $\{T/2\}\times Y$, since $\eta_{cyl}$
is equal to $1$ near that slice.

$\bullet$ For $\eta=\eta_{cyl},\eta_{end}$ we have $|\nabla^{n}\eta|\leq\left(\frac{2}{T}\right)^{n}$
.

Our proto-inverse will be the operator 
\begin{align*}
\tilde{\triangle}_{2,\mathfrak{q},(A,\varPhi)^{\#}}^{-1}:L^{2}(M_{\tau}',g_{\tau}')\rightarrow L_{2}^{2}(M_{\tau}',g_{\tau}')\\
(b',\psi')\rightarrow\eta_{cyl}\triangle_{2,\mathfrak{q},(A,\varPhi),cyl}^{-1}[\eta_{cyl}(b',\psi')]+\eta_{end}\triangle_{2,(A,\varPhi)^{\#},end}^{-1}[\eta_{end}(b',\psi')]
\end{align*}

We will see that $\tilde{\triangle}_{2,\mathfrak{q},(A,\varPhi)^{\#}}^{-1}$
provides a parametrix for $\triangle_{2,\mathfrak{q},(A,\varPhi)^{\#}}$:
this is because 
\begin{align*}
 & \triangle_{2,\mathfrak{q},(A,\varPhi)^{\#}}\left[\tilde{\triangle}_{2,\mathfrak{q},(A,\varPhi)^{\#}}^{-1}(b',\psi')\right]\\
= & \triangle_{2,\mathfrak{q},(A,\varPhi)^{\#}}\left[\eta_{cyl}\triangle_{2,\mathfrak{q},(A,\varPhi),cyl}^{-1}[\eta_{cyl}(b',\psi')]+\eta_{end}\triangle_{2,(A,\varPhi)^{\#},end}^{-1}[\eta_{end}(b',\psi')]\right]\\
= & \left\{ \mathcal{D}_{(A,\varPhi)}\eta_{cyl},\triangle_{2,\mathfrak{q},(A,\varPhi),cyl}^{-1}[\eta_{cyl}(b',\psi')]\right\} +\eta_{cyl}^{2}(b',\psi')\\
 & +\left\{ \mathcal{D}_{(A,\varPhi)}\eta_{end},\triangle_{2,(A,\varPhi)^{\#},end}^{-1}[\eta_{end}(b',\psi')]\right\} +\eta_{end}^{2}(b',\psi')\\
 & +\left\{ \mathcal{D}_{(A,\varPhi)}\eta_{cyl},\mathcal{D}_{(A,\varPhi)}\left(\triangle_{2,\mathfrak{q},(A,\varPhi),cyl}^{-1}[\eta_{cyl}(b',\psi')]\right)\right\} +\left\{ \mathcal{D}_{(A,\varPhi)}\eta_{end},\mathcal{D}_{(A,\varPhi)}\left(\triangle_{2,(A,\varPhi)^{\#},end}^{-1}[\eta_{end}(b',\psi')]\right)\right\} \\
= & \left\{ \mathcal{D}_{(A,\varPhi)}\eta_{cyl},\triangle_{2,\mathfrak{q},(A,\varPhi),cyl}^{-1}[\eta_{cyl}(b',\psi')]\right\} +\left\{ \mathcal{D}_{(A,\varPhi)}\eta_{end},\triangle_{2,(A,\varPhi)^{\#},end}^{-1}[\eta_{end}(b',\psi')]\right\} +(b',\psi')\\
 & +\left\{ \mathcal{D}_{(A,\varPhi)}\eta_{cyl},\mathcal{D}_{(A,\varPhi)}\left(\triangle_{2,\mathfrak{q},(A,\varPhi),cyl}^{-1}[\eta_{cyl}(b',\psi')]\right)\right\} +\left\{ \mathcal{D}_{(A,\varPhi)}\eta_{end},\mathcal{D}_{(A,\varPhi)}\left(\triangle_{2,(A,\varPhi)^{\#},end}^{-1}[\eta_{end}(b',\psi')]\right)\right\} 
\end{align*}

Here the notation $\{\cdot,\cdot\}$ is used to indicate a bilinear
pointwise multiplication between some (higher order) derivatives of
the cutoff functions and the elements in the domain. Also, the notation
$\mathcal{D}_{(A,\varPhi)}\eta_{cyl}$ means that this expression
involves (higher order) derivatives of the perturbation (and a priori
the configuration $(A,\varPhi)$, but whose precise form is not important
to us.

Notice that the first and last two terms are supported on the compact
subset $[1,T]\times Y$, where $(A,\varPhi)^{\#}=(A,\varPhi)$. Also,
we dropped the dependence on $\mathfrak{q}$ for the derivatives $\mathcal{D}_{(A,\varPhi)}\eta_{\bullet}$
since this perturbation affects only the cylindrical region. To analyze
if there is any dependence of $\mathcal{D}_{(A,\varPhi)}\eta_{\bullet}$
on $(A,\varPhi)$, we will study $\mathcal{D}_{(A,\varPhi)}\eta_{cyl}$
since the other case is exactly the same. We need to compute
\begin{equation}
\triangle_{2,\mathfrak{q},(A,\varPhi)^{\#}}\left(\eta_{cyl}(b_{cyl},\psi_{cyl})\right)\label{computation derivatives}
\end{equation}
where we defined
\[
(b_{cyl},\psi_{cyl})\equiv\triangle_{2,\mathfrak{q},(A,\varPhi),cyl}^{-1}[\eta_{cyl}(b',\psi')]
\]
Notice that we may write 
\[
\triangle_{2,\mathfrak{q},(A,\varPhi)^{\#}}=\left(\mathcal{D}_{(A,\varPhi)^{\#}}\mathfrak{F}_{\mathfrak{q}}\right)\circ\left(\mathcal{D}_{(A,\varPhi)^{\#}}\mathfrak{F}_{\mathfrak{q}}\right)^{*}=\left(\mathcal{D}_{(A,\varPhi)^{\#}}\mathfrak{F}_{\mathfrak{}}\right)\circ\left(\mathcal{D}_{(A,\varPhi)^{\#}}\mathfrak{F}_{\mathfrak{}}\right)^{*}=\triangle_{2,(A,\varPhi)^{\#}}
\]
since we are only interested in computing (\ref{computation derivatives})
on the region $[1,T]\times Y$, where $\eta_{cyl}$ is not constant.
The advantage of using this unperturbed Seiberg-Witten 'Laplacian'
is that we can give an explicit formula for it based on the equations
(\ref{derivative SW gauge-1}) and (\ref{dual derivative SW}). We
find that for arbitrary $(b',\psi')$ 
\begin{align*}
\triangle_{2,(A,\varPhi)^{\#}}(b',\psi')= & \left(\mathcal{D}_{(A,\varPhi)^{\#}}\mathfrak{F}\right)\left((d^{+})^{*}\rho^{*}b'+\rho^{*}(\psi'(\varPhi^{\#})^{*}),D_{A^{\#}}^{*}\psi'-b'\varPhi^{\#}\right)\\
= & [\rho(d^{+}\left[(d^{+})^{*}\rho^{*}b'+\rho^{*}(\psi'(\varPhi^{\#})^{*})\right]-\left\{ \varPhi^{\#}\left(D_{A^{\#}}^{*}\psi'-b'\varPhi^{\#}\right)^{*}+\left(D_{A^{\#}}^{*}\psi'-b'\varPhi^{\#}\right)(\varPhi^{\#})^{*}\right\} _{0},\\
 & D_{A^{\#}}\left(D_{A^{\#}}^{*}\psi'-b'\varPhi^{\#}\right)+\rho\left((d^{+})^{*}\rho^{*}b'+\rho^{*}(\psi'(\varPhi^{\#})^{*})\right)\varPhi^{\#}]
\end{align*}
therefore $\triangle_{2,(A,\varPhi)^{\#}}\left[(\eta_{cyl}b_{cyl},\eta_{cyl}\psi_{cyl})\right]$
becomes 
\begin{align*}
 & [\rho(d^{+}\left[(d^{+})^{*}\rho^{*}\left(\eta_{cyl}b_{cyl}\right)+\rho^{*}(\eta_{cyl}\psi_{cyl}(\varPhi^{\#})^{*})\right]\\
- & \left\{ \varPhi^{\#}\left(D_{A^{\#}}^{*}\left(\eta_{cyl}\psi_{cyl}\right)-(\eta_{cyl}b_{cyl})\varPhi^{\#}\right)^{*}+\left(D_{A^{\#}}^{*}\left(\eta_{cyl}\psi_{cyl}\right)-(\eta_{cyl}b_{cyl})\varPhi^{\#}\right)(\varPhi^{\#})^{*}\right\} _{0}\\
+ & [\rho(d^{+}\left[(d^{+})^{*}\rho^{*}(\eta_{cyl}b_{cyl})+\rho^{*}(\left(\eta_{cyl}\psi_{cyl}\right)(\varPhi^{\#})^{*})\right]\\
- & \left\{ \varPhi^{\#}\left(D_{A^{\#}}^{*}\left(\eta_{cyl}\psi_{cyl}\right)-(\eta_{cyl}b_{cyl})\varPhi^{\#}\right)^{*}+\left(D_{A^{\#}}^{*}\left(\eta_{cyl}\psi_{cyl}\right)-(\eta_{cyl}b_{cyl})\varPhi^{\#}\right)(\varPhi^{\#})^{*}\right\} 
\end{align*}
Since the Dirac operator $D$ satisfies the Leibniz Rule \cite[Prop. 3.38]{MR2273508}
\[
D(\eta\psi)=\rho(d\eta)\psi+\eta D\psi
\]
the only derivatives of $\eta_{cyl}$ that appear are those involving
its exterior derivative, which is independent of the configuration
$(A,\varPhi)^{\#}$ that is used. A similar story is be true for $\eta_{end}$.
Since $\triangle_{2,\mathfrak{q},(A,\varPhi),cyl}^{-1}[\eta_{cyl}(b',\psi')]$
and $\triangle_{2,(A,\varPhi)^{\#},end}^{-1}[\eta_{end}(b',\psi')]$
are elements of $L_{2,A^{\#}}^{2}(M_{\tau}')$, our previous discussion
in fact tells us that the operator 
\begin{align*}
K_{(A,\varPhi)}:L^{2}([1,T]\times Y)\rightarrow & L^{2}([1,T]\times Y)\\
(b',\psi')\rightarrow & \left\{ \mathcal{D}\eta_{cyl},\triangle_{2,\mathfrak{q},(A,\varPhi),cyl}^{-1}[\eta_{cyl}(b',\psi')]\right\} +\left\{ \mathcal{D}\eta_{end},\triangle_{2,(A,\varPhi)^{\#},end}^{-1}[\eta_{end}(b',\psi')]\right\} \\
 & +\left\{ \mathcal{D}_{(A,\varPhi)}\eta_{cyl},\mathcal{D}_{(A,\varPhi)}\left(\triangle_{2,\mathfrak{q},(A,\varPhi),cyl}^{-1}[\eta_{cyl}(b',\psi')]\right)\right\} \\
 & +\left\{ \mathcal{D}_{(A,\varPhi)}\eta_{end},\mathcal{D}_{(A,\varPhi)}\left(\triangle_{2,(A,\varPhi)^{\#},end}^{-1}[\eta_{end}(b',\psi')]\right)\right\} 
\end{align*}
can in fact be regarded as an operator 
\[
K_{(A,\varPhi)}:L^{2}([1,T]\times Y)\rightarrow L_{2}^{2}([1,T]\times Y)
\]
 and using the compact inclusion $L_{2}^{2}([1,T]\times Y)\hookrightarrow L^{2}([1,T]\times Y)$
on a compact manifold we conclude that $K_{(A,\varPhi)}$ is a compact
operator. This provides the desired parametrix.
\end{proof}
The next lemma addresses the bounds on the norms of this parametrix.
\begin{lem}
\label{uniform parametrices}The norms of the parametrices can be
chosen in a uniform way. That is, that there exist a constant $C_{T}$
so that for all solutions $(A,\varPhi)$ we have that $\|K_{(A,\varPhi)}\|\leq\frac{C_{T}}{T}.$
In fact, one can find a constant $C_{\infty}$, independent of $T$
such that $\|K_{(A,\varPhi)}\|\leq\frac{C_{\infty}}{T}$.
\end{lem}

\begin{proof}
It is possible to show that the norms of the parametrices are uniform,
that is, that there exist a constant $C_{T}$ so that for all solutions
$(A,\varPhi)$ we have that $\|K_{(A,\varPhi)}\|\leq\frac{C_{T}}{T}.$
In fact, we will show something better, which is that one could have
chosen a constant $C_{\infty}$ which is independent of $T$\textbf{
}, in other words, $\|K_{(A,\varPhi)}\|\leq\frac{C_{\infty}}{T}$.

Notice that a priori the only term that may not seem controllable
in terms of $T$ are
\begin{equation}
\left\{ \mathcal{D}\eta_{cyl},\triangle_{2,\mathfrak{q},(A,\varPhi),cyl}^{-1}[\eta_{cyl}(b',\psi')]\right\} \text{ and }\left\{ \mathcal{D}_{(A,\varPhi)}\eta_{cyl},\mathcal{D}_{(A,\varPhi)}\left(\triangle_{2,\mathfrak{q},(A,\varPhi),cyl}^{-1}[\eta_{cyl}(b',\psi')]\right)\right\} \label{control terms parametrix}
\end{equation}

We will discuss how to control the first term in (\ref{control terms parametrix})
since the second is exactly the same. If we take a sequence of solutions
$(A_{n},\varPhi_{n})$ on $M_{\tau_{n}}$ then on $[1,T]$ it will
converge strongly to a solution $(A_{\infty},\varPhi_{\infty})$ on
$Z_{Y,\xi}^{+}$ {[}this is because of Lemma (\ref{Mrowka Rohlin compactnessx})
in the Appendix, based on the compactness theorem 2.2.11 in \cite{MR2199446}{]}
and hence for all $(b',\psi')$ 
\[
\left\{ \mathcal{D}\eta_{cyl},\triangle_{2,\mathfrak{q},(A_{n},\varPhi_{n}),cyl}^{-1}[\eta_{cyl}(b',\psi')]\right\} 
\]
 converges to 
\[
\left\{ \mathcal{D}\eta_{cyl},\triangle_{2,\mathfrak{q},(A_{\infty},\varPhi_{\infty}),cyl}^{-1}[\eta_{cyl}(b',\psi')]\right\} 
\]

It is clear then that it would be enough to have a uniform bound on
the operator norms 
\[
\left\Vert \triangle_{2,\mathfrak{q},(A_{\infty},\varPhi_{\infty})}^{-1}\right\Vert _{L^{2}(Z_{Y,\xi}^{+})\rightarrow L_{2,A_{\infty}}^{2}(Z_{Y,\xi}^{+})}
\]

As we will make more explicitly in the next proof, since we are taking
a sequence of solutions $(A_{n},\varPhi_{n})$ which belong to the
zero dimensional strata of the moduli spaces $\mathcal{M}(M_{\tau_{n}},\mathfrak{s}_{\tau_{n}},[\mathfrak{c}])$,
the limiting solution $(A_{\infty},\varPhi_{\infty})$ must belong
to the zero dimensional strata of $\mathcal{M}(Z_{Y,\xi}^{+},\mathfrak{s},[\mathfrak{c}])$,
and since we are using gauge equivariant norms, there are only finitely
many values the previous operator norm can take (this is related to
the second point in the remarks we made after stating the invertibility
of the Laplacian). Therefore, we will have the uniform bound for the
operator $K_{(A,\varPhi)}$, that is, $\|K_{(A,\varPhi)}\|\leq\frac{C_{\infty}}{T}$
where $C_{\infty}$ is independent of $\tau,T$ and the solutions
$(A,\varPhi)$ used.

Therefore, there is no loss of generality in assuming that $T$ was
chosen from the beginning so that it would also satisfy the condition
\[
\|K_{(A,\varPhi)}\|_{L^{2}([1,T]\times Y)\rightarrow L_{2}^{2}([1,T]\times Y)}\leq\frac{C_{\infty}}{T}\leq\frac{1}{2}
\]
for all the solutions of the Seiberg Witten equations on $M_{\tau}$.
In particular, from the identity 
\[
\triangle_{2,\mathfrak{q},(A,\varPhi)^{\#}}\left[\tilde{\triangle}_{2,\mathfrak{q},(A,\varPhi)^{\#}}^{-1}(b',\psi')\right]=K_{(A,\varPhi)}(b',\psi')+(b',\psi')
\]
 we see that the operator norms satisfy 
\[
\|\triangle_{2,\mathfrak{q},(A,\varPhi)^{\#}}\tilde{\triangle}_{2,\mathfrak{q},(A,\varPhi)^{\#}}^{-1}-\text{Id}\|_{L_{2,A^{\#}}^{2}(M_{\tau}',g_{\tau}')\rightarrow L^{2}(M_{\tau}',g_{\tau'})}\leq\frac{1}{2}
\]

In particular we conclude that $\triangle_{2,\mathfrak{q},(A,\varPhi)^{\#}}\tilde{\triangle}_{2,\mathfrak{q},(A,\varPhi)^{\#}}^{-1}$
is invertible and they (and their inverses) are uniformly bounded
since 
\[
\frac{1}{2}\leq\|\triangle_{2,\mathfrak{q},(A,\varPhi)^{\#}}\tilde{\triangle}_{2,\mathfrak{q},(A,\varPhi)^{\#}}^{-1}\|_{L_{2,A^{\#}}^{2}(M_{\tau}',g_{\tau}')\rightarrow L^{2}(M_{\tau}',g_{\tau'})}\leq\frac{3}{2}
\]

Therefore the inverse of $\triangle_{2,\mathfrak{q},(A,\varPhi)^{\#}}$
is $\tilde{\triangle}_{2,\mathfrak{q},(A,\varPhi)^{\#}}^{-1}\left(\triangle_{2,\mathfrak{q},(A,\varPhi)^{\#}}\tilde{\triangle}_{2,\mathfrak{q},(A,\varPhi)^{\#}}^{-1}\right)^{-1}$. 
\end{proof}
Returning to our proof of Theorem (\ref{thm invertibility triangle-1}),
since $\left(\triangle_{2,\mathfrak{q},(A,\varPhi)^{\#}}\tilde{\triangle}_{2,\mathfrak{q},(A,\varPhi)^{\#}}^{-1}\right)^{-1}$
is uniformly bounded we just need to check that $\tilde{\triangle}_{2,\mathfrak{q},(A,\varPhi)^{\#}}^{-1}$
is uniformly bounded to conclude that $\triangle_{2,\mathfrak{q},(A,\varPhi)^{\#}}$
is uniformly bounded {[}a similar argument would work to give uniform
bounds on $\triangle_{2,\mathfrak{q},(A,\varPhi)^{\#}}^{-1}${]} .
Looking at the definition of $\tilde{\triangle}_{2,\mathfrak{q},(A,\varPhi)^{\#}}^{-1}$
it becomes clear that it suffices to show that $\eta_{cyl}\triangle_{2,\mathfrak{q},(A,\varPhi),cyl}^{-1}[\eta_{cyl}(b',\psi')]$
is uniformly bounded. 

Here we will use again the assumption we mentioned at the end of the
previous proof. Namely, we are now assuming that\textbf{ }the gauge
equivalence classes of our solutions $(A,\varPhi)\in\mathcal{M}(M_{\tau},\mathfrak{s}_{\tau},[\mathfrak{c}])$
all belong to the zero dimensional strata of the moduli spaces. Since
the Laplacians are gauge equivariant in the sense that
\[
\triangle_{2,\mathfrak{q},u\cdot(A,\varPhi)}[u\cdot(b,\psi)]=u\cdot\triangle_{2,\mathfrak{q},(A,\varPhi)}(b,\psi)
\]
and we are using the gauge equivariant norms $\|\cdot\|_{L_{k,A}^{2}}$
, then for each $\tau$ there are only finitely gauge equivalence
classes we need to worry about, which immediately implies that for
each $\tau$ we have a control on the Laplacians (and their inverses).
Clearly we still need to see what happens as we vary $\tau$.

Let $K$ be a subset of $(\mathbb{R}^{+}\times-Y)\cup([1,\infty)\times Y)$
and use $\|\triangle_{2,\mathfrak{q},(A,\varPhi)}\|_{A,K}$ or $\|\triangle_{2,\mathfrak{q},(A,\varPhi)}^{-1}\|_{A,K}$
to denote the operator norms of $\triangle_{2,\mathfrak{q},(A,\varPhi)}$
and $\triangle_{2,\mathfrak{q},(A,\varPhi)}^{-1}$ when restricted
to sections supported on $K$. Clearly if $K\subset K'$ then $\|\triangle_{2,\mathfrak{q},(A,\varPhi)}^{-1}\|_{A,K}\leq\|\triangle_{2,\mathfrak{q},(A,\varPhi)}^{-1}\|_{A,K'}$
. 

Now, recall that we are actually working with a sequence $\tau_{n}$
increasing to $\infty$, so for each $\tau_{n}$ let $[(A_{n},\varPhi_{n})]\in\mathcal{M}_{0}(M_{\tau_{n}},\mathfrak{s}_{\tau},[\mathfrak{c}])$
be a (gauge equivalence class of) solution belonging to the zero dimensional
strata. Notice that each compact subset $K\subset(\mathbb{R}^{+}\times-Y)\cup([1,\infty)\times Y)$
eventually belongs to all $M_{\tau_{n}}$ (once $\tau_{n}$ is sufficiently
large) so the compactness theorem in this case says that we can choose
representatives $(A_{n},\varPhi_{n})$ which converge to a solution
$(A_{\infty},\varPhi_{\infty})$ which solves the equations on $Z_{\xi,Y}^{+}$
and this convergence is strong when restricted to the compact subset
$K$. In particular, it is clear from this that
\begin{equation}
\|\triangle_{2,\mathfrak{q},(A_{n},\varPhi_{n})}^{-1}\|_{A_{n},K}\rightarrow\|\triangle_{2,\mathfrak{q},(A_{\infty},\varPhi_{\infty})}^{-1}\|_{A_{\infty},K}\label{convergence norms}
\end{equation}
In fact, we must also have that the limiting solution $(A_{\infty},\varPhi_{\infty})$
belongs to the zero dimensional strata because the different strata
are labeled by the index of the operator $Q_{\mathfrak{q},(A,\varPhi)}$
and this index can only decrease (this is how the broken trajectories
appear). However, since the index of each element in the sequence
was already zero then the index of the limiting configuration would
need to be negative if it were to decrease but transversality rules
this out, since we do not have negative dimensional moduli spaces.
Therefore the convergence is without broken trajectories, that is,
$[(A_{\infty},\varPhi_{\infty})]\in\mathcal{M}_{0}(Z_{Y,\xi}^{+},\mathfrak{s},[\mathfrak{c}])$
. In particular, the fact that no energy is lost along the half-cylinder
allows us to improve the convergence in (\ref{convergence norms})
to (we will say more about this in a moment)
\begin{equation}
\|\triangle_{2,\mathfrak{q},(A_{n},\varPhi_{n})}^{-1}\|_{A_{n},K_{t}}\rightarrow\|\triangle_{2,\mathfrak{q},(A_{\infty},\varPhi_{\infty})}^{-1}\|_{A_{\infty},K_{t}}\label{eq:convergence norms}
\end{equation}
where now $K_{t}=(\mathbb{R}^{+}\times-Y)\cup([1,t]\times Y)$ ($t>1$
is arbitrary). In particular, 
\[
\|\triangle_{2,\mathfrak{q},(A_{\infty},\varPhi_{\infty})}^{-1}\|_{A_{\infty},K_{t}}\leq\|\triangle_{2,\mathfrak{q},(A_{\infty},\varPhi_{\infty})}^{-1}\|_{A_{\infty},Z_{Y,\xi}^{+}}\leq C
\]
where 
\[
C=\max\left\{ \|\triangle_{2,\mathfrak{q},(A_{\infty},\varPhi_{\infty})}^{-1}\|_{A_{\infty},Z_{Y,\xi}^{+}}\mid[(A_{\infty},\varPhi_{\infty})]\in\mathcal{M}_{0}(Z_{Y,\xi}^{+},\mathfrak{s},[\mathfrak{c}])\right\} 
\]
Since $t$ and the sequence was arbitrary this clearly gives us the
uniform bound that we were after so we have proven Theorem (\ref{thm invertibility triangle-1}).

We will now say more about why the convergence (\ref{eq:convergence norms})
is true. For this we need to recall that thanks to the fiber product
description of our moduli spaces, we can restrict each solution $[(A_{n},\varPhi_{n})]$
to a solution on the cylindrical end moduli space $\mathcal{M}(\mathbb{R}^{+}\times-Y,\mathfrak{s}_{\xi},[\mathfrak{c}])$
, which we will denote as $[(A_{n},\varPhi_{n})]_{cyl}\in\mathcal{M}(\mathbb{R}^{+}\times-Y,\mathfrak{s}_{\xi},[\mathfrak{c}])$.

Likewise, the limiting solution $[(A_{\infty},\varPhi_{\infty})]$
can also be restricted to this moduli space so we have as well that
$[(A_{\infty},\varPhi_{\infty})]_{cyl}\in\mathcal{M}(\mathbb{R}^{+}\times-Y,\mathfrak{s}_{\xi},[\mathfrak{c}])$.
When we described the configuration spaces at the beginning of the
paper we used the topology of strong convergence on compact subsets
$L_{k,loc}^{2}$ to define the moduli spaces. However, as explained
in Theorem 13.3.5 of \cite{MR2388043}, the same moduli space $\mathcal{M}(\mathbb{R}^{+}\times-Y,\mathfrak{s}_{\xi},[\mathfrak{c}])$
can also be obtained if we had used the stronger topology of $L_{k}^{2}$
convergence along the entire half-cylinder $\mathbb{R}^{+}\times-Y$
{[}they really did this for the moduli space on the cylinder $\mathbb{R}\times Y$
but it does not affect our claim{]}. Therefore, the convergence of
$[(A_{n},\varPhi_{n})]_{cyl}$ towards $[(A_{\infty},\varPhi_{\infty})]_{cyl}$
can be regarded as a strong convergence with respect to the $L_{k,A_{\mathfrak{c}}}^{2}$
norm, where $A_{\mathfrak{c}}$ represents the translation invariant
connection associated to a smooth representative $\mathfrak{c}$ of
the critical point $[\mathfrak{c}]$. In other words, we can choose
representatives of $[(A_{n},\varPhi_{n})]_{cyl}$ and $[(A_{\infty},\varPhi_{\infty})]_{cyl}$
so that
\[
\begin{cases}
A_{n}=A_{\mathfrak{c}}+a_{n}\\
A_{\infty}=A_{\mathfrak{c}}+a_{\infty}\\
\varPhi_{n}=\varPhi_{\mathfrak{c}}+\phi_{n}\\
\varPhi_{\infty}=\varPhi_{\mathfrak{c}}+\phi_{\infty}
\end{cases}
\]
where $\varPhi_{\mathfrak{c}}$ is a translation invariant representative
of $\mathfrak{c}$ and we have that 
\begin{align*}
\lim_{n\rightarrow\infty}\|A_{n}-A_{\infty}\|_{L_{k}^{2}(\mathbb{R}^{+}\times-Y)}=\lim_{n\rightarrow\infty}\|a_{n}-a_{\infty}\|_{L_{k}^{2}(\mathbb{R}^{+}\times-Y)}=0\\
\lim_{n\rightarrow\infty}\|\varPhi_{n}-\varPhi_{\infty}\|_{L_{k,A_{\mathfrak{c}}}^{2}(\mathbb{R}^{+}\times-Y)}=\lim_{n\rightarrow\infty}\|\phi_{n}-\phi_{\infty}\|_{L_{k,A_{\mathfrak{c}}}^{2}(\mathbb{R}^{+}\times-Y)}=0
\end{align*}
The norms $\|\cdot\|_{L_{k,A_{n}}^{2}}$ and $\|\cdot\|_{L_{k,A_{\mathfrak{c}}}^{2}}$
can now be compared thanks to the Sobolev multiplication theorems
(since for example $\nabla_{A_{n}}\bullet=\nabla_{A_{\infty}}\bullet+(a_{n}-a_{\infty})\otimes\bullet$
with similar formulas for the higher derivatives{]} and the previous
limits make it clear that the operator norm convergence (\ref{convergence norms})
on compact subsets $K$ can be improved to the operator norm convergence
(\ref{eq:convergence norms}) on sets of the form ``half-cylinder
+compact''.

In the next section we explain the properties of the gluing map one
obtains using the invertibility of the Laplacian.
\begin{rem}
Many of the following arguments will have a similar structure to the
one before. Namely, because we are taking solutions belonging to the
zero dimensional strata for an individual $\tau$ we will find a bound,
but a priori this may depend on $\tau$. However, as we take $\tau_{n}$
sufficiently large the bounds end up being controlled by limiting
case '$\tau=\infty$', since we can invoke the strong convergence
on the half-cylindrical end. Since the arguments are essentially the
same in each case we will not repeat the strategy so we will just
say that it ``follows by similar arguments''.
\end{rem}

\subsection{6.5 Definition and some Properties of the Gluing Map:}

\ 

As explained before, if we write $V_{\mathfrak{q}}=\triangle_{2,\mathfrak{q},(A,\varPhi)^{\#}}(b',\psi')$
then the gluing equation (\ref{NEW GLUING-1}) is equivalent to solving
the equation 
\[
V_{\mathfrak{q}}=S_{\mathfrak{q},(A,\varPhi)^{\#}}(V_{\mathfrak{q}})-\mathfrak{F}_{\mathfrak{p}_{M_{\tau}'}}(A,\varPhi)^{\#}
\]

The solution of this equation requires an application of the contraction
mapping theorem, which requires us to first show that the map $S_{\mathfrak{q}}$
is a uniform contraction in the following sense (this is analogue
of lemma 3.1.8 in \cite{MR2199446}):
\begin{thm}
\label{Fixed Point Theorem}For every $k$ large enough there exist
constants $\alpha_{k}>0$, $\kappa_{k}\in(0,1/2)$ such that for every
$\tau$ large enough, every $N_{0}\geq1$ and every approximate solution
of the Seiberg Witten equations $(A,\varPhi)^{\#}$ on $M_{\tau}'$
, which comes from an actual solution $(A,\varPhi)$ on $M_{\tau}$
whose gauge equivalence class $[A,\varPhi]$ belongs to the zero dimensional
strata of the moduli space $\mathcal{M}_{0}(M_{\tau};\mathfrak{s}_{\tau};[\mathfrak{c}])$
,we have for all $V_{1},V_{2}\in L_{k}^{2}(M_{\tau}';i\mathfrak{su}(S^{+})\oplus S^{-},g_{\tau}';A^{\#})$
\[
\|V_{1}\|_{L_{k}^{2}(g_{\tau}',A^{\#})},\|V_{2}\|_{L_{k}^{2}(g_{\tau}';A^{\#})}\leq\alpha_{k}\implies\left\Vert S_{\mathfrak{q},(A,\varPhi)^{\#}}(V_{2})-S_{\mathfrak{q},(A,\varPhi)^{\#}}(V_{1})\right\Vert _{L_{k}^{2}(g_{\tau}',A^{\#})}\leq\kappa_{k}\|V_{2}-V_{1}\|_{L_{k}^{2}(g_{\tau}',A^{\#})}
\]
\end{thm}

\begin{proof}
Recall that
\[
S_{\mathfrak{q},(A,\varPhi)^{\#}}(V_{\mathfrak{q}})=-\mathcal{Q}\circ\left[\left(\mathcal{D}_{(A,\varPhi)^{\#}}\mathfrak{F}_{\mathfrak{q}}\right)^{*}\right]\left(\triangle_{2,\mathfrak{q},(A,\varPhi)^{\#}}^{-1}V_{\mathfrak{q}}\right)-P\left(\triangle_{2,\mathfrak{q},(A,\varPhi)^{\#}}^{-1}V_{\mathfrak{q}}\right)
\]
We will mention the main differences compared with the proof given
in \cite{MR2199446}. First of all, we need the bounds in proposition
11.4.1 in \cite{MR2388043}, which say that for $k\geq2$
\begin{equation}
\|\mathcal{D}_{(A,\varPhi)}^{l}\hat{\mathfrak{q}}\|\leq C(1+\|a\|_{L_{k}^{2}(Z)})^{2k(l+1)}(1+\|\varPhi\|_{L_{k,A}^{2}(Z)}))^{l+1}\label{bounds derivatives of the perturation}
\end{equation}

Here $C$ is a constant independent of the configuration and in this
theorem $Z$ denotes a finite cylinder, while $A=A_{ref}+a\otimes1$
for some reference configuration $A_{ref}$. First of all these bounds
can be used on the half-cylinder $\mathbb{R}^{+}\times(-Y)$ as well.
Simply decompose it as 
\[
\mathbb{R}^{+}\times(-Y)=\bigcup_{n\geq0}\underbrace{[n,n+1]\times(-Y)}_{Z_{n}}
\]
If $\bullet$ denotes an element in the domain of $\mathcal{D}_{(A,\varPhi)}^{l}\hat{\mathfrak{q}}$
then we have 
\begin{align*}
 & \|\mathcal{D}_{(A,\varPhi)}^{l}\hat{\mathfrak{q}}(\bullet)\|_{\mathbb{R}^{+}\times-Y}\\
= & \sum_{n=0}^{\infty}\|\mathcal{D}_{(A,\varPhi)}^{l}\hat{\mathfrak{q}}(\bullet)\|_{Z_{n}}\\
\leq & C\sum_{n=0}^{\infty}(1+\|a\|_{L_{k}^{2}(Z_{n})})^{2k(l+1)}(1+\|\varPhi\|_{L_{k,A}^{2}(Z_{n})}))^{l+1}\|\bullet\|_{Z_{n}}
\end{align*}
where in the last step we used the bounds coming from the operator
norm (\ref{bounds derivatives of the perturation}). If we define
\[
C_{n,(A,\varPhi)}=(1+\|a\|_{L_{k}^{2}(Z_{n})})^{2k(l+1)}(1+\|\varPhi\|_{L_{k,A}^{2}(Z_{n})}))^{l+1}
\]
then it is not too difficult to see that 
\[
C_{\max,(A,\varPhi)}=\max_{n}C_{n,(A,\varPhi)}<\infty
\]
One reason for this is that the previous quantities $C_{n,(A,\varPhi)}$
do not differ too much from those for the translation invariant solution
$C_{n,(A_{\mathfrak{c}},\varPhi_{\mathfrak{c}})}$, which are independent
of $n$.

To compare $C_{n,(A,\varPhi)}$ with $C_{n,(A_{\mathfrak{c}},\varPhi_{\mathfrak{c}})}$,
take our reference connection to be $A_{\mathfrak{c}}$ and use the
fact in this case $a=A-A_{\mathfrak{c}}$ can be chosen to be exponentially
decaying as well as $\varPhi-\varPhi_{\mathfrak{c}}$ , since $(A,\varPhi)$
is asymptotic to $(A_{\mathfrak{c}},\varPhi_{\mathfrak{c}})$ (\cite[Section 13.5 and Proposition 13.6.1]{MR2388043}
or the Appendix). Hence a term like 
\[
(1+\|\varPhi\|_{L_{k,A}^{2}(Z_{n})}))^{l+1}
\]
 can be bounded by 
\[
(1+\|\varPhi-\varPhi_{\mathfrak{c}}\|_{L_{k,A}^{2}(Z_{n})}+\|\varPhi_{\mathfrak{c}}\|_{L_{k,A}^{2}(Z_{n})}))^{l+1}
\]
where the term $\|\varPhi-\varPhi_{\mathfrak{c}}\|_{L_{k,A}^{2}(Z_{n})}$
contributes less as $n$ increases, say less than $\frac{1}{n^{2}}$
for sufficiently large $n$.

In any case we end up with 
\[
\|\mathcal{D}_{(A,\varPhi)}^{l}\hat{\mathfrak{q}}(\bullet)\|_{\mathbb{R}^{+}\times-Y}\leq CC_{\max,(A,\varPhi)}\sum_{n=0}^{\infty}\|\bullet\|_{Z_{n}}=CC_{\max,(A,\varPhi)}\|\bullet\|_{\mathbb{R}^{+}\times-Y}
\]
Since $\bullet$ was arbitrary this says that each $\mathcal{D}_{(A,\varPhi)}^{l}\hat{\mathfrak{q}}$
is a bounded operator on the half-cylinder. For each $\tau$, we are
only dealing with finitely many gauge equivalence classes of solutions
because of our assumption on the strata so the bounds are once again
controlled for a fixed $\tau$. By analogous arguments, one can find
bounds which actually become independent of $\tau$ so that $\|\mathcal{D}_{(A,\varPhi)}^{l}\hat{\mathfrak{q}}\|\leq C_{l}$
for some constant $C_{l}$ on the half-infinite cylinder. 

The other ingredient is that the leading term of $P\left(\triangle_{2,\mathfrak{q},(A,\varPhi)^{\#}}^{-1}V_{\mathfrak{q}}\right)$
is quadratic in the following sense. To emphasize its dependence on
$V$, we will write $P\left(\triangle_{2,\mathfrak{q},(A,\varPhi)^{\#}}^{-1}V_{\mathfrak{q}}\right)$
as 
\[
f(V)=\mathfrak{q}\left((A,\varPhi)^{\#}+(\mathcal{D}_{(A,\varPhi)^{\#}}\mathfrak{F}_{\mathfrak{q}})^{*}\triangle_{2,\mathfrak{q},(A,\varPhi)^{\#}}^{-1}V\right)-\mathfrak{q}(A,\varPhi)^{\#}-\left(\mathcal{D}_{(A,\varPhi)^{\#}}\mathfrak{q}\right)\circ(\mathcal{D}_{(A,\varPhi)^{\#}}\mathfrak{F}_{\mathfrak{q}})^{*}\triangle_{2,\mathfrak{q},(A,\varPhi)^{\#}}^{-1}V
\]
We want to compute $f'(V)$ and $f^{\prime\prime}(V)$ , that is,
the Banach spaces derivatives with respect to $V$ . For this define
the functions 
\[
\begin{cases}
f_{1}(V)=\mathfrak{q}\left((A,\varPhi)^{\#}+(\mathcal{D}_{(A,\varPhi)^{\#}}\mathfrak{F}_{\mathfrak{q}})^{*}\triangle_{2,\mathfrak{q},(A,\varPhi)^{\#}}^{-1}V\right)-\mathfrak{q}(A,\varPhi)^{\#}\\
f_{2}(V)=\left(\mathcal{D}_{(A,\varPhi)^{\#}}\mathfrak{q}\right)\circ(\mathcal{D}_{(A,\varPhi)^{\#}}\mathfrak{F}_{\mathfrak{q}})^{*}\triangle_{2,\mathfrak{q},(A,\varPhi)^{\#}}^{-1}V
\end{cases}
\]
so that 
\[
f(V)=f_{1}(V)-f_{2}(V)
\]
Since $f_{2}(V)$ is linear in $V$ it is easy to determine that 
\[
f'_{2}(V)=\left(\mathcal{D}_{(A,\varPhi)^{\#}}\mathfrak{q}\right)\circ(\mathcal{D}_{(A,\varPhi)^{\#}}\mathfrak{F}_{\mathfrak{q}})^{*}\triangle_{2,\mathfrak{q},(A,\varPhi)^{\#}}^{-1}(V)
\]
Clearly $f_{2}'$ is independent as a linear transformation of the
``basepoint'' (which is hidden in our notation) so we will have
that $f_{2}^{(n)}=0$ for $n\geq2$. To compute the derivative of
$f_{1}(V)$ think of the Taylor expansion of $\mathfrak{q}$ about
$(A,\varPhi)^{\#}$ (which plays the role of $0$ in our affine space
interpretation for the domain of $\mathfrak{q}$ so we can use corollary
4.4 in Chapter 1 from \cite{MR1666820}). In this way 
\[
f_{1}(V)=\left(\mathcal{D}_{(A,\varPhi)^{\#}}\mathfrak{q}\right)\circ\left((\mathcal{D}_{(A,\varPhi)^{\#}}\mathfrak{F}_{\mathfrak{q}})^{*}\triangle_{2,\mathfrak{q},(A,\varPhi)^{\#}}^{-1}\right)V+\frac{1}{2}\left(\mathcal{D}_{(A,\varPhi)^{\#}}^{2}\mathfrak{q}\right)V^{(2)}+\cdots+
\]
where $V^{(2)}=(V,V)$. Notice that the first term is exactly $f_{2}(V)$!
Therefore 
\[
\begin{cases}
f_{1}'(V)=f_{2}(V)\\
f_{2}''=\left(\mathcal{D}_{(A,\varPhi)^{\#}}^{2}\mathfrak{q}\right)
\end{cases}
\]
This means that the leading term for the Taylor expansion of $f(V)$
will be quadratic, that is 
\begin{equation}
f(V)=\frac{1}{2}\left(\mathcal{D}_{(A,\varPhi)^{\#}}^{2}\mathfrak{q}\right)V^{(2)}+\cdots+\label{quadratic expansion}
\end{equation}
 To see why this is important notice that in the case of $-\mathcal{Q}\circ\left[\left(\mathcal{D}_{(A,\varPhi)^{\#}}\mathfrak{F}_{\mathfrak{q}}\right)^{*}\right]\left(\triangle_{2,\mathfrak{q},(A,\varPhi)^{\#}}^{-1}V_{\mathfrak{q}}\right)$
Mrowka and Rollin found a bound (after eq. 3.14 \cite{MR2199446})
which can be adapted to our case to read 
\begin{align}
 & \left\Vert \mathcal{Q}\circ\left[\left(\mathcal{D}_{(A,\varPhi)^{\#}}\mathfrak{F}_{\mathfrak{q}}\right)^{*}\right]\left(\triangle_{2,\mathfrak{q},(A,\varPhi)^{\#}}^{-1}V_{2}\right)-\mathcal{Q}\circ\left[\left(\mathcal{D}_{(A,\varPhi)^{\#}}\mathfrak{F}_{\mathfrak{q}}\right)^{*}\right]\left(\triangle_{2,\mathfrak{q},(A,\varPhi)^{\#}}^{-1}V_{1}\right)\right\Vert _{L_{k}^{2}(g_{\tau}',A^{\#})}\label{eq:inequality Q}\\
\leq & C_{k}'\|V_{2}+V_{1}\|_{L_{k}^{2}(g_{\tau},A^{\#})}\|V_{2}-V_{1}\|_{L_{k}^{2}(g_{\tau}',A^{\#})}\nonumber 
\end{align}
where $C_{k}'$ is a constant which is independent of $\tau$ (once
it is large enough), the approximate solution $(A,\varPhi)$ and the
constant $N_{0}\geq1$ used in the perturbations defining the connected
sum along $Y$ operation.

Since we are assuming that $\|V_{1}\|_{L_{k}^{2}},\|V_{2}\|_{L_{k}^{2}}\leq\alpha_{k}$,
we can use the triangle inequality to obtain that 
\[
\|V_{2}+V_{1}\|_{L_{k}^{2}(g_{\tau},A^{\#})}\leq\|V_{2}\|_{L_{k}^{2}(g_{\tau},A^{\#})}+\|V_{1}\|_{L_{k}^{2}(g_{\tau},A^{\#})}\leq2\alpha_{k}
\]
so the inequality (\ref{eq:inequality Q}) reads
\[
\left\Vert \mathcal{Q}\circ\left[\left(\mathcal{D}_{(A,\varPhi)^{\#}}\mathfrak{F}_{\mathfrak{q}}\right)^{*}\right]\left(\triangle_{2,\mathfrak{q},(A,\varPhi)^{\#}}^{-1}V_{2}\right)-\mathcal{Q}\circ\left[\left(\mathcal{D}_{(A,\varPhi)^{\#}}\mathfrak{F}_{\mathfrak{q}}\right)^{*}\right]\left(\triangle_{2,\mathfrak{q},(A,\varPhi)^{\#}}^{-1}V_{1}\right)\right\Vert _{L_{k}^{2}(g_{\tau}',A^{\#})}\leq2\alpha_{k}C_{k}^{\prime}\|V_{2}-V_{1}\|_{L_{k}^{2}(g_{\tau}',A^{\#})}
\]
 Hence to make this contribution less than $\frac{\kappa_{k}}{2}\|V_{2}-V_{1}\|_{L_{k}^{2}(g_{\tau}',A^{\#})}$
we just need to take $\alpha_{k}<\frac{\kappa}{4C_{k}'}$. 

Likewise, since 
\[
P\left(\triangle_{2,\mathfrak{q},(A,\varPhi)^{\#}}^{-1}V_{2}\right)-P\left(\triangle_{2,\mathfrak{q},(A,\varPhi)^{\#}}^{-1}V_{1}\right)
\]
 is the same as 
\[
f(V_{2})-f(V_{1})
\]
and each has quadratic leading terms according to equation (\ref{quadratic expansion})
, the norm
\[
\left\Vert P\left(\triangle_{2,\mathfrak{q},(A,\varPhi)^{\#}}^{-1}V_{2}\right)-P\left(\triangle_{2,\mathfrak{q},(A,\varPhi)^{\#}}^{-1}V_{1}\right)\right\Vert _{L_{k}^{2}(g_{\tau}',A^{\#})}
\]
can now be bounded by and expression of the form 
\[
f(\alpha_{k},C_{k}^{\prime\prime})\|V_{2}-V_{1}\|_{L_{k}^{2}(g_{\tau}',A^{\#})}
\]
 where $f(\alpha_{k},C_{k}^{\prime\prime})$ will be some expression
in $\alpha_{k}$ whose particular details do not interest us and $C_{k}^{\prime\prime}$
denotes constants that do not depend on $\tau$ or the solution used.
In any case, the important thing is that we can again choose $\alpha_{k}$
so that $f(\alpha_{k},C_{k}^{\prime\prime})<\frac{\kappa_{k}}{2}$
and so combining both inequalities the result follows. 
\end{proof}
At this point we can use the Contraction Mapping Theorem (proposition
2.3.5 \cite{MR2199446}) to obtain our definition of the \textbf{gluing
map} (Theorem 3.1.9 \cite{MR2199446}): 
\begin{thm}
\label{Gluing Theorem}There exists constants $\alpha_{k},c_{k}>0$
such that for every $\tau$ large enough, every solution $(A,\varPhi)$
of the Seiberg-Witten equations on $M_{\tau}$ whose gauge equivalence
class belongs to the zero dimensional strata of the moduli space $\mathcal{M}(M_{\tau};\mathfrak{s}_{\tau};[\mathfrak{c}])$
and every constant $N_{0}\geq1$, there is a unique section $(b',\psi')$
on $M_{\tau}'$ such that 
\[
\mathfrak{G}_{\tau}(A,\varPhi)=(A,\varPhi)^{\#}+(\mathcal{D}_{\mathfrak{q},(A,\varPhi)^{\#}}\mathfrak{F}_{\mathfrak{q}})^{*}(b',\psi')
\]
is a solution of the Seiberg-Witten equations with $\|(b',\psi)\|_{L_{k+2}^{2}(g_{\tau'},A^{\#})}\leq\alpha_{k}$.
Furthermore, the map is gauge equivariant and induces a map 
\[
\mathfrak{G}_{\tau}:\mathcal{M}_{0}(M_{\tau};\mathfrak{s}_{\tau};[\mathfrak{c}])\rightarrow\mathcal{M}_{0}(M_{\tau}';\mathfrak{s}_{\tau}';[\mathfrak{c}])
\]
where $\mathcal{M}_{0}(M_{\tau};\mathfrak{s}_{\tau};[\mathfrak{c}])$
denotes the zero dimensional strata of $\mathcal{M}(M_{\tau};\mathfrak{s}_{\tau};[\mathfrak{c}])$.
Moreover both $\|(b',\psi')\|_{L_{k+2}^{2}(g_{\tau}',A^{\#})}$ and
$\|\mathfrak{G}_{\mathfrak{\tau}}(A,\varPhi)-(A,\varPhi)^{\#}\|_{L_{k+1}^{2}(g_{\tau}',A^{\#})}$
are bounded by
\begin{equation}
\|(b',\psi')\|_{L_{k+2}^{2}(g_{\tau}',A^{\#})},\;\;\|\mathfrak{G}_{\mathfrak{\tau}}(A,\varPhi)-(A,\varPhi)^{\#}\|_{L_{k+1}^{2}(g_{\tau}',A^{\#})}\leq c_{k}\|\mathfrak{F}_{\mathfrak{p}_{M_{\tau}'}}(A,\varPhi)^{\#}\|_{L_{k}^{2}(g_{\tau}',A^{\#})}\label{inequalities gluing}
\end{equation}
Furthermore, this map is an injection and since the construction is
reversible it is a bijection. Hence the $\mod2$ cardinality of $\mathcal{M}_{0}(M_{\tau};\mathfrak{s}_{\tau};[\mathfrak{c}])$
and $\mathcal{M}_{0}(M_{\tau}';\mathfrak{s}_{\tau}';[\mathfrak{c}])$
is the same.
\end{thm}

\begin{proof}
We need to verify that the gluing map preserves the dimensionality
of the zero dimensional strata. For this recall that if $[(A,\varPhi)]$
belongs to $\mathcal{M}_{0}(M_{\tau};\mathfrak{s}_{\tau};[\mathfrak{c}])$,
then the index of the operator 
\[
Q_{\mathfrak{q},(A,\varPhi)}=\mathbf{d}_{(A,\varPhi)}^{*}\oplus\mathcal{D}_{(A,\varPhi)}\mathfrak{F}_{\mathfrak{p}_{M_{\tau}}}
\]
is precisely the dimension of the strata to which $[(A_{n},\varPhi_{n})]$
belongs. Since the transversality condition already implied that $Q_{\mathfrak{q},(A,\varPhi)}$
was surjective we conclude that in fact $Q_{\mathfrak{q},(A,\varPhi)}$
is an invertible operator. 

Now we use the same splicing procedure as in the case of finding the
inverse for the Seiberg Witten Laplacians $\triangle_{2,\mathfrak{q},(A,\varPhi)^{\#}}$.
Namely, the operator $\eta_{end}Q_{(A_{0,\tau}',\varPhi_{\tau}')}(\eta_{end}\cdot)$
associated to the canonical solution $(A_{0,\tau'},\varPhi_{\tau}')$
on the AFAK end $Z'$ will be invertible on a suitable domain using
Lemma 3.1.4 in \cite{MR2199446}. Therefore, we can patch together
$\eta_{cyl}Q_{\mathfrak{q},(A,\varPhi)}^{-1}(\eta_{cyl}\cdot)$ and
$\eta_{end}Q_{(A_{0,\tau'},\varPhi_{\tau}')}^{-1}(\eta_{end}\cdot)$
to show that $Q_{\mathfrak{q},(A,\varPhi)^{\#}}$ will become invertible. 

To compare $Q_{\mathfrak{q},(A,\varPhi)^{\#}}$ and $Q_{\mathfrak{q},\mathfrak{G}_{\mathfrak{q}}(A,\varPhi)}$
notice that inequality (\ref{exponential decay approximation}) and
the bound in (\ref{inequalities gluing}) allow us to conclude that
the operator norms of $Q_{\mathfrak{q},(A,\varPhi)^{\#}}$ and $Q_{\mathfrak{q},\mathfrak{G}_{\mathfrak{q,\tau}}(A,\varPhi)}$
are very close to each other. Since being an invertible operator is
an open condition it follows that $Q_{\mathfrak{q},\mathfrak{G}_{\mathfrak{q},\tau}(A,\varPhi)}$
will have to be invertible as well.

Now we must address the injectivity of our map. It is essentially
the same as the proof of Corollary 3.2.2 in \cite{MR2199446}. If
the injectivity of the map is not true for $\tau$ large enough then
we obtain a sequence $\tau_{j}\rightarrow\infty$ and solutions to
the Seiberg Witten equations $(A_{j},\varPhi_{j})$ and $(\tilde{A}_{j},\tilde{\varPhi}_{j})$
on $M_{\tau_{j}}$ such that for all $j$, $[A_{j},\varPhi_{j}]\neq[\tilde{A}_{j},\tilde{\varPhi}_{j}]$
while $[\mathfrak{G}_{\tau_{j}}(A_{j},\varPhi_{j})]=[\mathfrak{G}_{\tau_{j}}(\tilde{A}_{j},\tilde{\varPhi}_{j})]$.
Moreover, after taking gauge transformation we can assume that they
have exponential decay and converge on every compact subset of $Z_{Y,\xi}^{+}$
to some solutions $(A_{\infty},\varPhi_{\infty})$ and $(\tilde{A}_{\infty},\tilde{\varPhi}_{\infty})$
. Moreover, for all $j$ we have $[A_{j},\varPhi_{j}]\neq[\tilde{A}_{j},\tilde{\varPhi}_{j}]$
as gauge equivalence classes . We want to show that if $(A_{\infty},\varPhi_{\infty})=(\tilde{A}_{\infty},\tilde{\varPhi}_{\infty})$
then 
\begin{equation}
\|(A_{j},\varPhi_{j})-(\tilde{A}_{j},\tilde{\varPhi}_{j})\|_{L_{k+1}^{2}(g_{\tau},A_{j})}\rightarrow0\label{strong convergence solutions}
\end{equation}

First of all, from (\ref{inequalities gluing}) and (\ref{exponential decay approximation})
we already know that have that 
\begin{equation}
\|\mathfrak{G}_{\tau_{j}}(A_{j},\varPhi_{j})-(A_{j},\varPhi_{j})^{\#}\|_{L_{k+1}^{2}(g_{\tau}',A^{\#})}\rightarrow0\label{eq: approaching norms}
\end{equation}
hence $\mathfrak{G}_{\tau_{j}}(A_{j},\varPhi_{j})$ converges on every
compact towards $(A_{\infty},\varPhi_{\infty})$ since $(A_{j},\varPhi_{j})$
does. Similarly $\mathfrak{G}_{\tau_{j}}(\tilde{A}_{j},\tilde{\varPhi}_{j})$
converges to $(\tilde{A}_{\infty},\tilde{\varPhi}_{\infty})$. The
fact that $\mathfrak{G}(A_{j},\varPhi_{j})$ and $\mathfrak{G}(\tilde{A}_{j},\tilde{\varPhi}_{j})$
are gauge equivalent for each $j$ implies that the limits are also
gauge equivalent. Hence the limits of $(A_{j},\varPhi_{j})$ and $(\tilde{A}_{j},\tilde{\varPhi}_{j})$
are gauge equivalent. After making further gauge transformations,
we can then assume that $(A_{j},\varPhi_{j})$ and $(\tilde{A}_{j},\tilde{\varPhi}_{j})$
converge toward the same limit $(A_{\infty},\varPhi_{\infty})$ on
$Z_{Y,\xi}^{+}$. In principle, this would be weak convergence along
the cylindrical end $\mathbb{R}^{+}\times Y$. However, by the discussion
from before when we analyzed the restriction of a solution to the
cylindrical moduli space $\mathcal{M}(\mathbb{R}^{+}\times-Y,\mathfrak{s}_{\xi},[\mathfrak{c}])$
, we can actually assume that the convergence is strong along the
entire cylindrical end, in other words, $(A_{j},\varPhi_{j})$ and
$(\tilde{A}_{j},\tilde{\varPhi}_{j})$ are converging strongly towards
$(A_{\infty},\varPhi_{\infty})$ on the cylindrical end as well. This
allows us to conclude that (\ref{strong convergence solutions}) is
true. 

Since we now have strong convergence along the cylinder then the estimates
in \cite{MR2199446} continue to hold in that we can find a ``radius''
$r$ small enough {[}independent of $\tau${]} for which whenever
there is $j$ such that $\|(A_{j},\varPhi_{j})-(\tilde{A}_{j},\tilde{\varPhi}_{j})\|_{L_{k+1}^{2}(g_{\tau},A_{j})}<r$
then $(A_{j},\varPhi_{j})$ and $(\tilde{A}_{j},\tilde{\varPhi}_{j})$
are gauge equivalent {[}this is a much weaker version of their proposition
3.2.1{]}. From (\ref{strong convergence solutions}) it is clear that
such $j$ will exist and hence we are done. 
\end{proof}
We have reached the proof of the naturality property for the contact
invariant under strong symplectic cobordisms, that is, Theorem (\ref{thm: Naturality}).
To see why $\widecheck{HM}_{\bullet}(W^{\dagger},\mathfrak{s}_{\omega})\mathbf{c}(\xi')=\mathbf{c}(\xi)$
recall that in the first part of this paper (section 5 to be more
specific) we showed that 
\[
\widecheck{HM}_{\bullet}(W^{\dagger},\mathfrak{s}_{\omega})\mathbf{c}(\xi')=\mathbf{c}(\xi',Y)
\]
The gluing theorem we just proved was aimed at showing that 
\begin{equation}
\mathbf{c}(\xi',Y)=\mathbf{c}(\xi)\label{iden 1}
\end{equation}
To see why $\mathbf{c}(\xi',Y)=\mathbf{c}(\xi)$ we need to apply
Theorem (\ref{Gluing Theorem}) to the case in which the second AFAK
end is $Z'=(0,\infty)\times Y$. As explained before, it is not difficult
to see that for this choice the corresponding manifolds $M_{\tau}'$
in fact all agree with each other in the sense that their metrics,
spinor bundles, symplectic forms, etc are the same, and in fact coincide
with the manifold $Z_{Y,\xi}^{+}$ used to define the contact invariant
of $(Y,\xi)$. In particular, we have that for all $\tau>0$ that
\[
|\mathcal{M}_{0}(M_{\tau}',\mathfrak{s}',[\mathfrak{c}])|=|\mathcal{M}_{0}(Z_{Y,\xi}^{+},\mathfrak{s},[\mathfrak{c}])|\;\mod2
\]
 Now choose $\tau_{large}$ such that 
\[
|\mathcal{M}_{0}(M_{\tau};\mathfrak{s}_{\tau_{large}};[\mathfrak{c}])|=|\mathcal{M}_{0}(M_{\tau}',\mathfrak{s}',[\mathfrak{c}])|=|\mathcal{M}_{0}(Z_{Y,\xi}^{+},\mathfrak{s},[\mathfrak{c}])|\mod2
\]
If we think of using the numbers $|\mathcal{M}_{0}(M_{\tau};\mathfrak{s}_{\tau_{large}};[\mathfrak{c}])|\mod2$
in order to define a chain-level element $c(\xi',Y,\tau_{large})\in\check{C}_{*}(-Y,\mathfrak{s}_{\xi})$
as in formula (\ref{chain level def}) (the $\tau$-hybrid invariant
we discussed before), then the previous identity says that at the
chain level
\[
c(\xi',Y,\tau_{large})=c(\xi)
\]
which in particular gives the identity of homology classes 
\begin{equation}
\mathbf{c}(\xi',Y,\tau_{large})=\mathbf{c}(\xi)\label{iden 2}
\end{equation}
Now, $c(\xi',Y,\tau_{large})$ is not the same chain-level element
as the element $c(\xi',Y)$ we used during the initial sections of
this paper. However, it is not difficult to see that we can use a
one parameter family of metrics $g(t)$ and perturbations $\mathfrak{p}_{0}(t)$
on $M_{\tau_{large}}$ (which is diffeomorphic to $W_{\xi',Y}^{\dagger}$)
to go from one element to the other. Therefore, one can consider a
parameterized moduli space and use the same argument as in section
5 to conclude that $c(\xi',Y,\tau_{large})$ and $c(\xi',Y)$ do define
the same homology element in $\widecheck{HM}_{\bullet}(-Y,\mathfrak{s}_{\xi})$,
in other words 
\begin{equation}
\mathbf{c}(\xi',Y,\tau_{large})=\mathbf{c}(\xi',Y)\label{iden 3}
\end{equation}
Combining the identities (\ref{iden 1}), (\ref{iden 2}) and (\ref{iden 3})
the naturality result follows, i.e, we have shown that for a strong
symplectic cobordism $(W,\omega):(Y,\xi)\rightarrow(Y',\xi')$ one
has 
\[
\widecheck{HM}_{\bullet}(W^{\dagger},\mathfrak{s}_{\omega})\mathbf{c}(\xi')=\mathbf{c}(\xi',Y)
\]

\section{7. Appendix. Energy and Compactness}

Now we will briefly discuss the compactness arguments invoked during
this paper. We will explain why the results in \cite{Zhang[2016]}
allows us to extend the compactness results of \cite{MR2199446} to
the corresponding versions with cylindrical ends. First we will explain
why the ``dilating the cone'' operation we have discussed in this
paper can be regarded as a geometric way to implement Taubes' perturbation
for the Seiberg-Witten equations on symplectic manifolds \cite[eq. 1.18]{MR1362874}.

We recall the perturbations used by Taubes. Let $(X,\omega)$ be a
symplectic manifold. The analysis will be entirely local, so in fact
$X$ may be regarded as an open ball or a non-compact manifold if
preferred. The positive part $S_{\omega}^{+}$ of the spinor bundle
determined by $\omega$ can be decomposed as 
\[
S_{\omega}^{+}=\mathbb{C}\varPhi_{0}\oplus\left\langle \varPhi_{0}\right\rangle ^{\perp}
\]
Here $\varPhi_{0}$ is the canonical spinor that appeared before in
our paper: we can regard it as the constant function $\varPhi_{0}:X\rightarrow\{1\}\subset\mathbb{C}$
identically equal to $1$. The canonical connection $A_{0}$ is then
the unique spin-c connection for which $D_{A_{0}}\varPhi_{0}=0$.
The perturbed Seiberg-Witten equations on symplectic manifolds due
to Taubes are
\[
SW_{Taubes}^{r}:\begin{cases}
\frac{1}{2}\rho(F_{A}^{+})-(\varPhi\varPhi^{*})_{0}=\frac{1}{2}\rho(F_{A_{0}}^{+})-\frac{i}{4}r\rho(\omega)\\
D_{A}\varPhi=0
\end{cases}
\]
Here $r$ is a parameter that eventually Taubes takes to be very large.
For $r=1$, the perturbations are cooked up in such a way that $(A_{0},\varPhi_{0})$
solves $SW_{Taubes}^{1}$, since the Clifford identities say that
$(\varPhi_{0}\varPhi_{0}^{*})_{0}=\frac{i}{4}\rho(\omega)$ . Moreover,
$(A_{0},\sqrt{r}\varPhi_{0})$ then solve $SW_{Taubes}^{r}$, for
essentially the same reasons. At this point it is useful to note that
one can also rescale the spinor as Taubes usually does. That is, one
can write $\varPhi=\sqrt{r}\phi$ so that $SW_{Taubes}^{r}$ become
\begin{equation}
SW_{Taubes}^{r}:\begin{cases}
\frac{1}{2}\rho(F_{A}^{+})-r(\phi\phi^{*})_{0}=\frac{1}{2}\rho(F_{A_{0}}^{+})-\frac{i}{4}r\rho(\omega)\\
D_{A}\varPhi=0
\end{cases}\label{SWR}
\end{equation}
To see how to obtain $SW_{Taubes}^{r}$ from $SW_{Taubes}^{1}$
\[
SW_{Taubes}:\begin{cases}
\frac{1}{2}\rho(F_{A}^{+})-(\varPhi\varPhi^{*})_{0}=\frac{1}{2}\rho(F_{A_{0}}^{+})-\frac{i}{4}\rho(\omega)\\
D_{A}\varPhi=0
\end{cases}
\]
suppose that we dilate the metric. That is, for a parameter $\tau>0$,
we define 
\[
g_{\tau}=\tau^{2}g
\]
where $g$ was a metric compatible with $\omega$, i.e, $\omega$
is a self-dual harmonic 2-form of point-wise norm $\sqrt{2}$ with
respect to $g$. Since $\tau$ is a constant then 
\[
\omega_{\tau}=\tau^{2}\omega
\]
 will continue to be a symplectic form, now compatible with the metric
$g_{\tau}$. The spinor bundle will clearly not change with respect
to this new metric, that is, 
\[
S_{\omega,\tau}=S_{\omega}
\]
while the Clifford map on one-forms is rescaled as 
\[
\rho_{\tau}=\frac{\rho}{\tau}
\]
On two forms we find that $\rho_{\tau}=\frac{\rho}{\tau^{2}}$, so
in particular $\rho_{\tau}(\omega_{\tau})=\rho(\omega)$. A dilation
is a very simple conformal change of metric, for which it is understood
how the Dirac operator changes \cite[eq. D.1]{MR2717225}. In our
case we find that 
\[
D_{A,g_{\tau}}\varPhi=\tau^{-1}D_{A,g}\varPhi
\]
 so in particular being a harmonic spinor (i.e, $D_{A}\varPhi=0$)
is a condition independent of the metric $g_{\tau}$ used, and the
canonical connection will be preserved under the dilation, since $A_{0}$
satisfies the property $D_{A_{0},g_{\tau}}\varPhi_{0}=0$ regardless
of the value of $\tau$. In other words
\[
A_{0,\tau}=A_{0}
\]
Moreover, the notion of self-duality is preserved under conformal
changes of the metric, in particular, dilations, which means that
the \textbf{Taubsian geometric perturbations }of the Seiberg-Witten
equations on the dilated symplectic manifolds $(X,g_{\tau},\omega_{\tau})$
\[
SW_{Taubes}^{\tau}:\begin{cases}
\frac{1}{2}\rho_{\tau}(F_{A}^{+})-(\varPhi\varPhi^{*})_{0}=\frac{1}{2}\rho_{\tau}(F_{A_{0},\tau}^{+})-\frac{i}{4}\rho_{\tau}(\omega_{\tau})\\
D_{A,g_{\tau}}\varPhi=0
\end{cases}
\]
can be rewritten in terms of the geometric structures on the original
data $(X,g,\omega)$ as {[}recall that on two-forms $\rho_{\tau}(F_{A}^{+})=\frac{\rho(F_{A}^{+})}{\tau^{2}}${]}
\begin{equation}
SW_{Taubes}^{\tau}:\begin{cases}
\frac{1}{2}\rho(F_{A}^{+})-\tau^{2}(\varPhi\varPhi^{*})_{0}=\frac{1}{2}\rho(F_{A_{0},\tau}^{+})-\frac{i}{4}\tau^{2}\rho(\omega)\\
D_{A,g}\varPhi=0
\end{cases}\label{SWR-1}
\end{equation}
Setting $\tau=\sqrt{r}$, we see that $SW_{Taubes}^{\tau}$ is indistinguishable
from $SW_{Taubes}^{r}$! Therefore, the results from papers which
work with the perturbations $SW_{Taubes}^{r}$ can be translated immediately
to our paper. In particular, the theorems from \cite{Zhang[2016]}
are readily available to our situation. The only caveat is that \cite{Zhang[2016]}
interprets the scaling of the spinors in a slightly different way.
For example, definition 3.5 of the configuration space in \cite{Zhang[2016]}
writes the decay condition as $(\varPhi-\sqrt{r}\varPhi_{0})\in L_{k,A_{0}}^{2}$,
while our definition (\ref{configuration space}) of the configuration
space write the decay condition as $(\varPhi-\varPhi_{0})\in L_{k,A_{0}}^{2}$.
In particular, when the results of \cite{Zhang[2016]} are translated
into our context, there are additional factors of $\sqrt{r}$ one
needs to remove.

The most important result proven from \cite{Zhang[2016]} for the
purposes of our problem is the uniform bound on the symplectic energy
for solutions of the Seiberg-Witten equations on manifolds with both
symplectic and cylindrical ends. More precisely, when $(A,\varPhi)$
is a configuration on $M_{\tau}$ we can restrict it to the symplectic
region $M_{\tau}\backslash(\mathbb{R}^{+}\times-Y)$ and consider
the \textbf{symplectic energy }
\begin{equation}
E_{\tau}(A,\varPhi)=\int_{M_{\tau}\backslash(\mathbb{R}^{+}\times-Y)}\left[\left(1-|\alpha|^{2}-|\beta|^{2}\right)^{2}+|F_{a}|^{2}+|\widehat{\nabla}_{A}\varPhi|^{2}\right]\text{vol}^{g_{\tau}}\label{eq:symplectic energy}
\end{equation}
which is defined in section 5.3 of \cite{Zhang[2016]}. Here we decomposed
$\varPhi$ as $(\alpha,\beta)$ and $A$ as $A=A_{0}+a$. In the formula
for \ref{eq:symplectic energy}, $\widehat{\nabla}_{A}$ refers to
the \textbf{twisted Chern connection, }which can be defined as 
\[
\widehat{\nabla}_{A}\varPhi=\nabla^{C}\varPhi+a\otimes\varPhi
\]
where $\nabla^{C}$ is the Chern connection the $AFAK$ end structure
determines (this is explained in section 1.3.2 of \cite{MR2199446}).
Notice that \cite{Zhang[2016]} writes $|\widehat{\nabla}_{A}\varPhi|^{2}$
as $|\nabla_{a}\alpha|^{2}+|\nabla_{A}'\beta|^{2}$, but these two
things in fact mean the same. The uniform bound on the symplectic
energy $E_{\tau}$ that we need is the following.
\begin{lem}
\label{uniform bound symplectic energy} Consider the sequence of
manifolds $M_{\tau}$ defined in section 6.2, equation (\ref{manifolds Mtau}).
There exists a constant $\kappa$ such that for every $\tau$ large
enough and every solution to the Seiberg-Witten equations $(A,\varPhi)$
on $M_{\tau}$ asymptotic along the symplectic end to $(A_{0},\varPhi_{0})$
and along the cylindrical end to a critical point $\mathfrak{c}$
of the three dimensional Seiberg-Witten equations, we have that 
\[
E_{\tau}(A,\varPhi)\leq\kappa
\]
\end{lem}

\begin{proof}
Notice that if the cylindrical end is not present, then this result
would follow from \cite[Lemma 2.2.7]{MR2199446}. In our case the
perturbation term $\varpi_{\tau}$ which appears in the previously
cited lemma of Mrowka and Rollin is not present, since this was only
introduced for transversality purposes, which is not needed for the
manifolds $M_{\tau}$, thanks to the use of the abstract perturbations
along the cylindrical end and the fiber product description of the
moduli spaces. 

As mentioned by Mrowka and Rollin, when there cylindrical end is not
present, the fact that $\kappa$ is independent of $\tau$ is a consequence
of the uniform boundedness of the injectivity radius and curvature,
as well as the other geometric quantities involved in the definition
of these families of $AFAK$ manifolds $M_{\tau}$, as described in
\cite[Lemma 2.1.6]{MR2199446}. In other words, when the cylindrical
end is not present, finding a uniform control on the symplectic energy
$E_{\tau}$ is a consequence of knowing that $E_{\tau_{0}}$ is bounded
for some $\tau_{0}$ sufficiently large (which was the content of
\cite[Lemma 3.17]{MR1474156}) and then appealing to the uniform geometry
of the manifolds $M_{\tau}$. 

Therefore, for our situation what we need to know is that in the presence
of a cylindrical end it is still the case that one can find $\tau_{0}$
sufficiently large so that $E_{\tau_{0}}(A,\varPhi)\leq\kappa_{0}$
for all solutions $(A,\varPhi)$ on $M_{\tau_{0}}$ and for some $\kappa_{0}$
independent of the Seiberg-Witten solution $(A,\varPhi)$ on $M_{\tau_{0}}$.
In other words, we want the analogue of the above mentioned \cite[Lemma 3.17]{MR1474156}
in the presence of a cylindrical end. The fact that $\kappa_{0}$
can be taken to be independent of $\tau_{0}$ will follow again from
the same comments we just made in the previous paragraph regarding
the uniform boundedness of the geometric quantities involved in the
family of the $AFAK$ manifolds. The existence of such a $\tau_{0}$
now follows from \cite[Proposition 5.12]{Zhang[2016]}, given that
we can reinterpret the equations on an arbitrary $M_{\tau}$ as equations
on the fixed $M_{1}$ , at the cost of perturbing the curvature equation
with factors of $\tau$, as we explained at the beginning of the appendix
when we discussed the Taubsian perturbations \ref{SWR} and \ref{SWR-1}.
\end{proof}
Now we proceed to explain some of the auxiliary lemmas used in some
of the proofs of our paper. The first one appeared in Section 6.3,
where the pregluing map was being constructed.
\begin{lem}
\label{C0 control}Consider the sequence of manifolds $M_{\tau}$
defined in section 6.2, equation (\ref{manifolds Mtau}). We can find
a compact set $C=[1,T]\times Y$ contained in the symplectic region
$M_{\tau}\backslash(\mathbb{R}^{+}\times-Y)$ with the following significance:
for every $\tau$ large enough and for every solution to the Seiberg-Witten
equations $(A,\varPhi)$ on $M_{\tau}$ asymptotic along the symplectic
end to $(A_{0},\varPhi_{0})$ and along the cylindrical end to a critical
point $\mathfrak{c}$ of the three dimensional Seiberg-Witten equations,
we have $|\alpha|\geq\frac{1}{2}$ on $M_{\tau}\backslash[(\mathbb{R}^{+}\times-Y)\cup C]$,
where we wrote $\varPhi$ as $(\alpha,\beta)$ along the symplectic
end.
\end{lem}

\begin{rem}
Notice that $\tau$ will in fact not depend on the particular critical
point $\mathfrak{c}$ we are asymptotic to.
\end{rem}

\begin{proof}
Observe that this was lemma 2.2.8 in \cite{MR2199446}. The proof
Mrowka and Rollin provided proceeded by contradiction, and eventually
gives rise to a sequence of solutions to the Seiberg-Witten equations
$(A_{j},\varPhi_{j})$ on $M_{\tau_{j}}$. However, they just care
about the restrictions of these solutions to some balls $B(x_{j},\sigma_{\tau_{j}}(x_{j})/\kappa)$
centered at $x_{j}$ of radius $\sigma_{\tau_{j}}(x_{j})/\kappa$.
The important feature of these balls is that they are all contained
in the symplectic region $M_{\tau_{j}}\backslash(\mathbb{R}^{+}\times-Y)$,
so what is happening along the cylinder is of no importance, given
that we already know the uniform control on the symplectic energies
$E_{\tau}$ thanks to the previous lemma. Therefore, the proof they
give in fact goes through in our setup, since on the symplectic region
we are using the same perturbations as Mrowka and Rollin.
\end{proof}
\begin{rem}
The subsequence Mrowka and Rollin obtained in the proof of their lemma
2.2.8 in \cite{MR2199446} can also be obtained from Proposition 5.3
in \cite{Zhang[2016b]}, if in the notation of \cite{Zhang[2016b]}
we take the sequence $\{(M_{n},g_{n},p_{n})\}$ to be the sequence
of pointed balls $B(x_{j},\sigma_{\tau_{j}}(x_{j})/\kappa)$ that
arise from the proof by contradiction.
\end{rem}

The next lemma appeared right after the pre-gluing map (\ref{PREgluing map})
was defined.
\begin{lem}
\label{exponential decay SW map}Lemma 2.5.4 in \cite{MR2199446}
still holds. That is, there is a $\delta>0$ and $T$ large enough
such that for every $N_{0}\geq1,k\in\mathbb{N}$ , $\tau$ satisfying
$\tau\geq T+N_{0}$ and every solution $(A,\varPhi)$ of the Seiberg-Witten
equations on $M_{\tau}$, we have that $(A,\varPhi)^{\#}$ satisfies
the Seiberg-Witten equations on $\{\sigma_{\tau}'\leq T\}\subset M_{\tau}'$
and 
\begin{equation}
|\mathfrak{F}_{\mathfrak{p}_{M_{\tau}'}}(A,\varPhi)^{\#}|_{C^{k}(g_{\tau}',A^{\#})}\leq c_{k}e^{-\delta\sigma_{\tau}}\label{exponential decay approximation-1}
\end{equation}
 on $\{\sigma_{\tau}'\geq T\}\subset M_{\tau}'$.
\end{lem}

\begin{rem}
Again, we are assuming that the solutions are asymptotic along the
cylindrical end to some critical point $\mathfrak{c}$, but the constants
of the lemma are insensitive to the particular critical point being
used.
\end{rem}

\begin{proof}
Our proof in our lemma is essentially the same as the proof of lemma
2.5.4 in \cite{MR2199446}. The only thing we need to know is that
we can find a gauge with uniform exponential decay, which was Corollary
2.2.10 in \cite{MR2199446}. In return, the proof of this Corollary
was modeled on the proof of Corollary 3.16 in \cite{MR1474156}. That
last corollary required knowing that the symplectic energy $E_{\tau}$
is uniformly bounded. But this is precisely the content of our Lemma
\ref{uniform bound symplectic energy}.

\end{proof}

The following lemma was used in the proof of lemma (\ref{uniform parametrices}).
\begin{lem}
\label{Mrowka Rohlin compactnessx} Let $(A_{n},\varPhi_{n})$ be
a sequence of solutions to the Seiberg-Witten equations on $M_{\tau_{n}}$.
Then after making gauge transformations if necessary, we can assume
that $(A_{n},\varPhi_{n})$ converges weakly to a solution $(A_{\infty},\varPhi_{\infty})$
on $Z_{Y,\xi}^{+}$. In particular, on every compact set $K$ of $Z_{Y,\xi}^{+}$,
which can be regarded as a subset of $M_{\tau_{n}}$ for all $\tau_{n}$
sufficiently large, we have strong convergence of $(A_{n},\varPhi_{n})$
to $(A_{\infty},\varPhi_{\infty})$.

Therefore, if we take a sequence of solutions $(A_{n},\varPhi_{n})$
on $M_{\tau_{n}}$ and restrict them to $[1,T]\times Y\subset M_{\tau_{n}}$,
the solutions will converge (after gauge) strongly to a solution $(A_{\infty},\varPhi_{\infty})$
on $[1,T]\times Y\subset Z_{Y,\xi}^{+}$.
\end{lem}

\begin{proof}
Notice that when we restrict the solutions $(A_{n},\varPhi_{n})$
to the cylindrical ends $\mathbb{R}^{+}\times-Y$, which is the same
for all $M_{\tau_{n}}$, we already know that they will converge in
the weak sense to some solution on $\mathbb{R}^{+}\times-Y$ , and
hence strongly on compact subsets of $\mathbb{R}^{+}\times-Y$ (after
gauging if necessary). This compactness result is based on Proposition
24.6.4 of \cite{MR2388043}, which only needs a bound on the topological
energy $\mathcal{E}_{\mathfrak{q}}^{top}$ \cite[Definition 24.6.3]{MR2388043}.
That the topological energy continues to be bounded in our situation
is explained in remark $4$ after the proof of Theorem $6.1$ in \cite{Zhang[2016]}.

Therefore the only new thing to understand is why the convergence
still holds when we restrict the solutions to the symplectic ends,
since after that one can do a patching argument, like in the proof
of Proposition 24.6.4 in \cite{MR2388043}. The convergence on the
symplectic end follows the same argument as in the proof of the compactness
theorem 2.2.11 in \cite{MR2199446}, which is a consequence of being
able to find gauge transformations along the symplectic end which
give rise to a uniform exponential decay for $(A_{n},\varPhi_{n})$
with respect to $(A_{0},\varPhi_{0})$. As we mentioned previously,
in the proof of lemma (\ref{exponential decay SW map}), this is a
consequence of the uniform control on the symplectic energy $E_{\tau_{n}}$,
i.e, Lemma \ref{uniform bound symplectic energy}.
\end{proof}
\printindex{}

\textcolor{black}{\bibliographystyle{amsplain}
\bibliography{references}
}
\begin{quote}
Department of Mathematics, Rutgers University.

\textit{\small{}E-mail addresses}{\small{}:} \textsf{\footnotesize{}me3qr@virginia.edu, mariano.echeverria@rutgers.edu}{\footnotesize\par}
\end{quote}

\end{document}